\documentclass[a4paper,oneside,10pt]{article}%
\usepackage{amsmath}
\usepackage{amsfonts}
\usepackage{amssymb}
\usepackage{graphicx}
\usepackage{color}
\usepackage[square,numbers,sort&compress]{natbib}%
\providecommand{\U}[1]{\protect\rule{.1in}{.1in}}

\newtheorem{theorem}{Theorem}[section]

\newtheorem{corollary}[theorem]{Corollary}

\newtheorem{assumption}[theorem]{Assumption}
\newtheorem{example}[theorem]{Example}

\newtheorem{lemma}[theorem]{Lemma}

\newtheorem{remark}[theorem]{Remark}

\newenvironment{proof}[1][Proof]{\noindent\textbf{#1.} }{\ \rule{0.5em}{0.5em}}
\numberwithin{equation}{section}

\begin{document}

\title{A global stochastic maximum principle for fully coupled forward-backward
stochastic systems }
\author{Mingshang Hu\thanks{Zhongtai Securities Institute for Financial Studies,
Shandong University, Jinan, Shandong 250100, PR China. humingshang@sdu.edu.cn.
Research supported by NSF (No. 11671231) and Young Scholars Program of
Shandong University (No. 2016WLJH10). }
\and Shaolin Ji\thanks{Zhongtai Securities Institute for Financial Studies,
Shandong University, Jinan, Shandong 250100, PR China. jsl@sdu.edu.cn
(Corresponding author). Research supported by NSF No. 11571203.}
\and Xiaole Xue\thanks{Zhongtai Securities Institute for Financial Studies,
Shandong University, Jinan 250100, China. Email: xiaolexue1989@gmail.com,
xuexiaole.good@163.com. Research supported by NSF (No. 11801315) and Natural
Science Foundation of Shandong Province(ZR2018QA001).} }
\maketitle

\textbf{Abstract}. We study a stochastic optimal control problem for fully
coupled forward-backward stochastic control systems with a nonempty control
domain. For our problem, the first-order and second-order variational
equations are fully coupled linear FBSDEs. Inspired by Hu \cite{Hu17}, we
develop a new decoupling approach by introducing an adjoint equation which is
a quadratic BSDE. By revealing the relations among the terms of the
first-order Taylor's expansions, we estimate the orders of them and derive a
global stochastic maximum principle which includes a completely new term.
Applications to stochastic linear quadratic control problems are investigated.



\addcontentsline{toc}{section}{\hspace*{1.8em}Abstract}

\section{Introduction}

It is well known that deriving maximum principles, namely, necessary
conditions for optimality, is an important approach in solving optimal control
problems (see \cite{YongZhou} and the references therein).
Boltyanski-Gamkrelidze-Pontryagin \cite{B-G-P56} announced the Pontryagin's
maximum principle for the first time for deterministic control systems in
1956. They introduced the spike variation and studied the first-order term in
a sort of Taylor's expansion with respect to this perturbation. But for
stochastic control systems, if the diffusion terms depend on the controls,
then one can't follow this idea for deterministic control systems. The reason
is that the It\^{o} integral $%
{\displaystyle\int\nolimits_{t}^{t+\varepsilon}}
\sigma(s)dB(s)$ is only of order $\sqrt{\varepsilon}$ which leads to the
first-order expansion method failed. To overcome this difficulty, Peng
\cite{Peng90} first introduced the second-order term in the Taylor expansion
of the variation and obtained the global maximum principle for the classical
stochastic optimal control problem. Since then, many researchers investigate
this kind of optimal control problems for various stochastic systems (see
\cite{YingHu006,YingHu001,Tang003,Tang004,Zhou002}).

Peng \cite{Peng93} generalized the classical stochastic optimal control
problem to one where the cost functional is defined by $Y(0)$. Here
$(Y(\cdot),Z(\cdot))$ is the solution of the following backward stochastic
differential equation (BSDE) (\ref{intro--bsde0}):%
\begin{equation}
\left\{
\begin{array}
[c]{rl}%
-dY(t)= & f(t,X(t),Y(t),Z(t),u(t))dt-Z(t)dB(t),\\
Y(T)= & \phi(X(T)).
\end{array}
\right.  \label{intro--bsde0}%
\end{equation}
Since El Karoui et al. \cite{ElKaetal97} defined a more general class of
stochastic recursive utilities in economic theory by solutions of BSDEs, this
new kind of stochastic optimal control problem is called the stochastic
recursive optimal control problem. When the control domain is convex, one can
avoid spike variation method and deduce a so-called local stochastic maximum
principle. Peng \cite{Peng93} first established a local stochastic maximum
principle for the classical stochastic recursive optimal control problem. The
local stochastic maximum principles for other various problems were studied in
(Dokuchaev and Zhou \cite{Dokuchaev-Zhou}, Ji and Zhou \cite{Ji-Zhou}, Peng
\cite{Peng93}, Shi and Wu \cite{Shi-Wu}, Xu \cite{Xu95}, Zhou \cite{Zhou003},
see also the references therein). But when the control domain is nonconvex,
one encounters an essential difficulty when trying to derive the first-order
and second-order expansions for the BSDE (\ref{intro--bsde0}) and it is
proposed as an open problem in Peng \cite{Peng99}. Recently, Hu \cite{Hu17}
studied this open problem and obtained a completely novel global maximum
principle. In \cite{Hu17}, Hu found that there are closely relations among the
terms of the first-order Taylor's expansions, i.e.,%
\begin{equation}%
\begin{array}
[c]{l}%
Y_{1}\left(  t\right)  =p\left(  t\right)  X_{1}\left(  t\right)  ,\\
Z_{1}(t)=p(t)\delta\sigma(t)I_{E_{\epsilon}}(t)+[\sigma_{x}(t)p(t)+q(t)]X_{1}%
\left(  t\right)  ,
\end{array}
\label{intro-relation-hu}%
\end{equation}
where $(p\left(  \cdot\right)  ,q(\cdot))$ is the solution of the adjoint
equation. And the BSDE satisfied by $(p\left(  \cdot\right)  ,q(\cdot))$
possesses a linear generator. Notice that the variation of $Z(t)$ includes the
term $\langle p(t),\delta\sigma(t)\rangle I_{E_{\varepsilon}}(t)$. Hu
\cite{Hu17} proposed to do Taylor's expansions at $\bar{Z}(t)+p(t)\delta
\sigma(t)I_{E_{\epsilon}}(t)$ and deduced the maximum principle.

Motivated by the leader-follower stochastic differential games and other
problems in mathematical finance, Yong \cite{Yong10} studied a fully coupled
controlled FBSDE with mixed initial-terminal conditions. In \cite{Yong10},
Yong regarded $Z(\cdot)$ as a control process and then applied the Ekeland
variational principle to obtain an optimality variational principle which
contains unknown parameters. Note that using the similar approach, Wu
\cite{Wu13} studied a stochastic recursive optimal control problem. In this
paper, we study the following stochastic optimal control problem: minimize the
cost functional
\[
J(u(\cdot))=Y(0)
\]
subject to the following fully coupled forward-backward stochastic
differential equation (FBSDE) (see \cite{YingHu002, Ma-WZZ, Ma-Yong-FBSDE,
Ma-ZZ, Zhang17} and the references therein):
\begin{equation}
\left\{
\begin{array}
[c]{rl}%
dX(t)= & b(t,X(t),Y(t),Z(t),u(t))dt+\sigma(t,X(t),Y(t),Z(t),u(t))dB(t),\\
dY(t)= & -g(t,X(t),Y(t),Z(t),u(t))dt+Z(t)dB(t),\\
X(0)= & x_{0},\ Y(T)=\phi(X(T)),
\end{array}
\right.  \label{intro--fbsde}%
\end{equation}
where the control variable $u(\cdot)$ takes values in a nonempty subset of
$\mathbb{R}^{k}$. In fact, our model is a special one in Yong \cite{Yong10}.
But our object is to get rid of the unknown parameters in the optimality
variational principle in \cite{Yong10,Wu13} and obtain a global stochastic
maximum principle for the above fully coupled control system. In order to do
this, we should study the variational equations of the BSDE in
(\ref{intro--fbsde}). But as pointed out in \cite{Yong10}, the
regularity/integrability of process $Z(\cdot)$ seems to be not enough in the
case when a second order expansion is necessary. Fortunately, inspired by Hu
\cite{Hu17}, we overcome this difficulty based on the following two findings.
The first one is although the first-order and second-order variational
equations are fully coupled linear FBSDEs, we can decouple them by
establishing the relations among the first-order Taylor's expansions, i.e.,
\begin{equation}%
\begin{array}
[c]{l}%
Y_{1}\left(  t\right)  =p\left(  t\right)  X_{1}\left(  t\right)  ,\\
Z_{1}(t)=\Delta(t)I_{E_{\epsilon}}(t)+K_{1}(t)X_{1}\left(  t\right)  ,
\end{array}
\label{intro--relation}%
\end{equation}
where $\Delta(t)$ satisfies the following algebra equation
\begin{equation}
\Delta(t)=p(t)(\sigma(t,\bar{X}(t),\bar{Y}(t),\bar{Z}(t)+\Delta
(t),u(t))-\sigma(t,\bar{X}(t),\bar{Y}(t),\bar{Z}(t),\bar{u}(t)))
\label{intro--algebra}%
\end{equation}
and $(p\left(  \cdot\right)  ,q(\cdot))$ is the adjoint process which
satisfies a quadratic BSDE. By the results of Lepeltier and San Martin
\cite{LS02}, we obtain the existence of solution to this nonlinear adjoint
equation. Utilizing the uniqueness result of the linear fully coupled FBSDE in
the appendix, we also prove the uniqueness of solution to this adjoint
equation. The second finding is that the first-order variation $Z_{1}(t)$ has
a unique decomposition by the relations (\ref{intro--relation}). This point
inspires us that we should do Taylor's expansions at $\bar{Z}(t)+\Delta
(t)I_{E_{\epsilon}}(t)$. The advantage of this approach is that the reminder
term of Taylor's expansions $K_{1}(t)X_{1}\left(  t\right)  $ has good
estimate which avoids the difficulty to do estimates such as $E[%
{\displaystyle\int\nolimits_{0}^{T}}
\mid Z(t)\mid^{2+\varepsilon}dt]<\infty$, for some $\varepsilon>0$. For this
reason, the obtained maximum principle will include a new term $\Delta(t)$
which is determined uniquely by $u(t)$, $\bar{u}(t)$, and the optimal state
$(\bar{X}(t)$, $\bar{Y}(t)$, $\bar{Z}(t))$. The readers may refer to
subsection \ref{heuristic} for a heuristic derivation.

By assuming $q(\cdot)$ is a bounded process, we derive the first-order and
second-order variational equations and deduce a global maximum principle which
includes a new term $\Delta(t)$. Furthermore, we study the case in which
$q(\cdot)$ may be unbounded. But for this case, we only obtain the maximum
principle when $\sigma(t,x,y,z,u)$ is linear in $z$, i.e.,
\[
\sigma(t,x,y,z,u)=A(t)z+\sigma_{1}(t,x,y,u).
\]
Finally, applications to stochastic linear quadratic control problems are investigated.

The rest of the paper is organized as follows. In section 2, we give the
preliminaries and formulation of our problem. A global stochastic maximum
principle is obtained by spike variation method in section 3. Especially, to
illustrate our main approach, we give a heuristic derivation in subsection
\ref{heuristic} before we prove the maximum principle strictly. In section 4,
a linear quadratic control problem is investigated based on the obtained
estimates in section 3. In appendix, we give some results that will be used in
our proofs.

\section{ Preliminaries and problem formulation}

Let $(\Omega,\mathcal{F},P)$ be a complete probability space on which a
standard $d$-dimensional Brownian motion $B=(B_{1}(t),B_{2}(t),...B_{d}%
(t))_{0\leq t\leq T}^{\intercal}$ is defined. Assume that $\mathbb{F=}%
\{\mathcal{F}_{t},0\leq t\leq T\}$ is the $P$-augmentation of the natural
filtration of $B$, where $\mathcal{F}_{0}$ contains all $P$-null sets of
$\mathcal{F}$. Denote by $\mathbb{R}^{n}$ the $n$-dimensional real Euclidean
space and $\mathbb{R}^{k\times n}$ the set of $k\times n$ real matrices. Let
$\langle\cdot,\cdot\rangle$ (resp. $\left\vert \cdot\right\vert $) denote the
usual scalar product (resp. usual norm) of $\mathbb{R}^{n}$ and $\mathbb{R}%
^{k\times n}$. The scalar product (resp. norm) of $M=(m_{ij})$, $N=(n_{ij}%
)\in\mathbb{R}^{k\times n}$ is denoted by $\langle M,N\rangle
=tr\{MN^{\intercal}\}$ (resp.$\Vert M\Vert=\sqrt{MM^{\intercal}}$), where the
superscript $^{\intercal}$ denotes the transpose of vectors or matrices.

We introduce the following spaces.

$L_{\mathcal{F}_{T}}^{p}(\Omega;\mathbb{R}^{n})$ : the space of $\mathcal{F}%
_{T}$-measurable $\mathbb{R}^{n}$-valued random variables $\eta$ such that
\[
||\eta||_{p}:=(\mathbb{E}[|\eta|^{p}])^{\frac{1}{p}}<\infty,
\]

$L_{\mathcal{F}_{T}}^{\infty}(\Omega;\mathbb{R}^{n})$: the space of
$\mathcal{F}_{T}$-measurable $\mathbb{R}^{n}$-valued random variables $\eta$
such that $||\eta||_{\infty}:=\underset{\omega\in\Omega}{\mathrm{ess~sup}%
}\left\Vert \eta\right\Vert <\infty$,

$L_{\mathcal{F}}^{p}([0,T];\mathbb{R}^{n})$: the space of $\mathbb{F}$-adapted
and $p$-th integrable stochastic processes on $[0,T]$ such that
\[
\mathbb{E}\left[  \int_{0}^{T}\left\vert f(t)\right\vert ^{p}dt\right]
<\infty,
\]

$L_{\mathcal{F}}^{\infty}(0,T;\mathbb{R}^{n})$: the space of $\mathbb{F}%
$-adapted and uniformly bounded stochastic processes on $[0,T]$ such that
\[
||f(\cdot)||_{\infty}=\underset{(t,\omega)\in\lbrack0,T]\times\Omega
}{\mathrm{ess~sup}}|f(t)|<\infty,
\]

$L_{\mathcal{F}}^{p,q}([0,T];\mathbb{R}^{n})$: the space of $\mathbb{F}%
$-adapted stochastic processes on $[0,T]$ such that
\[
||f(\cdot)||_{p,q}=\left\{  \mathbb{E}\left[  \left(  \int_{0}^{T}%
|f(t)|^{p}dt\right)  ^{\frac{q}{p}}\right]  \right\}  ^{\frac{1}{q}}<\infty,
\]

$L_{\mathcal{F}}^{p}(\Omega;C([0,T],\mathbb{R}^{n}))$: the space of
$\mathbb{F}$-adapted continuous stochastic processes on $[0,T]$ such that
\[
\mathbb{E}\left[  \sup\limits_{0\leq t\leq T}\left\vert f(t)\right\vert
^{p}\right]  <\infty.
\]

\subsection{$L^{p}$ estimate for fully coupled FBSDEs}

We first give an $L^{p}$-estimate for the following fully coupled
forward-backward stochastic differential equation:
\begin{equation}
\left\{
\begin{array}
[c]{rl}%
dX(t)= & b(t,X(t),Y(t),Z(t))dt+\sigma(t,X(t),Y(t),Z(t))dB(t),\\
dY(t)= & -g(t,X(t),Y(t),Z(t))dt+Z(t)dB(t),\\
X(0)= & x_{0},\ Y(T)=\phi(X(T)),
\end{array}
\right.  \label{fbsde}%
\end{equation}
where
\[
b:\Omega\times\lbrack0,T]\times\mathbb{R}^{n}\times\mathbb{R}^{m}%
\times\mathbb{R}^{m\times d}\rightarrow\mathbb{R}^{n},
\]%
\[
\sigma:\Omega\times\lbrack0,T]\times\mathbb{R}^{n}\times\mathbb{R}^{m}%
\times\mathbb{R}^{m\times d}\rightarrow\mathbb{R}^{n\times d},
\]%
\[
g:\Omega\times\lbrack0,T]\times\mathbb{R}^{n}\times\mathbb{R}^{m}%
\times\mathbb{R}^{m\times d}\rightarrow\mathbb{R}^{m},
\]%
\[
\phi:\Omega\times\mathbb{R}^{n}\rightarrow\mathbb{R}^{m}.
\]
A solution to (\ref{fbsde}) is a triplet of $\mathbb{F}$-adapted process
$\Theta(\cdot):=(X(\cdot),Y(\cdot),Z(\cdot))$. We impose the following assumption.

\begin{assumption}
\label{assum-1}(i) $\psi=b,\sigma,g,\phi$ are uniformly Lipschitz continuous
with respect to $x,y,z$, that is, there exist constants $L_{i}>0$, $i=1,2,3$
such that%
\[%
\begin{array}
[c]{rl}%
|b(t,x_{1},y_{1},z_{1})-b(t,x_{2},y_{2},z_{2})| & \leq L_{1}|x_{1}%
-x_{2}|+L_{2}(|y_{1}-y_{2}|+|z_{1}-z_{2}|),\\
|\sigma(t,x_{1},y_{1},z_{1})-\sigma(t,x_{2},y_{2},z_{2})| & \leq L_{1}%
|x_{1}-x_{2}|+L_{2}|y_{1}-y_{2}|+L_{3}|z_{1}-z_{2}|,\\
|g(t,x_{1},y_{1},z_{1})-g(t,x_{2},y_{2},z_{2}) & \leq L_{1}(|x_{1}%
-x_{2}|+|y_{1}-y_{2}|+|z_{1}-z_{2}|),\\
|\phi(t,x_{1})-\phi(t,x_{2})| & \leq L_{1}|x_{1}-x_{2}|,
\end{array}
\]
for all $t,\omega,x_{i},y_{i},z_{i}$, $i=1,2$. \newline(ii) For a given $p>1$,
$\phi(0)\in L_{\mathcal{F}_{T}}^{p}(\Omega;\mathbb{R}^{m})$, $b(\cdot
,0,0,0)\in L_{\mathcal{F}}^{1,p}([0,T];\mathbb{R}^{n})$, $g(\cdot,0,0,0)\in
L_{\mathcal{F}}^{1,p}([0,T];\mathbb{R}^{m})$, $\sigma(\cdot,0,0,0)\in
L_{\mathcal{F}}^{2,p}([0,T];\mathbb{R}^{n\times d})$.
\end{assumption}

For $p>1$, set
\begin{equation}
\Lambda_{p}:=C_{p}2^{p+1}(1+T^{p})c_{1}^{p}, \label{def-Lambda}%
\end{equation}
where $c_{1}=\max\{L_{2},L_{3}\},$ $C_{p}$ is defined in Lemma \ref{sde-bsde}
in appendix.

\begin{theorem}
Suppose Assumption \ref{assum-1} holds and\ $\Lambda_{p}<1$ for some $p>1$.
$\ $Then \eqref{fbsde} admits a unique solution $(X\left(  \cdot\right)
,Y\left(  \cdot\right)  ,Z\left(  \cdot\right)  )\in L_{\mathcal{F}}%
^{p}(\Omega;C([0,T],\mathbb{R}^{n}))\times L_{\mathcal{F}}^{p}(\Omega
;C([0,T],\mathbb{R}^{m}))\times L_{\mathcal{F}}^{2,p}([0,T];\mathbb{R}%
^{m\times d})$ and
\[%
\begin{array}
[c]{l}%
||(X,Y,Z)||_{p}^{p}=\mathbb{E}\left\{  \sup\limits_{t\in\lbrack0,T]}\left[
|X(t)|^{p}+|Y(t)|^{p}\right]  +\left(  \int_{0}^{T}|Z(t)|^{2}dt\right)
^{\frac{p}{2}}\right\} \\
\ \leq C\mathbb{E}\left\{  \left(  \int_{0}^{T}[|b|+|g|](t,0,0,0)dt\right)
^{p}+\left(  \int_{0}^{T}|\sigma(t,0,0,0)|^{2}dt\right)  ^{\frac{p}{2}}%
+|\phi(0)|^{p}+|x_{0}|^{p}\right\}  ,
\end{array}
\]
where $C$ depends on $T$, $p$, $L_{1}$, $c_{1}$. \label{est-fbsde-lp}
\end{theorem}

\begin{proof}
Without loss of generality, we only prove the case $n=m=d=1$.

Let $\mathcal{L}$ denote the space of all $\mathbb{F}$-adapted processes
$(Y(\cdot),Z(\cdot))$ such that
\[
\mathbb{E}\left[  \sup\limits_{0\leq t\leq T}|Y(t)|^{p}+\left(  \int_{0}%
^{T}|Z(t)|^{2}dt\right)  ^{\frac{p}{2}}\right]  <\infty.
\]
For each given $(y,z)\in\mathcal{L}$, consider the following FBSDE:
\begin{equation}
\left\{
\begin{array}
[c]{rl}%
dX(t)= & b(t,X(t),y(t),z(t))dt+\sigma(t,X(t),y(t),z(t))dB(t),\\
dY(t)= & -g(t,X(t),Y(t),Z(t))dt+Z(t)dB(t),\\
X(0)= & x_{0},\ Y(T)=\phi(X(T)).
\end{array}
\right.  \label{fbsde-y0}%
\end{equation}
Under Assumption \ref{assum-1}, it is easy to deduce that the solution
$(Y(\cdot),Z(\cdot))$ of \eqref{fbsde-y0} belongs to $\mathcal{L}$. Denote the
operator $(y(\cdot),z(\cdot))\rightarrow(Y(\cdot),Z(\cdot))$ by $\Gamma$. For
two elements $(y^{i},z^{i})\in\mathcal{L}$, $i=1,2$, let $(X^{i}(\cdot
),Y^{i}(\cdot),Z^{i}(\cdot))$ be the corresponding solution to \eqref{fbsde-y0}.

Set
\[
\Delta y=y^{1}-y^{2},\text{ }\Delta z=z^{1}-z^{2},\text{ }\Delta X=X^{1}%
-X^{2},\text{ }\Delta Y=Y^{1}-Y^{2},\text{ }\Delta Z=Z^{1}-Z^{2}.
\]
Then
\begin{equation}
\left\{
\begin{array}
[c]{rl}%
d\Delta X(t)= & \left[  \alpha_{1}(t)\Delta X(t)+\beta_{1}(t)\Delta
y(t)+\gamma_{1}(t)\Delta z(t)\right]  dt+\left[  \alpha_{2}(t)\Delta
X(t)+\beta_{2}(t)\Delta y(t)+\gamma_{2}(t)\Delta z(t)\right]  dB(t),\\
d\Delta Y(t)= & -\left[  \alpha_{3}(t)\Delta X(t)+\beta_{3}(t)\Delta
Y(t)+\gamma_{3}(t)\Delta Z(t)\right]  dt+\Delta Z(t)dB(t),\\
\Delta X(0)= & 0,\ \Delta Y(T)=\lambda(T)\Delta X(T),
\end{array}
\right.
\end{equation}
where
\[
\alpha_{1}(t)=\left\{
\begin{array}
[c]{ll}%
\frac{b(t,X^{1}(t),y^{1}(t),z^{1}(t))-b(t,X^{2}(t),y^{1}(t),z^{1}(t))}{\Delta
X(t)},\  & \text{if}\ \Delta X(t)\neq0,\\
0, & \text{if}\ \Delta X(t)=0,
\end{array}
\right.
\]
and $\alpha_{i}(t)$, $\beta_{i}(t)$, $\gamma_{i}(t)$, $\lambda(T)$ are defined
similarly. Furthermore, $\alpha_{i}(t)$, $\beta_{i}(t)$, $\gamma_{i}(t)$,
$\lambda(T)$ are bounded by Lipschitz constants of the corresponding
coefficients. Especially, $|\beta_{1}(t)|,|\gamma_{1}(t)|,|\beta
_{2}(t)|,|\gamma_{2}(t)|\leq c_{1}$. Due to Lemma \ref{sde-bsde}, we obtain
\begin{equation}%
\begin{array}
[c]{l}%
\mathbb{E}\left[  \sup\limits_{0\leq t\leq T}\left(  |\Delta X(t)|^{p}+|\Delta
Y(t)|^{p}\right)  +\left(  \int_{0}^{T}|\Delta Z(t)|^{2}dt\right)  ^{\frac
{p}{2}}\right] \\
\leq C_{p}\mathbb{E}\left\{  \left[  \int_{0}^{T}(|\beta_{1}(t)||\Delta
y(t)|+|\gamma_{1}(t)||\Delta z(t)|)dt\right]  ^{p}+\left[  \int_{0}^{T}\left(
|\beta_{2}(t)|^{2}|\Delta y(t)|^{2}+|\gamma_{2}(t)|^{2}|\Delta z(t)|^{2}%
\right)  dt\right]  ^{\frac{p}{2}}\right\} \\
\leq C_{p}2^{p+1}\left(  1+T^{p}\right)  c_{1}^{p}\mathbb{E}\left[
\sup\limits_{0\leq t\leq T}|\Delta y(t)|^{p}+\left(  \int_{0}^{T}|\Delta
z(t)|^{2}dt\right)  ^{\frac{p}{2}}\right] \\
=\Lambda_{p}\mathbb{E}\left[  \sup\limits_{0\leq t\leq T}|\Delta
y(t)|^{p}+\left(  \int_{0}^{T}|\Delta z(t)|^{2}dt\right)  ^{\frac{p}{2}%
}\right]  .
\end{array}
\label{fbsde-delty}%
\end{equation}
Since $\Lambda_{p}<1$, the operator $\Gamma$ is a contraction mapping and has
a unique fixed point $(Y(\cdot),Z(\cdot))$.\ Let $X(\cdot)$ be the solution of
\eqref{fbsde} with respect to the fixed point $(Y(\cdot),Z(\cdot))$. Thus,
$(X(\cdot),Y(\cdot),Z(\cdot))$ is the unique solution to \eqref{fbsde}.

Let $\Theta^{0}:=(X^{0}(\cdot),Y^{0}(\cdot),Z^{0}(\cdot))$ be the solution to
\eqref{fbsde-y0} with $y=0$, $z=0$. From \eqref{fbsde-delty},
\[
||(Y-Y^{0},Z-Z^{0})||\leq\Lambda_{p}^{\frac{1}{p}}||(Y-0,Z-0)||=\Lambda
_{p}^{\frac{1}{p}}||(Y,Z)||.
\]
By triangle inequality,
\[%
\begin{array}
[c]{rl}%
||(Y,Z)||\leq & ||(Y-Y^{0},Z-Z^{0})||+||(Y^{0},Z^{0})||\\
\leq & \Lambda_{p}^{\frac{1}{p}}||(Y,Z)||+||(Y^{0},Z^{0})||,
\end{array}
\]
which leads to
\[
||(Y,Z)||\leq\left(  1-\Lambda_{p}^{\frac{1}{p}}\right)  ^{-1}||(Y^{0}%
,Z^{0})||.
\]
By Lemma \ref{sde-bsde} in appendix, we obtain
\[%
\begin{array}
[c]{rl}%
||(Y^{0},Z^{0})||^{p}\leq & C_{p}\mathbb{E}\left[  |\phi(0)|^{p}+|x_{0}%
|^{p}+\left(  \int_{0}^{T}[|b|+|g|](t,0,0,0)dt\right)  ^{p}+\left(  \int%
_{0}^{T}|\sigma(t,0,0,0)|^{2}dt\right)  ^{\frac{p}{2}}\right]  ,
\end{array}
\]
where $C_{p}\ $depends on $T$, $p$, $L_{1}$. Thus we have%
\[
||(Y,Z)||_{p}^{p}\leq C^{\prime}\mathbb{E}\left[  |\phi(0)|^{p}+|x_{0}%
|^{p}+\left(  \int_{0}^{T}[|b|+|g|](t,0,0,0)dt\right)  ^{p}+\left(  \int%
_{0}^{T}|\sigma(t,0,0,0)|^{2}dt\right)  ^{\frac{p}{2}}\right]  ,
\]
where$\ C^{\prime}=C_{p}(1-\Lambda_{p}^{\frac{1}{p}})^{-p}$. By Lemma
\ref{sde-bsde}, we can obtain the desired result.
\end{proof}

\begin{remark}
In the case $p=2$, Pardoux and Tang obtained the $L^{2}$-estimate in
\cite{Pardoux-Tang} (see also \cite{Cvi-Zhang}). Instead of assuming that
$L_{2}$ and $L_{3}$ are small enough as in \cite{Cvi-Zhang}, we assume
$\Lambda_{p}<1$ in this paper. There are other conditions in \cite{Cvi-Zhang}
which can guarantee the existence and uniqueness of \eqref{fbsde}. The readers
may apply the method introduced in the above theorem to obtain the $L^{p}%
$-estimate of \eqref{fbsde} for these conditions similarly.
\end{remark}

\subsection{Problem formulation}

Consider the following fully coupled stochastic control system:
\begin{equation}
\left\{
\begin{array}
[c]{rl}%
dX(t)= & b(t,X(t),Y(t),Z(t),u(t))dt+\sigma(t,X(t),Y(t),Z(t),u(t))dB(t),\\
dY(t)= & -g(t,X(t),Y(t),Z(t),u(t))dt+Z(t)dB(t),\\
X(0)= & x_{0},\ Y(T)=\phi(X(T)),
\end{array}
\right.  \label{state-eq}%
\end{equation}
where%
\[
b:[0,T]\times\mathbb{R}\times\mathbb{R}\times\mathbb{R}^{d}\times
U\rightarrow\mathbb{R},
\]%
\[
\sigma:[0,T]\times\mathbb{R}\times\mathbb{R}\times\mathbb{R}^{d}\times
U\rightarrow\mathbb{R}^{d},
\]%
\[
g:[0,T]\times\mathbb{R}\times\mathbb{R}\times\mathbb{R}^{d}\times
U\rightarrow\mathbb{R},
\]%
\[
\phi:\mathbb{R}\rightarrow\mathbb{R}.
\]

An admissible control $u(\cdot)$ is an $\mathbb{F}$-adapted process with
values in $U$ such that%
\[
\sup\limits_{0\leq t\leq T}\mathbb{E}[|u(t)|^{8}]<\infty,
\]
where the control domain $U$ is a nonempty subset of $\mathbb{R}^{k}$. Denote
the admissible control set by $\mathcal{U}[0,T]$.

Our optimal control problem is to minimize the cost functional
\[
J(u(\cdot))=Y(0)
\]
over $\mathcal{U}[0,T]$:%
\begin{equation}
\underset{u(\cdot)\in\mathcal{U}[0,T]}{\inf}J(u(\cdot)). \label{obje-eq}%
\end{equation}

\section{Stochastic maximum principle}

We derive maximum principle (necessary condition for optimality) for the
optimization problem (\ref{obje-eq}) in this section. For simplicity of
presentation, we only study the case $d=1$, and then present the results for
the general case in subsection \ref{sec-general}. In this section, the
constant $C$ will change from line to line in our proof.

We impose the following assumptions on the coefficients of \eqref{state-eq}.

\begin{assumption}
For $\psi=b,$ $\sigma,$ $g$ and $\phi$, we suppose

(i) $\psi$, $\psi_{x}$, $\psi_{y}$, $\psi_{z}$ are continuous in $(x,y,z,u)$;
$\psi_{x}$, $\psi_{y}$, $\psi_{z}$ are bounded; there exists a constant
$\ L>0$ such that%
\[%
\begin{array}
[c]{rl}%
|\psi(t,x,y,z,u)| & \leq L\left(  1+|x|+|y|+|z|+|u|\right)  ,\\
|\sigma(t,0,0,z,u)-\sigma(t,0,0,z,u^{\prime})| & \leq L(1+|u|+|u^{\prime}|).
\end{array}
\]

(ii) For any $2\leq\beta\leq8$,\ $\Lambda_{\beta}:=C_{\beta}2^{\beta
+1}(1+T^{\beta})c_{1}^{\beta}<1$, where $c_{1}=\max\{L_{2},L_{3}\}$,
$L_{2}=\max\{||b_{y}||_{\infty},||b_{z}||_{\infty},||\sigma_{y}||_{\infty}\}$,
$L_{3}=||\sigma_{z}||_{\infty}$, $C_{\beta}$ is defined in Lemma
\ref{sde-bsde} in appendix for $L_{1}=\max\{||b_{x}||_{\infty},||\sigma
_{x}||_{\infty},||g_{x}||_{\infty},||g_{y}||_{\infty},||g_{z}||_{\infty
},||\phi_{x}||_{\infty}\}$.

(iii) $\psi_{xx}$, $\psi_{xy}$, $\psi_{yy}$ , $\psi_{xz}$, $\psi_{yz}$,
$\psi_{zz}$ are continuous in $(x,y,z,u)$; $\psi_{xx}$, $\psi_{xy}$,
$\psi_{yy}$, $\psi_{xz}$, $\psi_{yz}$ ,$\psi_{zz}$ are bounded.

\label{assum-2}
\end{assumption}

Under Assumption \ref{assum-2}(i)-(ii), for any $u(\cdot)\in\mathcal{U}[0,T]$,
the state equation \eqref{state-eq} has a unique solution by Theorem
\ref{est-fbsde-lp}.

Let $\bar{u}(\cdot)$ be optimal and $(\bar{X}(\cdot),\bar{Y}(\cdot),\bar
{Z}(\cdot))$ be the corresponding state processes of (\ref{state-eq}). Since
the control domain is not necessarily convex, we resort to spike variation
method. For any $u(\cdot)\in\mathcal{U}[0,T]$ and $0<\epsilon<T$, define
\[
u^{\epsilon}(t)=\left\{
\begin{array}
[c]{lll}%
\bar{u}(t), & \ t\in\lbrack0,T]\backslash E_{\epsilon}, & \\
u(t), & \ t\in E_{\epsilon}, &
\end{array}
\right.
\]
where $E_{\epsilon}\subset\lbrack0,T]$ is\ a measurable set with
$|E_{\epsilon}|=\epsilon$. Let $(X^{\epsilon}(\cdot),Y^{\epsilon}%
(\cdot),Z^{\epsilon}(\cdot))$ be the state processes of (\ref{state-eq})
associated with $u^{\epsilon}(\cdot)$.

For simplicity, for $\psi=b$, $\sigma$, $g$, $\phi$ and $\kappa=x$, $y$, $z$,
denote%
\[%
\begin{array}
[c]{rl}%
\psi(t)= & \psi(t,\bar{X}(t),\bar{Y}(t),\bar{Z}(t),\bar{u}(t)),\\
\psi_{\kappa}(t)= & \psi_{\kappa}(t,\bar{X}(t),\bar{Y}(t),\bar{Z}(t),\bar
{u}(t)),\\
\delta\psi(t)= & \psi(t,\bar{X}(t),\bar{Y}(t),\bar{Z}(t),u(t))-\psi(t),\\
\delta\psi_{\kappa}(t)= & \psi_{\kappa}(t,\bar{X}(t),\bar{Y}(t),\bar
{Z}(t),u(t))-\psi_{\kappa}(t),\\
\delta\psi(t,\Delta)= & \psi(t,\bar{X}(t),\bar{Y}(t),\bar{Z}(t)+\Delta
(t),u(t))-\psi(t),\\
\delta\psi_{\kappa}(t,\Delta)= & \psi_{\kappa}(t,\bar{X}(t),\bar{Y}(t),\bar
{Z}(t)+\Delta(t),u(t))-\psi_{\kappa}(t),
\end{array}
\]
where $\Delta(\cdot)$ is an $\mathbb{F}$--adapted process. Moreover, denote
$D\psi$ is the gradient of $\psi$ with respect to $x$, $y$, $z$, and
$D^{2}\psi$ is the Hessian matrix of $\psi$ with respect to $x$, $y$, $z$,%
\begin{align*}
D\psi(t)  &  =D\psi(t,\bar{X}(t),\bar{Y}(t),\bar{Z}(t),\bar{u}(t)),\\
D^{2}\psi(t)  &  =D^{2}\psi(t,\bar{X}(t),\bar{Y}(t),\bar{Z}(t),\bar{u}(t)).
\end{align*}

\begin{lemma}
\label{est-epsilon-bar}Suppose Assumption \ref{assum-2}(i)-(ii) hold. Then for
any $2\leq\beta\leq8$ we have
\begin{equation}
\mathbb{E}\left[  \sup\limits_{t\in\lbrack0,T]}\left(  |X^{\epsilon}%
(t)-\bar{X}(t)|^{\beta}+|Y^{\epsilon}(t)-\bar{Y}(t)|^{\beta}\right)  \right]
+\mathbb{E}\left[  \left(  \int_{0}^{T}|Z^{\epsilon}(t)-\bar{Z}(t)|^{2}%
dt\right)  ^{\frac{\beta}{2}}\right]  =O\left(  \epsilon^{\frac{\beta}{2}%
}\right)  .
\end{equation}

\end{lemma}

\begin{proof}
Let
\[%
\begin{array}
[c]{rl}%
\xi^{1,\epsilon}(t) & :=X^{\epsilon}(t)-\bar{X}(t);\\
\eta^{1,\epsilon}(t) & :=Y^{\epsilon}(t)-\bar{Y}(t);\\
\zeta^{1,\epsilon}(t) & :=Z^{\epsilon}(t)-\bar{Z}(t);\\
\Theta(t) & :=(\bar{X}(t),\bar{Y}(t),\bar{Z}(t));\\
\Theta^{\epsilon}(t) & :=(X^{\epsilon}(t),Y^{\epsilon}(t),Z^{\epsilon}(t)).
\end{array}
\]
We have
\begin{equation}
\left\{
\begin{array}
[c]{rl}%
d\xi^{1,\epsilon}(t)= & \left[  \tilde{b}_{x}^{\epsilon}(t)\xi^{1,\epsilon
}(t)+\tilde{b}_{y}^{\epsilon}(t)\eta^{1,\epsilon}(t)+\tilde{b}_{z}^{\epsilon
}(t)\zeta^{1,\epsilon}(t)+\delta b(t)I_{E_{\epsilon}}(t)\right]  dt\\
& +\left[  \tilde{\sigma}_{x}^{\epsilon}(t)\xi^{1,\epsilon}(t)+\tilde{\sigma
}_{y}^{\epsilon}(t)\eta^{1,\epsilon}(t)+\tilde{\sigma}_{z}^{\epsilon}%
(t)\zeta^{1,\epsilon}(t)+\delta\sigma(t)I_{E_{\epsilon}}(t)\right]  dB(t),\\
\xi^{1,\epsilon}(0)= & 0,
\end{array}
\right.  \label{ep-bar-x}%
\end{equation}%
\begin{equation}
\left\{
\begin{array}
[c]{rl}%
d\eta^{1,\epsilon}(t)= & -\left[  \tilde{g}_{x}^{\epsilon}(t)\xi^{1,\epsilon
}(t)+\tilde{g}_{y}^{\epsilon}(t)\eta^{1,\epsilon}(t)+\tilde{g}_{z}^{\epsilon
}(t)\zeta^{1,\epsilon}(t)+\delta g(t)I_{E_{\epsilon}}(t)\right]
dt+\zeta^{1,\epsilon}(t)dB(t),\\
\eta^{1,\epsilon}(T)= & \tilde{\phi}_{x}^{\epsilon}(T)\xi^{1,\epsilon}(T),
\end{array}
\right.  \label{ep-bar-y}%
\end{equation}
where
\[
\tilde{b}_{x}^{\epsilon}(t)=\int_{0}^{1}b_{x}(t,\Theta(t)+\theta
(\Theta^{\epsilon}(t)-\Theta(t)),u^{\epsilon}(t))d\theta
\]
and $\tilde{b}_{y}^{\epsilon}(t)$, $\tilde{b}_{z}^{\epsilon}(t)$,
$\tilde{\sigma}_{x}^{\epsilon}(t)$, $\tilde{\sigma}_{y}^{\epsilon}(t)$,
$\tilde{\sigma}_{z}^{\epsilon}(t),$ $\tilde{g}_{x}^{\epsilon}(t)$, $\tilde
{g}_{y}^{\epsilon}(t)$, $\tilde{g}_{z}^{\epsilon}(t)$ and $\tilde{\phi}%
_{x}^{\epsilon}(T)$ are defined similarly.

Noting that $\left(  \xi^{1,\epsilon}(t),\eta^{1,\epsilon}(t),\zeta
^{1,\epsilon}(t)\right)  $ is the solution to \eqref{ep-bar-x} and
\eqref{ep-bar-y}, and
\[
\mathbb{E}\left[  \left(  \int_{E_{\epsilon}}|u(t)|dt\right)  ^{\beta}\right]
\leq\epsilon^{\beta-1}\mathbb{E}\left[  \int_{E_{\epsilon}}|u(t)|^{\beta
}dt\right]  ,
\]
then, by Theorem \ref{est-fbsde-lp}, we get
\[%
\begin{array}
[c]{ll}
& \mathbb{E}\left[  \sup\limits_{t\in\lbrack0,T]}\left(  |\xi^{1,\epsilon
}(t)|^{\beta}+|\eta^{1,\epsilon}(t)|^{\beta}\right)  +\left(  \int_{0}%
^{T}|\zeta^{1,\epsilon}(t)|^{2}dt\right)  ^{\frac{\beta}{2}}\right] \\
& \ \ \leq C\mathbb{E}\left[  \left(  \int_{0}^{T}\left(  |\delta
b(t)|I_{E_{\epsilon}}(t)+|\delta g(t)|I_{E_{\epsilon}}(t)\right)  dt\right)
^{\beta}+\left(  \int_{0}^{T}|\delta\sigma(t)|^{2}I_{E_{\epsilon}%
}(t)dt\right)  ^{\frac{\beta}{2}}\right] \\
& \ \ \leq C\mathbb{E}\left[  \left(  \int_{E_{\epsilon}}(1+|\bar{X}%
(t)|+|\bar{Y}(t)|+|\bar{Z}(t)|+|u(t)|+|\bar{u}(t)|)dt\right)  ^{\beta}\right.
\\
& \text{ \ \ \ \ \ \ \ }\left.  +\left(  \int_{E_{\epsilon}}(1+|\bar
{X}(t)|^{2}+|\bar{Y}(t)|^{2}+|u(t)|^{2}+|\bar{u}(t)|^{2})dt\right)
^{\frac{\beta}{2}}\right] \\
& \ \ \leq C\left(  \epsilon^{\beta}+\epsilon^{\frac{\beta}{2}}\right)
\left(  1+\sup\limits_{t\in\lbrack0,T]}\mathbb{E}\left[  |\bar{X}(t)|^{\beta
}+|\bar{Y}(t)|^{\beta}+|u(t)|^{\beta}+|\bar{u}(t)|^{\beta}\right]  \right)
+C\epsilon^{\frac{\beta}{2}}\mathbb{E}\left[  \left(  \int_{0}^{T}|\bar
{Z}(t)|^{2}dt\right)  ^{\frac{\beta}{2}}\right] \\
& \ \ \leq C\epsilon^{\frac{\beta}{2}}.
\end{array}
\]

\end{proof}

\subsection{A heuristic derivation\label{heuristic}}

Before giving the strict proof of the stochastic maximum principle, we
illustrate how to obtain our results formally in this subsection.

By Lemma \ref{est-epsilon-bar}, we have $X^{\epsilon}(t)-\bar{X}(t)\sim
O(\sqrt{\epsilon})$, $Y^{\epsilon}(t)-\bar{Y}(t)\sim O(\sqrt{\epsilon})$ and
$Z^{\epsilon}(t)-\bar{Z}(t)\sim O(\sqrt{\epsilon})$. Suppose that
\begin{equation}%
\begin{array}
[c]{lll}%
X^{\epsilon}(t)-\bar{X}(t) & = & X_{1}(t)+X_{2}(t)+o(\epsilon),\\
Y^{\epsilon}(t)-\bar{Y}(t) & = & Y_{1}(t)+Y_{2}(t)+o(\epsilon),\\
Z^{\epsilon}(t)-\bar{Z}(t) & = & Z_{1}(t)+Z_{2}(t)+o(\epsilon),
\end{array}
\label{heur-1}%
\end{equation}
where $X_{1}(t)\sim O(\sqrt{\epsilon})$, $X_{2}(t)\sim O(\epsilon)$,
$Y_{1}(t)\sim O(\sqrt{\epsilon})$, $Y_{2}(t)\sim O(\epsilon)$, $Z_{1}(t)\sim
O(\sqrt{\epsilon})$ and $Z_{2}(t)\sim O(\epsilon)$.

It is well-known that the solution $Z$ of the FBSDE (\ref{state-eq}) is
closely related to the diffusion term $\sigma$ of the forward SDE of
(\ref{state-eq}). When we adopt the spike variation method and calculate the
variational equation of $X$, the diffusion term of the variational equation
should include the term $\delta\sigma(t)I_{E_{\epsilon}}(t)$. So we guess that
$Z_{1}(t)$ has the following form%
\begin{equation}
Z_{1}(t)=\Delta(t)I_{E_{\epsilon}}(t)+Z_{1}^{\prime}(t). \label{heur-2}%
\end{equation}
where $\Delta(t)$ is an $\mathbb{F}$--adapted process and $Z_{1}^{\prime}(t)$
has good estimates similarly as $X_{1}(t)$. But this form of $Z_{1}(t)$ leads
to great difficulties when we do Taylor's expansion of the coefficients $b,$
$\sigma$ and $g$ with respect to $Z$. Fortunately, we find that $\Delta(t)$
can be determined uniquely by $u(t)$, $\bar{u}(t)$, and the optimal state
$(\bar{X}(t)$, $\bar{Y}(t)$, $\bar{Z}(t))$. Note that in Hu \cite{Hu17},%
\begin{equation}
\Delta(t)=p(t)\left(  \sigma(t,\bar{X}(t),u(t))-\sigma(t,\bar{X}(t),\bar
{u}(t))\right)  \label{heur-hu}%
\end{equation}
where $p(t)$ is the adjoint process. Although $\Delta(t)$ appears in the
expansion of $Z^{\epsilon}(t)-\bar{Z}(t)$, by (\ref{heur-hu}) it is clearly
that $\Delta(t)$ includes the spike variation of control variables. In our
context, we will see lately that $\Delta(t)$ is determined by an algebra
equation (\ref{heur-7}). Thus, when we derive the variational equations, we
should keep the $\Delta(t)I_{E_{\epsilon}}(t)$ term unchanged and do Taylor's
expansions at $\bar{Z}(t)+\Delta(t)I_{E_{\epsilon}}(t)$. This idea is first
applied to a partially coupled FBSDE control system by Hu \cite{Hu17}.
Following this idea, we can derive the first-order and second-order
variational equations for our control system (\ref{state-eq}). The expansions
for $b$ and $\sigma$ are given as follows:%
\[%
\begin{array}
[c]{l}%
b(t,X^{\epsilon}(t),Y^{\epsilon}(t),Z^{\epsilon}(t),u^{\epsilon}(t))-b(t)\\
=b(t,\bar{X}(t)+X_{1}(t)+X_{2}(t),\bar{Y}(t)+Y_{1}(t)+Y_{2}(t),\bar
{Z}(t)+\Delta(t)I_{E_{\epsilon}}(t)+Z_{1}^{\prime}(t)+Z_{2}(t),u^{\epsilon
}(t))-b(t)+o(\epsilon)\\
=b_{x}(t)(X_{1}(t)+X_{2}(t))+b_{y}(t)(Y_{1}(t)+Y_{2}(t))+b_{z}(t)(Z_{1}%
^{\prime}(t)+Z_{2}(t))\\
\text{ }+\frac{1}{2}[X_{1}(t),Y_{1}(t),Z_{1}^{\prime}(t)]D^{2}b(t)[X_{1}%
(t),Y_{1}(t),Z_{1}^{\prime}(t)]^{\intercal}+\delta b(t,\Delta)I_{E_{\epsilon}%
}(t)+o(\epsilon),
\end{array}
\]%
\[%
\begin{array}
[c]{l}%
\sigma(t,X^{\epsilon}(t),Y^{\epsilon}(t),Z^{\epsilon}(t),u^{\epsilon
}(t))-\sigma(t)\\
=\sigma(t,\bar{X}(t)+X_{1}(t)+X_{2}(t),\bar{Y}(t)+Y_{1}(t)+Y_{2}(t),\bar
{Z}(t)+\Delta(t)I_{E_{\epsilon}}(t)+Z_{1}^{\prime}(t)+Z_{2}(t),u^{\epsilon
}(t))-\sigma(t)+o(\epsilon)\\
=\sigma_{x}(t)(X_{1}(t)+X_{2}(t))+\sigma_{y}(t)(Y_{1}(t)+Y_{2}(t))+\sigma
_{z}(t)(Z_{1}^{\prime}(t)+Z_{2}(t))\\
\text{ }+\delta\sigma_{x}(t,\Delta)X_{1}(t)I_{E_{\epsilon}}(t)+\delta
\sigma_{y}(t,\Delta)Y_{1}(t)I_{E_{\epsilon}}(t)+\delta\sigma_{z}%
(t,\Delta)Z_{1}^{\prime}(t)I_{E_{\epsilon}}(t)\\
\text{ }+\frac{1}{2}[X_{1}(t),Y_{1}(t),Z_{1}^{\prime}(t)]D^{2}\sigma
(t)[X_{1}(t),Y_{1}(t),Z_{1}^{\prime}(t)]^{\intercal}+\delta\sigma
(t,\Delta)I_{E_{\epsilon}}(t)+o(\epsilon).
\end{array}
\]
Note that
\[
\int_{0}^{T}\delta b_{x}(t,\Delta)X_{1}(t)I_{E_{\epsilon}}(t)dt\sim
o(\epsilon)\text{ and }\int_{0}^{T}\delta\sigma_{x}(t,\Delta)X_{1}%
(t)I_{E_{\epsilon}}(t)dB(t)\sim O(\epsilon).
\]
So we omit $\delta b_{x}(t,\Delta)X_{1}(t)I_{E_{\epsilon}}(t)$ in the
expansions of $b$ and keep $\delta\sigma_{x}(t,\Delta)X_{1}(t)I_{E_{\epsilon}%
}(t)$ in the expansions of $\sigma$. The expansions for $g$ and $\phi$ are
similar to the expansions for $b$. Then, we obtain the following variational
equations:%
\begin{equation}
\left\{
\begin{array}
[c]{rl}%
d(X_{1}(t)+X_{2}(t))= & \{b_{x}(t)(X_{1}(t)+X_{2}(t))+b_{y}(t)(Y_{1}%
(t)+Y_{2}(t))+b_{z}(t)(Z_{1}^{\prime}(t)+Z_{2}(t))\\
& +\frac{1}{2}[X_{1}(t),Y_{1}(t),Z_{1}^{\prime}(t)]D^{2}b(t)[X_{1}%
(t),Y_{1}(t),Z_{1}^{\prime}(t)]^{\intercal}+\delta b(t,\Delta)I_{E_{\epsilon}%
}(t)\}dt\\
& +\{\sigma_{x}(t)(X_{1}(t)+X_{2}(t))+\sigma_{y}(t)(Y_{1}(t)+Y_{2}%
(t))+\sigma_{z}(t)(Z_{1}^{\prime}(t)+Z_{2}(t))\\
& +\delta\sigma_{x}(t,\Delta)X_{1}(t)I_{E_{\epsilon}}(t)+\delta\sigma
_{y}(t,\Delta)Y_{1}(t)I_{E_{\epsilon}}(t)+\delta\sigma_{z}(t,\Delta
)Z_{1}^{\prime}(t)I_{E_{\epsilon}}(t)\\
& +\frac{1}{2}[X_{1}(t),Y_{1}(t),Z_{1}^{\prime}(t)]D^{2}\sigma(t)[X_{1}%
(t),Y_{1}(t),Z_{1}^{\prime}(t)]^{\intercal}+\delta\sigma(t,\Delta
)I_{E_{\epsilon}}(t)\}dB(t),\\
X_{1}(0)+X_{2}(0)= & 0,
\end{array}
\right.  \label{heur-4}%
\end{equation}%
\begin{equation}
\left\{
\begin{array}
[c]{rl}%
d(Y_{1}(t)+Y_{2}(t))= & -\{g_{x}(t)(X_{1}(t)+X_{2}(t))+g_{y}(t)(Y_{1}%
(t)+Y_{2}(t))+g_{z}(t)(Z_{1}^{\prime}(t)+Z_{2}(t))\\
& +\frac{1}{2}[X_{1}(t),Y_{1}(t),Z_{1}^{\prime}(t)]D^{2}g(t)[X_{1}%
(t),Y_{1}(t),Z_{1}^{\prime}(t)]^{\intercal}+\delta g(t,\Delta)I_{E_{\epsilon}%
}(t)\}dt\\
& +(Z_{1}(t)+Z_{2}(t))dB(t),\\
Y_{1}(T)+Y_{2}(T)= & \phi_{x}(\bar{X}(T))(X_{1}(T)+X_{2}(T))+\frac{1}{2}%
\phi_{xx}(\bar{X}(T))X_{1}^{2}(T).
\end{array}
\right.  \label{heur-4'}%
\end{equation}
Now, we need to derive the first-and second-order variational equations from
(\ref{heur-4}) and (\ref{heur-4'}). Firstly, it is easy to establish the
first-order variational equation for $X_{1}(t)$:%
\begin{equation}%
\begin{array}
[c]{rl}%
dX_{1}(t)= & \left[  b_{x}(t)X_{1}(t)+b_{y}(t)Y_{1}(t)+b_{z}(t)(Z_{1}%
(t)-\Delta(t)I_{E_{\epsilon}}(t))\right]  dt\\
& +\left[  \sigma_{x}(t)X_{1}(t)+\sigma_{y}(t)Y_{1}(t)+\sigma_{z}%
(t)(Z_{1}(t)-\Delta(t)I_{E_{\epsilon}}(t))+\delta\sigma(t,\Delta
)I_{E_{\epsilon}}(t)\right]  dB(t),\\
X_{1}(0)= & 0.
\end{array}
\label{heur-5'}%
\end{equation}
Notice that $Y_{1}(T)=\phi_{x}(\bar{X}(T))X_{1}(T)$. So we guess that
$Y_{1}\left(  t\right)  =p\left(  t\right)  X_{1}\left(  t\right)  $ where
$p\left(  t\right)  $ is the solution of the following adjoint equation
\[
\left\{
\begin{array}
[c]{rl}%
dp(t)= & -\Upsilon(t)dt+q(t)dB(t),\\
p(T)= & \phi_{x}(\bar{X}(T)),
\end{array}
\right.
\]
where $\Upsilon(t)$ is some adapted process which will be determined later. It
is clear that $Y_{1}\left(  t\right)  =p\left(  t\right)  X_{1}\left(
t\right)  $ should include all $O(\sqrt{\epsilon})$-terms of the drift term of
(\ref{heur-4'}). Applying It\^{o}'s formula to $p\left(  t\right)
X_{1}\left(  t\right)  $, we can determine that%
\begin{equation}
\left\{
\begin{array}
[c]{rl}%
dY_{1}(t)= & -\left[  g_{x}(t)X_{1}(t)+g_{y}(t)Y_{1}(t)+g_{z}(t)(Z_{1}%
(t)-\Delta(t)I_{E_{\epsilon}}(t))-q(t)\delta\sigma(t,\Delta)I_{E_{\epsilon}%
}(t)\right]  dt+Z_{1}(t)dB(t),\\
Y_{1}(T)= & \phi_{x}(\bar{X}(T))X_{1}(T),
\end{array}
\right.  \label{heur-5}%
\end{equation}
and $(p(\cdot),q(\cdot))$ satisfies the following equation:%
\begin{equation}
\left\{
\begin{array}
[c]{rl}%
dp(t)= & -\left\{  g_{x}(t)+g_{y}(t)p(t)+g_{z}(t)K_{1}(t)+b_{x}(t)p(t)+b_{y}%
(t)p^{2}(t)\right. \\
& \left.  +b_{z}(t)K_{1}(t)p(t)+\sigma_{x}(t)q(t)+\sigma_{y}(t)p(t)q(t)+\sigma
_{z}(t)K_{1}(t)q(t)\right\}  dt+q(t)dB(t),\\
p(T)= & \phi_{x}(\bar{X}(T)),
\end{array}
\right.  \label{eq-p}%
\end{equation}
where%
\begin{equation}
K_{1}(t)=(1-p(t)\sigma_{z}(t))^{-1}\left[  \sigma_{x}(t)p(t)+\sigma
_{y}(t)p^{2}(t)+q(t)\right]  . \label{def-k1}%
\end{equation}
Thus, we obtain the relationship%
\begin{equation}%
\begin{array}
[c]{l}%
Y_{1}\left(  t\right)  =p\left(  t\right)  X_{1}\left(  t\right)  ,\\
Z_{1}(t)=(1-p(t)\sigma_{z}(t))^{-1}p(t)(\delta\sigma(t,\Delta)-\sigma
_{z}(t)\Delta(t))I_{E_{\epsilon}}(t)+K_{1}(t)X_{1}\left(  t\right)  .
\end{array}
\label{heur-6}%
\end{equation}
Combining (\ref{heur-2}) and (\ref{heur-6}), we obtain
\begin{align*}
\Delta(t)  &  =(1-p(t)\sigma_{z}(t))^{-1}p(t)(\delta\sigma(t,\Delta
)-\sigma_{z}(t)\Delta(t)),\\
Z_{1}^{\prime}(t)  &  =K_{1}(t)X_{1}\left(  t\right)  ,
\end{align*}
which implies the following algebra equation
\begin{equation}
\Delta(t)=p(t)\delta\sigma(t,\Delta). \label{heur-7}%
\end{equation}
From (\ref{heur-4}), (\ref{heur-4'}), (\ref{heur-5'}) and (\ref{heur-5}), it
is easy to deduce that $(X_{2}(\cdot),Y_{2}(\cdot))$ satisfies the following
equation:
\begin{equation}
\left\{
\begin{array}
[c]{rl}%
dX_{2}(t)= & \{b_{x}(t)X_{2}(t)+b_{y}(t)Y_{2}(t)+b_{z}(t)Z_{2}(t)+\delta
b(t,\Delta)I_{E_{\epsilon}}(t)\\
& +\frac{1}{2}\left[  X_{1}(t),Y_{1}(t),Z_{1}(t)-\Delta(t)I_{E_{\epsilon}%
}(t)\right]  D^{2}b(t)\left[  X_{1}(t),Y_{1}(t),Z_{1}(t)-\Delta
(t)I_{E_{\epsilon}}(t)\right]  ^{\intercal}\}dt\\
& +\left\{  \sigma_{x}(t)X_{2}(t)+\sigma_{y}(t)Y_{2}(t)+\sigma_{z}%
(t)Z_{2}(t)+\delta\sigma_{x}(t,\Delta)X_{1}(t)I_{E_{\epsilon}}(t)+\delta
\sigma_{y}(t,\Delta)Y_{1}(t)I_{E_{\epsilon}}(t)\right. \\
& +\delta\sigma_{z}(t,\Delta)\left(  Z_{1}(t)-\Delta(t)I_{E_{\epsilon}%
}(t)\right) \\
& \left.  +\frac{1}{2}\left[  X_{1}(t),Y_{1}(t),Z_{1}(t)-\Delta
(t)I_{E_{\epsilon}}(t)\right]  D^{2}\sigma(t)\left[  X_{1}(t),Y_{1}%
(t),Z_{1}(t)-\Delta(t)I_{E_{\epsilon}}(t)\right]  ^{\intercal}\right\}
dB(t),\\
dY_{2}(t)= & -\left\{  g_{x}(t)X_{2}(t)+g_{y}(t)Y_{2}(t)+g_{z}(t)Z_{2}%
(t)+\left[  q(t)\delta\sigma(t,\Delta)+\delta g(t,\Delta)\right]
I_{E_{\epsilon}}(t)\right. \\
& \left.  +\frac{1}{2}\left[  X_{1}(t),Y_{1}(t),Z_{1}(t)-\Delta
(t)I_{E_{\epsilon}}(t)\right]  D^{2}g(t)\left[  X_{1}(t),Y_{1}(t),Z_{1}%
(t)-\Delta(t)I_{E_{\epsilon}}(t)\right]  ^{\intercal}\right\}  dt+Z_{2}%
(t)dB(t),\\
X_{2}(0)= & 0,\text{ }Y_{2}(T)=\phi_{x}(\bar{X}(T))X_{2}(T)+\frac{1}{2}%
\phi_{xx}(\bar{X}(T))X_{1}^{2}(T).
\end{array}
\right.  \label{heur-8}%
\end{equation}
In the following two subsections, we give the rigorous proofs for the above
heuristic derivations.

\subsection{First-order expansion}

From the heuristic derivation, in order to obtain the first-order variational
equation of \eqref{state-eq}, we need to introduce the first-order adjoint
equation (\ref{eq-p}). Since the generator of \eqref{eq-p} does not satisfy
Lipschitz condition, we firstly explore the solvability of \eqref{eq-p}.

For $\beta_{0}>0$ and $y\in\mathbb{R}$, set
\[
G(y)=L_{1}+\left(  L_{2}+L_{1}+\beta_{0}^{-1}L_{1}L_{2}\right)  |y|+\left[
L_{2}+\beta_{0}^{-1}(L_{1}L_{2}+L_{2}^{2})\right]  y^{2}+\beta_{0}^{-1}%
L_{2}^{2}|y|^{3},\ y\in\mathbb{R}\text{.}%
\]
Let $s(\cdot)$ be the maximal solution to the following equation:%
\begin{equation}
s(t)=L_{1}+\int_{t}^{T}G(s(r))dr,\;t\in\lbrack0,T]; \label{u-ode}%
\end{equation}
and $l(\cdot)$ be the minimal solution to the following equation:
\begin{equation}
l(t)=-L_{1}-\int_{t}^{T}G(l(r))dr,\;t\in\lbrack0,T]. \label{l-ode}%
\end{equation}
Moreover, set
\begin{equation}
t_{1}=T-\int_{-\infty}^{-L_{1}}\frac{1}{G(y)}dy,\ \ t_{2}=T-\int_{L_{1}%
}^{\infty}\frac{1}{G(y)}dy,\ \ t^{\ast}=t_{1}\vee t_{2}. \label{def-t}%
\end{equation}

\begin{lemma}
\label{t-star} For given $\beta_{0}>0$, then there exists a $\delta>0$ such
that when $L_{2}<\delta$, we have $t^{\ast}<0$.
\end{lemma}

\begin{proof}
We only prove that there exists a $\delta>0$ such that when $L_{2}<\delta$, we
have $t_{2}<0$. Note that $G(y)$ is a monotonic function with respect to
$L_{2}$. As $L_{2}\rightarrow0$,
\[
\frac{1}{G(y)}\uparrow\frac{1}{L_{1}(1+|y|)}.
\]
Applying the monotone convergence theorem, we obtain
\[
\int_{L_{1}}^{\infty}\frac{1}{G(y)}dy\uparrow\int_{L_{1}}^{\infty}\frac
{1}{L_{1}(1+|y|)}dy=\infty.
\]
By the definition of $t_{2}$, the result is obvious.
\end{proof}

\begin{assumption}
\label{assum-3} There exists a positive constant $\beta_{0}\in(0,1)$ such
that
\[
t^{\ast}<0,
\]
and%
\begin{equation}
\lbrack s(0)\vee(-l(0))]L_{3}\leq1-\beta_{0}. \label{assum-cl}%
\end{equation}

\end{assumption}

\begin{remark}
Note that $G(\cdot)$ is independent of $L_{3}$. Therefore by Lemma
\ref{t-star}, Assumption \ref{assum-3} holds when $L_{2}$ and $L_{3}$ are
small enough.
\end{remark}

\begin{theorem}
\label{exist-unique-BSDE}Suppose Assumptions \ref{assum-2}(i)-(ii) and
\ref{assum-3} hold. Then \eqref{eq-p} has a bounded solution $\ $such that
\[
|p(t)|\leq\lbrack s(0)\vee(-l(0))],
\]
and $q(\cdot)\in L_{\mathcal{F}}^{2,\beta}([0,T];\mathbb{R})$, for any
$\beta\geq1$.
\end{theorem}

\begin{proof}
The proof of the existence is a direct consequence of Theorem 4 in
\cite{LS02}. Furthermore, similar to Corollary 4 in \cite{HuBSDEquad}, we can
obtain $q(\cdot)\in L_{\mathcal{F}}^{2,\beta}\left(  [0,T];\mathbb{R}\right)
$ for any $\beta\geq1$.
\end{proof}

\begin{remark}
The proof of the uniqueness for $(p(\cdot),q(\cdot))$ can be found in Theorem
\ref{unique-pq} in the appendix.
\end{remark}

In order to introduce the first-order variational equation, we study the
following algebra equation
\begin{equation}
\Delta(t)=p(t)(\sigma(t,\bar{X}(t),\bar{Y}(t),\bar{Z}(t)+\Delta
(t),u(t))-\sigma(t,\bar{X}(t),\bar{Y}(t),\bar{Z}(t),\bar{u}(t))),\;t\in
\lbrack0,T], \label{def-delt}%
\end{equation}
where $u\left(  \cdot\right)  $ is a given admissible control.

\begin{remark}
It should be note that the $\Delta(t)$ depends on the optimal control $\bar
{u}(\cdot)$, the adjoint process $p(\cdot)$, and the control $u(\cdot)$.
\end{remark}

\begin{lemma}
\label{delta-exist}Under the same assumptions as in Theorem
\ref{exist-unique-BSDE}. Then (\ref{def-delt}) has a unique adapted solution
$\Delta(\cdot)$. Moreover,%
\begin{equation}%
\begin{array}
[c]{c}%
|\Delta(t)|\leq C(1+|\bar{X}(t)|+|\bar{Y}(t)|+|u(t)|+|\bar{u}(t)|),\text{
}t\in\lbrack0,T],\\
\sup\limits_{0\leq t\leq T}\mathbb{E}[|\Delta(t)|^{8}]<\infty,
\end{array}
\label{del-new-3}%
\end{equation}
where $C$ is a constant depending on $\beta_{0}$, $L$, $L_{1}$, $L_{2}$,
$L_{3}$, $T$.
\end{lemma}

\begin{proof}
We first prove uniqueness. Let $\Delta(\cdot)$ and $\Delta^{\prime}(\cdot)$ be
two adapted solutions to (\ref{def-delt}). Then%
\begin{equation}%
\begin{array}
[c]{l}%
\left\vert \Delta(t)-\Delta^{\prime}(t)\right\vert \\
=|p(t)||\sigma(t,\bar{X}(t),\bar{Y}(t),\bar{Z}(t)+\Delta(t),u(t))-\sigma
(t,\bar{X}(t),\bar{Y}(t),\bar{Z}(t)+\Delta^{\prime}(t),u(t))|\\
\leq\lbrack s(0)\vee(-l(0))]L_{3}\left\vert \Delta(t)-\Delta^{\prime
}(t)\right\vert \\
\leq(1-\beta_{0})\left\vert \Delta(t)-\Delta^{\prime}(t)\right\vert ,
\end{array}
\label{del-new-1}%
\end{equation}
which implies $\Delta(t)=\Delta^{\prime}(t)$ for $t\in\lbrack0,T]$. Now we
construct a contraction mapping on $L_{\mathcal{F}}^{2}([0,T];\mathbb{R})$ to
prove the existence. For each given $\tilde{\Delta}(\cdot)\in L_{\mathcal{F}%
}^{2}([0,T];\mathbb{R})$, define the operator $\tilde{\Delta}(\cdot
)\rightarrow\Delta(\cdot)$ by $\Gamma$, where%
\[
\Delta(t)=p(t)(\sigma(t,\bar{X}(t),\bar{Y}(t),\bar{Z}(t)+\tilde{\Delta
}(t),u(t))-\sigma(t,\bar{X}(t),\bar{Y}(t),\bar{Z}(t),\bar{u}(t))),\;t\in
\lbrack0,T].
\]
Following the same steps as (\ref{del-new-1}), we can get%
\begin{equation}
|\Delta(t)|\leq(1-\beta_{0})\left\vert \tilde{\Delta}(t)\right\vert
+p(t)\delta\sigma(t),\;t\in\lbrack0,T], \label{del-new-2}%
\end{equation}
which implies $\Delta(\cdot)\in L_{\mathcal{F}}^{2}([0,T];\mathbb{R})$. For
each $\tilde{\Delta}_{i}(\cdot)\in L_{\mathcal{F}}^{2}([0,T];\mathbb{R})$,
denote $\Delta_{i}(\cdot)=\Gamma\left(  \tilde{\Delta}_{i}(\cdot)\right)  $,
$i=1$, $2$. Similar to (\ref{del-new-1}), we have%
\[
\mathbb{E}\left[  \int_{0}^{T}\left\vert \Delta_{1}(t)-\Delta_{2}%
(t)\right\vert ^{2}dt\right]  \leq(1-\beta_{0})^{2}\mathbb{E}\left[  \int%
_{0}^{T}\left\vert \tilde{\Delta}_{1}(t)-\tilde{\Delta}_{2}(t)\right\vert
^{2}dt\right]  .
\]
Thus, by the contraction mapping theorem, (\ref{def-delt}) has an adapted
solution $\Delta(\cdot)\in L_{\mathcal{F}}^{2}([0,T];\mathbb{R})$. Moreover,
for any adapted solution $\Delta(\cdot)$ to (\ref{def-delt}), it follows from
(\ref{del-new-2}) with $\tilde{\Delta}(\cdot)=\Delta(\cdot)$ that%
\[
|\Delta(t)|\leq\beta_{0}^{-1}p(t)\delta\sigma(t)\leq C(1+|\bar{X}(t)|+|\bar
{Y}(t)|+|u(t)|+|\bar{u}(t)|),\text{ }t\in\lbrack0,T],
\]
which implies (\ref{del-new-3}).
\end{proof}

\begin{remark}
If the diffusion term is independent of $z$, that is $\sigma_{z}(\cdot)=0$,
then we obtain $\Delta(t)=p(t)\delta\sigma(t)$. If the diffusion term contains
$z$ in a linear form, for example $\sigma(t,x,y,z,u)=A(t)z+\sigma
_{1}(t,x,y,u)$, then we obtain
\[
\Delta(t)=\left(  1-p(t)A(t)\right)  ^{-1}p(t)\left(  \sigma_{1}(t,\bar
{X}(t),\bar{Y}(t),u(t))-\sigma_{1}(t,\bar{X}(t),\bar{Y}(t),\bar{u}(t))\right)
.
\]

\end{remark}

Now we introduce the first-order variational equation:
\begin{equation}
\left\{
\begin{array}
[c]{rl}%
dX_{1}(t)= & \left[  b_{x}(t)X_{1}(t)+b_{y}(t)Y_{1}(t)+b_{z}(t)(Z_{1}%
(t)-\Delta(t)I_{E_{\epsilon}}(t))\right]  dt\\
& +\left[  \sigma_{x}(t)X_{1}(t)+\sigma_{y}(t)Y_{1}(t)+\sigma_{z}%
(t)(Z_{1}(t)-\Delta(t)I_{E_{\epsilon}}(t))+\delta\sigma(t,\Delta
)I_{E_{\epsilon}}(t)\right]  dB(t),\\
X_{1}(0)= & 0,
\end{array}
\right.  \label{new-form-x1}%
\end{equation}
and
\begin{equation}
\left\{
\begin{array}
[c]{lll}%
dY_{1}(t) & = & -\left[  g_{x}(t)X_{1}(t)+g_{y}(t)Y_{1}(t)+g_{z}%
(t)(Z_{1}(t)-\Delta(t)I_{E_{\epsilon}}(t))-q(t)\delta\sigma(t,\Delta
)I_{E_{\epsilon}}(t)\right]  dt+Z_{1}(t)dB(t),\\
Y_{1}(T) & = & \phi_{x}(\bar{X}(T))X_{1}(T).
\end{array}
\right.  \label{new-form-y1}%
\end{equation}

\begin{remark}
The existence and uniqueness of $(X_{1}(\cdot),Y_{1}(\cdot),Z_{1}(\cdot))$ in
\eqref{new-form-x1} and \eqref{new-form-y1} is guaranteed by Theorem
\ref{est-fbsde-lp}.
\end{remark}

\begin{assumption}
The solution $q(\cdot)$ of (\ref{eq-p}) is a bounded process.
\label{assm-q-bound}
\end{assumption}

The relationship between $(Y_{1}(t),Z_{1}(t))$ and $X_{1}(t)$ as pointed out
in our heuristic derivation is obtained in the following lemma.

\begin{lemma}
\label{lemma-y1}Suppose Assumptions \ref{assum-2}(i)-(ii), \ref{assum-3} and
\ref{assm-q-bound} hold. Then we have
\begin{align*}
Y_{1}(t)  &  =p(t)X_{1}(t),\\
Z_{1}(t)  &  =K_{1}(t)X_{1}(t)+\Delta(t)I_{E_{\epsilon}}(t),
\end{align*}
where $p(\cdot)$ is the solution of \eqref{eq-p} and $K_{1}\left(
\cdot\right)  $ is given in (\ref{def-k1}).
\end{lemma}

\begin{proof}
Consider the following stochastic differential equation:
\begin{equation}
\left\{
\begin{array}
[c]{rl}%
d\tilde{X}_{1}(t)= & \left\{  \left[  b_{x}(t)+b_{y}(t)p(t)+b_{z}%
(t)K_{1}(t)\right]  \tilde{X}_{1}(t)\right\}  dt\\
& +\left\{  \left[  \sigma_{x}(t)+p(t)\sigma_{y}(t)+\sigma_{z}(t)K_{1}%
(t)\right]  \tilde{X}_{1}(t)+\delta\sigma(t,\Delta)I_{E_{\epsilon}%
}(t)\right\}  dB(t),\\
\tilde{X}_{1}(0)= & 0.
\end{array}
\right.  \label{eq-x1}%
\end{equation}
It is easy to check that there exists a unique solution $\tilde{X}_{1}(\cdot)$
of \eqref{eq-x1}.

Set%
\begin{align}
\tilde{Y}_{1}(t)  &  =p(t)\tilde{X}_{1}(t),\label{first order-relation}\\
\tilde{Z}_{1}(t)  &  =K_{1}(t)\tilde{X}_{1}(t)+\Delta(t)I_{E_{\epsilon}%
}(t).\nonumber
\end{align}
Applying It\^{o}'s lemma to $p(t)\tilde{X}_{1}(t)$,
\[%
\begin{array}
[c]{lll}%
d\tilde{Y}_{1}(t) & = & -\left[  g_{x}(t)\tilde{X}_{1}(t)+g_{y}(t)\tilde
{Y}_{1}(t)+g_{z}(t)\tilde{Z}_{1}(t)-g_{z}(t)\Delta(t)I_{E_{\epsilon}%
}(t)-q(t)\delta\sigma(t,\Delta)I_{E_{\epsilon}}(t)\right]  dt+\tilde{Z}%
_{1}(t)dB(t).
\end{array}
\]
Thus $(\tilde{X}_{1}(\cdot),\tilde{Y}_{1}(\cdot),\tilde{Z}_{1}(\cdot))$ solves
(\ref{new-form-x1}) and (\ref{new-form-y1}), by Theorem \ref{est-fbsde-lp},
\ $(\tilde{X}_{1}(\cdot),\tilde{Y}_{1}(\cdot),\tilde{Z}_{1}(\cdot
))=(X_{1}(\cdot),Y_{1}(\cdot),Z_{1}(\cdot))$. This completes the proof.
\end{proof}

Then we have the following estimates.

\begin{lemma}
\label{est-one-order}Suppose Assumptions \ref{assum-2}, \ref{assum-3} and
\ref{assm-q-bound} hold. Then for any $2\leq\beta\leq8$, we have the following
estimates
\begin{equation}
\mathbb{E}\left[  \sup\limits_{t\in\lbrack0,T]}\left(  |X_{1}(t)|^{\beta
}+|Y_{1}(t)|^{\beta}\right)  \right]  +\mathbb{E}\left[  \left(  \int_{0}%
^{T}|Z_{1}(t)|^{2}dt\right)  ^{\beta/2}\right]  =O(\epsilon^{\beta/2}),
\label{est-x1-y1}%
\end{equation}%
\[
\mathbb{E}\left[  \sup\limits_{t\in\lbrack0,T]}\left(  |X^{\epsilon}%
(t)-\bar{X}(t)-X_{1}(t)|^{2}+|Y^{\epsilon}(t)-\bar{Y}(t)-Y_{1}(t)|^{2}\right)
\right]  +\mathbb{E}\left[  \int_{0}^{T}|Z^{\epsilon}(t)-\bar{Z}%
(t)-Z_{1}(t)|^{2}dt\right]  =O(\epsilon^{2}),
\]%
\[
\mathbb{E}\left[  \sup\limits_{t\in\lbrack0,T]}(|X^{\epsilon}(t)-\bar
{X}(t)-X_{1}(t)|^{4}+|Y^{\epsilon}(t)-\bar{Y}(t)-Y_{1}(t)|^{4})\right]
+\mathbb{E}\left[  \left(  \int_{0}^{T}|Z^{\epsilon}(t)-\bar{Z}(t)-Z_{1}%
(t)|^{2}dt\right)  ^{2}\right]  =o(\epsilon^{2}).
\]

\end{lemma}

\begin{proof}
By Theorem \ref{est-fbsde-lp}, we have
\begin{equation}%
\begin{array}
[c]{l}%
\mathbb{E}\left[  \sup\limits_{t\in\lbrack0,T]}\left(  |X_{1}(t)|^{\beta
}+|Y_{1}(t)|^{\beta}\right)  +\left(  \int_{0}^{T}|Z_{1}(t)|^{2}dt\right)
^{\beta/2}\right] \\
\leq C\mathbb{E}\left[  \left(  \int_{0}^{T}\left[  \left(  |b_{z}%
(t)|+|g_{z}(t)|\right)  |\Delta(t)|+|q(t)\delta\sigma(t,\Delta)|\right]
I_{E_{\epsilon}}(t)dt\right)  ^{\beta}\right] \\
\text{ \ }+C\mathbb{E}\left[  \left(  \int_{0}^{T}\left[  |\sigma_{z}%
(t)\Delta(t)|^{2}+|\delta\sigma(t,\Delta)|^{2}\right]  I_{E_{\epsilon}%
}(t)dt\right)  ^{\beta/2}\right] \\
\leq C\mathbb{E}\left[  \left(  \int_{E_{\epsilon}}\left(  1+|\bar
{X}(t)|+\left\vert \bar{Y}(t)\right\vert +\left\vert \bar{u}%
(t)|+|u(t)\right\vert \right)  dt\right)  ^{\beta}\right] \\
\text{ \ }+C\mathbb{E}\left[  \left(  \int_{E_{\epsilon}}\left(  1+|\bar
{X}(t)|^{2}+\left\vert \bar{Y}(t)\right\vert ^{2}+\left\vert \bar{u}%
(t)|^{2}+|u(t)\right\vert ^{2}\right)  dt\right)  ^{\beta/2}\right] \\
\leq C\epsilon^{^{\beta/2}}.
\end{array}
\label{est-1-order}%
\end{equation}
We use the notations $\xi^{1,\epsilon}(t)$, $\eta^{1,\epsilon}(t)$ and
$\zeta^{1,\epsilon}(t)$ in the proof of Lemma \ref{est-epsilon-bar} and let
\[%
\begin{array}
[c]{rl}%
\xi^{2,\epsilon}(t) & :=X^{\epsilon}(t)-\bar{X}(t)-X_{1}(t);\\
\eta^{2,\epsilon}(t) & :=Y^{\epsilon}(t)-\bar{Y}(t)-Y_{1}(t);\\
\zeta^{2,\epsilon}(t) & :=Z^{\epsilon}(t)-\bar{Z}(t)-Z_{1}(t);\\
\Theta(t) & :=(\bar{X}(t),\bar{Y}(t),\bar{Z}(t));\\
\Theta(t,\Delta I_{E_{\epsilon}}) & :=(\bar{X}(t),\bar{Y}(t),\bar{Z}%
(t)+\Delta(t)I_{E_{\epsilon}}(t));\\
\Theta^{\epsilon}(t) & :=(X^{\epsilon}(t),Y^{\epsilon}(t),Z^{\epsilon}(t)).
\end{array}
\]
Note that
\[%
\begin{array}
[c]{l}%
\delta\sigma(t,\Delta)I_{E_{\epsilon}}(t)=\sigma(t,\bar{X}(t),\bar{Y}%
(t),\bar{Z}(t)+\Delta(t)I_{E_{\epsilon}}(t),u^{\epsilon}(t))-\sigma
(t)=\sigma(t,\Theta(t,\Delta I_{E_{\epsilon}}(t)),u^{\epsilon}(t))-\sigma(t).
\end{array}
\]
We have%
\[%
\begin{array}
[c]{l}%
\sigma(t,\Theta^{\epsilon}(t),u^{\epsilon}(t))-\sigma(t)-\delta\sigma
(t,\Delta)I_{E_{\epsilon}}(t)\\
=\sigma(t,\Theta^{\epsilon}(t),u^{\epsilon}(t))-\sigma(t,\Theta(t,\Delta
I_{E_{\epsilon}}(t)),u^{\epsilon}(t))\\
=\tilde{\sigma}_{x}^{\epsilon}(t)(X^{\epsilon}(t)-\bar{X}(t))+\tilde{\sigma
}_{y}^{\epsilon}(t)(Y^{\epsilon}(t)-\bar{Y}(t))+\tilde{\sigma}_{z}^{\epsilon
}(t)(Z^{\epsilon}(t)-\bar{Z}(t)-\Delta(t)I_{E_{\epsilon}}(t)),
\end{array}
\]
where
\[
\tilde{\sigma}_{x}^{\epsilon}(t)=\int_{0}^{1}\sigma_{x}(t,\Theta(t,\Delta
I_{E_{\epsilon}}(t))+\theta(\Theta^{\epsilon}(t)-\Theta(t,\Delta
I_{E_{\epsilon}}(t))),u^{\epsilon}(t))d\theta,
\]
and $\tilde{\sigma}_{y}^{\epsilon}(t)$, $\tilde{\sigma}_{z}^{\epsilon}%
(t)$\ are defined similarly. \ 

Recall that $\tilde{b}_{x}^{\epsilon}(t)$, $\tilde{b}_{y}^{\epsilon}(t)$,
$\tilde{b}_{z}^{\epsilon}(t)$, $\tilde{g}_{x}^{\epsilon}(t)$, $\tilde{g}%
_{y}^{\epsilon}(t)$, $\tilde{g}_{z}^{\epsilon}(t)$ and $\tilde{\phi}%
_{x}^{\epsilon}(T)$ are defined in Lemma \ref{est-epsilon-bar}. Then,%
\begin{equation}
\left\{
\begin{array}
[c]{ll}%
d\xi^{2,\epsilon}(t)= & \left[  \tilde{b}_{x}^{\epsilon}(t)\xi^{2,\epsilon
}(t)+\tilde{b}_{y}^{\epsilon}(t)\eta^{2,\epsilon}(t)+\tilde{b}_{z}^{\epsilon
}(t)\zeta^{2,\epsilon}(t)+A_{1}^{\epsilon}(t)\right]  dt\\
& +\left[  \tilde{\sigma}_{x}^{\epsilon}(t)\xi^{2,\epsilon}(t)+\tilde{\sigma
}_{y}^{\epsilon}(t)\eta^{2,\epsilon}(t)+\tilde{\sigma}_{z}^{\epsilon}%
(t)\zeta^{2,\epsilon}(t))+B_{1}^{\epsilon}(t)\right]  dB(t),\\
\xi^{2,\epsilon}(0)= & 0,
\end{array}
\right.  \label{deri-x-x1}%
\end{equation}%
\[
\left\{
\begin{array}
[c]{rl}%
d\eta^{2,\epsilon}(t)= & -\left[  \tilde{g}_{x}^{\epsilon}(t)\xi^{2,\epsilon
}(t)+\tilde{g}_{y}^{\epsilon}(t)\eta^{2,\epsilon}(t)+\tilde{g}_{z}^{\epsilon
}(t)\zeta^{2,\epsilon}(t)+C_{1}^{\epsilon}(t)\right]  dt+\zeta^{2,\epsilon
}(t)dB(t),\\
\eta^{2,\epsilon}(T)= & \tilde{\phi}_{x}^{\epsilon}(T)\xi^{2,\epsilon
}(T)+D_{1}^{\epsilon}(T),
\end{array}
\right.
\]
where%
\[%
\begin{array}
[c]{rl}%
A_{1}^{\epsilon}(t)= & (\tilde{b}_{x}^{\epsilon}(t)-b_{x}(t))X_{1}%
(t)+(\tilde{b}_{y}^{\epsilon}(t)-b_{y}(t))Y_{1}(t)+(\tilde{b}_{z}^{\epsilon
}(t)-b_{z}(t))Z_{1}(t)+b_{z}(t)\Delta(t)I_{E_{\epsilon}}(t)+\delta
b(t)I_{E_{\epsilon}}(t),\\
B_{1}^{\epsilon}(t)= & (\tilde{\sigma}_{x}^{\epsilon}(t)-\sigma_{x}%
(t))X_{1}(t)+(\tilde{\sigma}_{y}^{\epsilon}(t)-\sigma_{y}(t))Y_{1}%
(t)+(\tilde{\sigma}_{z}^{\epsilon}(t)-\sigma_{z}(t))K_{1}(t)X_{1}(t),\\
C_{1}^{\epsilon}(t)= & (\tilde{g}_{x}^{\epsilon}(t)-g_{x}(t))X_{1}%
(t)+(\tilde{g}_{y}^{\epsilon}(t)-g_{y}(t))Y_{1}(t)+(\tilde{g}_{z}^{\epsilon
}(t)-g_{z}(t))Z_{1}(t)+\delta g(t)I_{E_{\epsilon}}(t)\\
& +g_{z}(t)\Delta(t)I_{E_{\epsilon}}(t)+q(t)\delta\sigma(t,\Delta
)I_{E_{\epsilon}}(t),\\
D_{1}^{\epsilon}(T)= & (\tilde{\phi}_{x}^{\epsilon}(T)-\phi_{x}(\bar
{X}(T)))X_{1}(T).
\end{array}
\]
By Theorem \ref{est-fbsde-lp}, we obtain
\[%
\begin{array}
[c]{l}%
\mathbb{E}\left[  \sup\limits_{t\in\lbrack0,T]}\left(  |\xi^{2,\epsilon
}(t)|^{2}+|\eta^{2,\epsilon}(t)|^{2}\right)  +\int_{0}^{T}|\zeta^{2,\epsilon
}(t)|^{2}dt\right] \\
\leq C\mathbb{E}\left[  \left(  \int_{0}^{T}\left(  |A_{1}^{\epsilon
}(t)|+|C_{1}^{\epsilon}(t)|\right)  dt\right)  ^{2}+\int_{0}^{T}%
|B_{1}^{\epsilon}(t)|^{2}dt+|D_{1}^{\epsilon}(T)|^{2}\right] \\
\leq C\mathbb{E}\left[  \left(  \int_{0}^{T}|A_{1}^{\epsilon}(t)|dt\right)
^{2}+\left(  \int_{0}^{T}|C_{1}^{\epsilon}(t)|dt\right)  ^{2}+\int_{0}%
^{T}|B_{1}^{\epsilon}(t)|^{2}dt+|D_{1}^{\epsilon}(T)|^{2}\right]  .
\end{array}
\]
Now we estimate term by term as follows.

(1) Since
\begin{equation}%
\begin{array}
[c]{l}%
\mathbb{E}\left[  \left(  \int_{0}^{T}|\tilde{b}_{z}^{\epsilon}(t)-b_{z}%
(t)||Z_{1}(t)|dt\right)  ^{2}\right] \\
\leq C\mathbb{E}\left[  \int_{0}^{T}|\tilde{b}_{z}^{\epsilon}(t)-b_{z}%
(t)|^{2}dt\int_{0}^{T}|Z_{1}(t)|^{2}dt\right] \\
\leq C\left\{  \mathbb{E}\left[  \left(  \int_{0}^{T}|\tilde{b}_{z}^{\epsilon
}(t)-b_{z}(t)|^{2}dt\right)  ^{2}\right]  \right\}  ^{\frac{1}{2}}\left\{
\mathbb{E}\left[  \left(  \int_{0}^{T}|Z_{1}(t)|^{2}dt\right)  ^{2}\right]
\right\}  ^{\frac{1}{2}}\\
\leq C\left\{  \mathbb{E}\left[  \sup\limits_{t\in\lbrack0,T]}\left(
|\xi^{1,\epsilon}(t)|^{4}+|\eta^{1,\epsilon}(t)|^{4}\right)  +\left(  \int%
_{0}^{T}\left(  |\zeta^{1,\epsilon}(t)|^{2}+|\delta b_{z}(t)|^{2}%
I_{E_{\epsilon}}(t)\right)  dt\right)  ^{2}\right]  \right\}  ^{\frac{1}{2}%
}\left\{  \mathbb{E}\left[  \left(  \int_{0}^{T}|Z_{1}(t)|^{2}dt\right)
^{2}\right]  \right\}  ^{\frac{1}{2}}\\
\leq C\epsilon^{2},
\end{array}
\label{est-b_x-z1}%
\end{equation}
the estimate of $\mathbb{E}\left[  \left(  \int_{0}^{T}\left(  \tilde{b}%
_{x}^{\epsilon}(t)-b_{x}(t)\right)  X_{1}(t)dt\right)  ^{2}\right]  $ and
$\mathbb{E}\left\{  \left[  \int_{0}^{T}\left(  \tilde{b}_{y}^{\epsilon
}(t)-b_{y}(t)\right)  Y_{1}(t)dt\right]  ^{2}\right\}  $ is the same as
\eqref{est-b_x-z1},
\begin{equation}
\mathbb{E}\left[  \left(  \int_{0}^{T}|b_{z}(t)\Delta(t)I_{E_{\epsilon}%
}(t)|dt\right)  ^{2}\right]  \leq C\epsilon\int_{E_{\epsilon}}\mathbb{E}%
[|\Delta(t)|^{2}]dt\leq C\epsilon^{2}, \label{est-delta}%
\end{equation}%
\[%
\begin{array}
[c]{ll}%
\mathbb{E}\left[  (\int_{0}^{T}|\delta b(t)I_{E_{\epsilon}}(t)|dt)^{2}\right]
& \leq\mathbb{E}\left[  \left(  \int_{E_{\epsilon}}\left(  1+|\bar
{X}(t)|+\left\vert \bar{Y}(t)\right\vert +\left\vert \bar{Z}(t)\right\vert
+\left\vert u(t)\right\vert +\left\vert \bar{u}(t)\right\vert \right)
dt\right)  ^{2}\right] \\
& \leq\epsilon\mathbb{E}\left[  \int_{E_{\epsilon}}\left(  1+|\bar{X}%
(t)|^{2}+\left\vert \bar{Y}(t)\right\vert ^{2}+\left\vert \bar{Z}%
(t)\right\vert ^{2}+\left\vert u(t)\right\vert ^{2}+\left\vert \bar
{u}(t)\right\vert ^{2}\right)  dt\right] \\
& \leq C\epsilon^{2},
\end{array}
\]
then,
\[
\mathbb{E}\left[  (\int_{0}^{T}|A_{1}^{\epsilon}(t)|dt)^{2}\right]  \leq
C\epsilon^{2}.
\]

(2)
\begin{equation}%
\begin{array}
[c]{l}%
\mathbb{E}\left[  \int_{0}^{T}|\tilde{\sigma}_{z}^{\epsilon}(t)-\sigma
_{z}(t)|^{2}|K_{1}(t)X_{1}(t)|^{2}dt\right] \\
\leq C\mathbb{E}\left[  \sup\limits_{0\leq t\leq T}|X_{1}(t)|^{2}\int_{0}%
^{T}|\tilde{\sigma}_{z}^{\epsilon}(t)-\sigma_{z}(t)|^{2}dt\right] \\
\leq C\left\{  \mathbb{E}\left[  \sup\limits_{0\leq t\leq T}|X_{1}%
(t)|^{4}\right]  \right\}  ^{\frac{1}{2}}\left\{  \mathbb{E}\left[
\sup\limits_{t\in\lbrack0,T]}\left(  |\xi^{1,\epsilon}(t)|^{4}+|\eta
^{1,\epsilon}(t)|^{4}\right)  \right.  \right. \\
\ \ \left.  \left.  +\left(  \int_{0}^{T}\left(  |\zeta^{1,\epsilon}%
(t)-\Delta(t)I_{E_{\epsilon}}(t)|^{2}+|\delta\sigma_{z}(t,\Delta
)|^{2}I_{E_{\epsilon}}(t)\right)  dt\right)  ^{2}\right]  \right\}  ^{\frac
{1}{2}}\\
\leq C\epsilon\left\{  \epsilon^{2}+\epsilon\int_{E_{\epsilon}}\mathbb{E}%
[|\Delta(t)|^{4}]dt\right\}  ^{\frac{1}{2}}\\
\leq C\epsilon^{2},
\end{array}
\label{est-sigmazx1}%
\end{equation}
and the estimate of $\mathbb{E}\left[  \int_{0}^{T}|\tilde{\sigma}%
_{x}^{\epsilon}(t)-\sigma_{x}(t)|^{2}|X_{1}(t)|^{2}dt\right]  $ and
$\mathbb{E}\left[  \int_{0}^{T}|\tilde{\sigma}_{y}^{\epsilon}(t)-\sigma
_{y}(t)|^{2}|Y_{1}(t)|^{2}dt\right]  $ is the same as \eqref{est-sigmazx1}.
Thus,
\begin{equation}
\mathbb{E}\left[  \int_{0}^{T}|B_{1}^{\epsilon}(t)|^{2}dt\right]  \leq
C\epsilon^{2}.
\end{equation}

(3)%
\begin{equation}
\mathbb{E}\left[  |D_{1}^{\epsilon}(T)|^{2}\right]  \leq C\left\{
\mathbb{E}\left[  \sup\limits_{0\leq t\leq T}|X_{1}(t)|^{4}\right]  \right\}
^{\frac{1}{2}}\left\{  \mathbb{E}\left[  \sup\limits_{0\leq t\leq T}%
|\xi^{1,\epsilon}(t)|^{4}\right]  \right\}  ^{\frac{1}{2}}\leq C\epsilon^{2}.
\end{equation}

(4) The estimate of $\mathbb{E}\left[  (\int_{0}^{T}|C_{1}^{\epsilon
}(t)|dt)^{2}\right]  $ is the same as $\mathbb{E}\left[  (\int_{0}^{T}%
|A_{1}^{\epsilon}(t)|dt)^{2}\right]  $.

Similarly, we obtain
\[%
\begin{array}
[c]{l}%
\mathbb{E}\left\{  \sup\limits_{t\in\lbrack0,T]}\left[  |\xi^{2,\epsilon
}(t)|^{4}+|\eta^{2,\epsilon}(t)|^{4}\right]  +\left(  \int_{0}^{T}%
|\zeta^{2,\epsilon}(t)|^{2}dt\right)  ^{2}\right\} \\
\leq C\mathbb{E}\left\{  \left(  \int_{0}^{T}\left(  |A_{1}^{\epsilon
}(t)|+|C_{1}^{\epsilon}(t)|\right)  dt\right)  ^{4}+\left(  \int_{0}^{T}%
|B_{1}^{\epsilon}(t)|^{2}dt\right)  ^{2}+\left\vert D_{1}^{\epsilon
}(T)\right\vert ^{4}\right\} \\
=o(\epsilon^{2}).
\end{array}
\]
This completes the proof.
\end{proof}

\subsection{Second-order expansion}

Noting that $Z_{1}(t)=K_{1}(t)X_{1}(t)+\Delta(t)I_{E_{\epsilon}}(t)$\ in Lemma
\ref{lemma-y1}, then we introduce the second-order variational equation as
follows:
\begin{equation}
\left\{
\begin{array}
[c]{rl}%
dX_{2}(t)= & \left\{  b_{x}(t)X_{2}(t)+b_{y}(t)Y_{2}(t)+b_{z}(t)Z_{2}%
(t)+\delta b(t,\Delta)I_{E_{\epsilon}}(t)\right. \\
& \left.  +\frac{1}{2}\left[  X_{1}(t),Y_{1}(t),K_{1}(t)X_{1}(t)\right]
D^{2}b(t)\left[  X_{1}(t),Y_{1}(t),K_{1}(t)X_{1}(t)\right]  ^{\intercal
}\right\}  dt\\
& +\left\{  \sigma_{x}(t)X_{2}(t)+\sigma_{y}(t)Y_{2}(t)+\frac{1}{2}\left[
X_{1}(t),Y_{1}(t),K_{1}(t)X_{1}(t)\right]  D^{2}\sigma(t)\left[
X_{1}(t),Y_{1}(t),K_{1}(t)X_{1}(t)\right]  ^{\intercal}\right. \\
& \left.  +\sigma_{z}(t)Z_{2}(t)+\left[  \delta\sigma_{x}(t,\Delta
)X_{1}(t)+\delta\sigma_{y}(t,\Delta)Y_{1}(t)\right]  I_{E_{\epsilon}%
}(t)+\delta\sigma_{z}(t,\Delta)K_{1}(t)X_{1}(t)I_{E_{\epsilon}}(t)\right\}
dB(t),\\
X_{2}(0)= & 0,
\end{array}
\right.  \label{new-form-x2}%
\end{equation}
and
\begin{equation}
\left\{
\begin{array}
[c]{ll}%
dY_{2}(t)= & -\left\{  g_{x}(t)X_{2}(t)+g_{y}(t)Y_{2}(t)+g_{z}(t)Z_{2}%
(t)+\left[  q(t)\delta\sigma(t,\Delta)+\delta g(t,\Delta)\right]
I_{E_{\epsilon}}(t)\right. \\
& \left.  +\frac{1}{2}\left[  X_{1}(t),Y_{1}(t),K_{1}(t)X_{1}(t)\right]
D^{2}g(t)\left[  X_{1}(t),Y_{1}(t),K_{1}(t)X_{1}(t)\right]  ^{\intercal
}\right\}  dt+Z_{2}(t)dB(t),\\
Y_{2}(T)= & \phi_{x}(\bar{X}(T))X_{2}(T)+\frac{1}{2}\phi_{xx}(\bar{X}%
(T))X_{1}^{2}(T).
\end{array}
\right.  \label{new-form-y2}%
\end{equation}
In the following lemma, we estimate the orders of $X_{2}(\cdot)$, $Y_{2}%
(\cdot)$, $Z_{2}(\cdot)$, and $Y^{\epsilon}(0)-\bar{Y}(0)-Y_{1}(0)-Y_{2}(0)$.

\begin{lemma}
\label{est-second-order} Suppose Assumptions \ref{assum-2}, \ref{assum-3} and
\ref{assm-q-bound} hold. Then for any $2\leq\beta\leq4$ we have%
\[%
\begin{array}
[c]{rl}%
\mathbb{E}\left[  \sup\limits_{t\in\lbrack0,T]}(|X_{2}(t)|^{2}+|Y_{2}%
(t)|^{2})\right]  +\mathbb{E}\left[  \int_{0}^{T}|Z_{2}(t)|^{2}dt\right]  &
=O(\epsilon^{2}),\\
\mathbb{E}\left[  \sup\limits_{t\in\lbrack0,T]}(|X_{2}(t)|^{\beta}%
+|Y_{2}(t)|^{\beta})\right]  +\mathbb{E}\left[  \left(  \int_{0}^{T}%
|Z_{2}(t)|^{2}dt\right)  ^{\frac{\beta}{2}}\right]  & =o(\epsilon^{\frac
{\beta}{2}}),\\
Y^{\epsilon}(0)-\bar{Y}(0)-Y_{1}(0)-Y_{2}(0) & =o(\epsilon).
\end{array}
\]

\end{lemma}

\begin{proof}
By Theorem \ref{est-fbsde-lp}, we have%
\[%
\begin{array}
[c]{l}%
\mathbb{E}\left[  \sup\limits_{t\in\lbrack0,T]}(|X_{2}(t)|^{2}+|Y_{2}%
(t)|^{2})+\int_{0}^{T}|Z_{2}(t)|^{2}dt\right] \\
\leq C\mathbb{E}\left[  \left(  \int_{0}^{T}[(|\delta b(t,\Delta
)|+|\delta\sigma(t,\Delta)|+|\delta g(t,\Delta)|)I_{E_{\epsilon}}%
(t)+|X_{1}(t)|^{2}+|Y_{1}(t)|^{2}]dt\right)  ^{2}\right] \\
\text{ \ }+C\mathbb{E}\left[  \int_{0}^{T}\left[  |X_{1}(t)|^{4}%
+|Y_{1}(t)|^{4}+(|X_{1}(t)|^{2}+|Y_{1}(t)|^{2})I_{E_{\epsilon}}(t)\right]
dt\right] \\
\leq C\epsilon\mathbb{E}\left[  \int_{E_{\epsilon}}(1+|\bar{X}(t)|^{2}%
+|\bar{Y}(t)|^{2}+|\bar{Z}(t)|^{2}+|u(t)|^{2}+|\bar{u}(t)|^{2})dt\right] \\
\text{ \ }+C\mathbb{E}\left[  \sup\limits_{t\in\lbrack0,T]}\left(
|X_{1}(t)|^{4}+|Y_{1}(t)|^{4}\right)  \right]  +C\epsilon\mathbb{E}\left[
\sup\limits_{t\in\lbrack0,T]}(|X_{1}(t)|^{2}+|Y_{1}(t)|^{2})\right] \\
\leq C\epsilon^{2},
\end{array}
\]%
\begin{equation}%
\begin{array}
[c]{l}%
\mathbb{E}\left[  \sup\limits_{t\in\lbrack0,T]}(|X_{2}(t)|^{\beta}%
+|Y_{2}(t)|^{\beta})+\left(  \int_{0}^{T}|Z_{2}(t)|^{2}dt\right)
^{\frac{\beta}{2}}\right] \\
\leq C\mathbb{E}\left[  \left(  \int_{0}^{T}[(|\delta b(t,\Delta
)|+|\delta\sigma(t,\Delta)|+|\delta g(t,\Delta)|)I_{E_{\epsilon}}%
(t)+|X_{1}(t)|^{2}+|Y_{1}(t)|^{2}]dt\right)  ^{\beta}\right] \\
\text{ \ }+C\mathbb{E}\left[  \left(  \int_{0}^{T}\left[  |X_{1}%
(t)|^{4}+|Y_{1}(t)|^{4}+(|X_{1}(t)|^{2}+|Y_{1}(t)|^{2})I_{E_{\epsilon}%
}(t)\right]  dt\right)  ^{\frac{\beta}{2}}\right] \\
\leq C\epsilon^{\frac{\beta}{2}}\mathbb{E}\left[  \left(  \int_{E_{\epsilon}%
}(1+|\bar{X}(t)|^{2}+|\bar{Y}(t)|^{2}+|\bar{Z}(t)|^{2}+|u(t)|^{2}+|\bar
{u}(t)|^{2})dt\right)  ^{\frac{\beta}{2}}\right] \\
\text{ \ }+C\mathbb{E}\left[  \sup\limits_{t\in\lbrack0,T]}\left(
|X_{1}(t)|^{2\beta}+|Y_{1}(t)|^{2\beta}\right)  \right]  +C\epsilon
^{\frac{\beta}{2}}\mathbb{E}\left[  \sup\limits_{t\in\lbrack0,T]}%
(|X_{1}(t)|^{\beta}+|Y_{1}(t)|^{\beta})\right] \\
=o(\epsilon^{\frac{\beta}{2}}).
\end{array}
\end{equation}
Now, we focus on the last estimate. We use the same notations $\xi
^{1,\epsilon}(t)$, $\eta^{1,\epsilon}(t)$, $\zeta^{1,\epsilon}(t)$,
$\xi^{2,\epsilon}(t)$, $\eta^{2,\epsilon}(t)$ and $\zeta^{2,\epsilon}(t)$ in
the proof of Lemma \ref{est-epsilon-bar} and Lemma \ref{est-one-order}. Let
\[%
\begin{array}
[c]{rl}%
\xi^{3,\epsilon}(t) & :=X^{\epsilon}(t)-\bar{X}(t)-X_{1}(t)-X_{2}(t);\\
\eta^{3,\epsilon}(t) & :=Y^{\epsilon}(t)-\bar{Y}(t)-Y_{1}(t)-Y_{2}(t);\\
\zeta^{3,\epsilon}(t) & :=Z^{\epsilon}(t)-\bar{Z}(t)-Z_{1}(t)-Z_{2}(t);\\
\Theta(t) & :=(\bar{X}(t),\bar{Y}(t),\bar{Z}(t));\\
\Theta(t,\Delta I_{E_{\epsilon}}) & :=(\bar{X}(t),\bar{Y}(t),\bar{Z}%
(t)+\Delta(t)I_{E_{\epsilon}}(t));\\
\Theta^{\epsilon}(t) & :=(X^{\epsilon}(t),Y^{\epsilon}(t),Z^{\epsilon}(t)).
\end{array}
\]

Define $\widetilde{D^{2}b^{\epsilon}}(t)$%
\[
\widetilde{D^{2}b^{\epsilon}}(t)=2\int_{0}^{1}\int_{0}^{1}\theta
D^{2}b(t,\Theta(t,\Delta I_{E_{\epsilon}})+\lambda\theta(\Theta^{\epsilon
}(t)-\Theta(t,\Delta I_{E_{\epsilon}})),u^{\epsilon}(t))d\theta d\lambda,
\]

$\widetilde{D^{2}\sigma^{\epsilon}}(t)$, $\widetilde{D^{2}g^{\epsilon}}(t)$
and $\tilde{\phi}_{xx}^{\epsilon}(T)$ are defined similarly. Then, we have
\begin{equation}
\left\{
\begin{array}
[c]{ll}%
d\xi^{3,\epsilon}(t)= & \left\{  b_{x}(t)\xi^{3,\epsilon}(t)+b_{y}%
(t)\eta^{3,\epsilon}(t)+b_{z}(t)\zeta^{3,\epsilon}(t)+A_{2}^{\epsilon
}(t)\right\}  dt\\
& +\left\{  \sigma_{x}(t)\xi^{3,\epsilon}(t)+\sigma_{y}(t)\eta^{3,\epsilon
}(t)+\sigma_{z}(t)\zeta^{3,\epsilon}(t)+B_{2}^{\epsilon}(t)\right\}  dB(t),\\
\xi^{3,\epsilon}(0)= & 0,
\end{array}
\right.  \label{x-x1-x2}%
\end{equation}
and%
\begin{equation}
\left\{
\begin{array}
[c]{lll}%
d\eta^{3,\epsilon}(t) & = & -\{g_{x}(t)\xi^{3,\epsilon}(t)+g_{y}%
(t)\eta^{3,\epsilon}(t)+g_{z}(t)\zeta^{3,\epsilon}(t)+C_{2}^{\epsilon
}(t)\}dt-\zeta^{3,\epsilon}(t)dB(t),\\
\eta^{3,\epsilon}(T) & = & \phi_{x}(\bar{X}(T))\xi^{3,\epsilon}(T)+D_{2}%
^{\epsilon}(T),
\end{array}
\right.  \label{y-y1-y2}%
\end{equation}
where%
\[%
\begin{array}
[c]{ll}%
A_{2}^{\epsilon}(t)= & \left[  \delta b_{x}(t,\Delta)\xi^{1,\epsilon
}(t)+\delta b_{y}(t,\Delta)\eta^{1,\epsilon}(t)+\delta b_{z}(t,\Delta)\left(
\zeta^{1,\epsilon}(t)-\Delta(t)I_{E_{\epsilon}}(t)\right)  \right]
I_{E_{\epsilon}}(t)\\
& +\frac{1}{2}\left[  \xi^{1,\epsilon}(t),\eta^{1,\epsilon}(t),\zeta
^{1,\epsilon}(t)-\Delta(t)I_{E_{\epsilon}}(t)\right]  \widetilde{D^{2}%
b^{\epsilon}}(t)\left[  \xi^{1,\epsilon}(t),\eta^{1,\epsilon}(t),\zeta
^{1,\epsilon}(t)-\Delta(t)I_{E_{\epsilon}}(t)\right]  ^{\intercal}\\
& -\frac{1}{2}\left[  X_{1}(t),Y_{1}(t),K_{1}(t)X_{1}(t)\right]
D^{2}b(t)\left[  X_{1}(t),Y_{1}(t),K_{1}(t)X_{1}(t)\right]  ^{\intercal},
\end{array}
\]%
\[%
\begin{array}
[c]{ll}%
B_{2}^{\epsilon}(t)= & \left[  \delta\sigma_{x}(t,\Delta)\xi^{2,\epsilon
}(t)+\delta\sigma_{y}(t,\Delta)\eta^{2,\epsilon}(t)+\delta\sigma_{z}%
(t,\Delta)\zeta^{2,\epsilon}(t)\right]  I_{E_{\epsilon}}(t)\\
& +\frac{1}{2}\left[  \xi^{1,\epsilon}(t),\eta^{1,\epsilon}(t),\zeta
^{1,\epsilon}(t)-\Delta(t)I_{E_{\epsilon}}(t)\right]  \widetilde{D^{2}%
\sigma^{\epsilon}}(t)\left[  \xi^{1,\epsilon}(t),\eta^{1,\epsilon}%
(t),\zeta^{1,\epsilon}(t)-\Delta(t)I_{E_{\epsilon}}(t)\right]  ^{\intercal}\\
& -\frac{1}{2}\left[  X_{1}(t),Y_{1}(t),K_{1}(t)X_{1}(t)\right]  D^{2}%
\sigma(t)\left[  X_{1}(t),Y_{1}(t),K_{1}(t)X_{1}(t)\right]  ^{\intercal},
\end{array}
\]%
\[%
\begin{array}
[c]{ll}%
C_{2}^{\epsilon}(t)= & \left[  \delta g_{x}(t,\Delta)\xi^{1,\epsilon
}(t)+\delta g_{y}(t,\Delta)\eta^{1,\epsilon}(t)+\delta g_{z}(t,\Delta)\left(
\zeta^{1,\epsilon}(t)-\Delta(t)I_{E_{\epsilon}}(t)\right)  \right]
I_{E_{\epsilon}}(t)\\
& +\frac{1}{2}\left[  \xi^{1,\epsilon}(t),\eta^{1,\epsilon}(t),\zeta
^{1,\epsilon}(t)-\Delta(t)I_{E_{\epsilon}}(t)\right]  \widetilde{D^{2}%
g^{\epsilon}}(t)\left[  \xi^{1,\epsilon}(t),\eta^{1,\epsilon}(t),\zeta
^{1,\epsilon}(t)-\Delta(t)I_{E_{\epsilon}}(t)\right]  ^{\intercal}\\
& -\frac{1}{2}\left[  X_{1}(t),Y_{1}(t),K_{1}(t)X_{1}(t)\right]
D^{2}g(t)\left[  X_{1}(t),Y_{1}(t),K_{1}(t)X_{1}(t)\right]  ^{\intercal},\\
D_{2}^{\epsilon}(T)= & \frac{1}{2}\tilde{\phi}_{xx}^{\epsilon}(T)\xi
^{1,\epsilon}(T)^{2}-\frac{1}{2}\phi_{xx}(\bar{X}(T))X_{1}^{2}(T).
\end{array}
\]

We introduce the following fully coupled FBSDE:
\begin{equation}
\left\{
\begin{array}
[c]{rl}%
dh(t)= & \left[  g_{y}(t)h(t)+b_{y}(t)m(t)+\sigma_{y}(t)n(t)\right]
dt+\left[  g_{z}(t)h(t)+b_{z}(t)m(t)+\sigma_{z}(t)n(t)\right]  dB(t),\\
h(0)= & 1,\\
dm(t)= & -\left[  g_{x}(t)h(t)+b_{x}(t)m(t)+\sigma_{x}(t)n(t)\right]
dt+n(t)dB(t),\\
m(T)= & \phi_{x}(\bar{X}(T))h(T).
\end{array}
\right.
\end{equation}
It has a unique solution due to Theorem \ref{est-fbsde-lp}. Applying It\^{o}'s
formula to
\[
m(t)\xi^{3,\epsilon}(t)-h(t)\eta^{3,\epsilon}(t),
\]
we have
\begin{equation}%
\begin{array}
[c]{ll}%
|\eta^{3,\epsilon}(0)| & =\left\vert \mathbb{E}\left[  h(T)D_{2}^{\epsilon
}(T)+\int_{0}^{T}\left(  m(t)A_{2}^{\epsilon}(t)+n(t)B_{2}^{\epsilon
}(t)+h(t)C_{2}^{\epsilon}(t)\right)  dt\right]  \right\vert \\
& \leq\mathbb{E}\left[  \left\vert h(T)D_{2}^{\epsilon}(T)\right\vert
+\int_{0}^{T}\left(  \left\vert m(t)A_{2}^{\epsilon}(t)\right\vert +\left\vert
n(t)B_{2}^{\epsilon}(t)\right\vert +\left\vert h(t)C_{2}^{\epsilon
}(t)\right\vert \right)  dt\right]  .
\end{array}
\label{second order-estimate}%
\end{equation}
We estimate each term as follows.

(1) {
\[%
\begin{array}
[c]{ll}%
\mathbb{E}\left[  |h(T)D_{2}^{\epsilon}(T)|\right]  & \leq\left\{
\mathbb{E}\left[  |h(T)|^{2}\right]  \right\}  ^{\frac{1}{2}}\left\{
\mathbb{E}\left[  |D_{2}^{\epsilon}(T)|^{2}\right]  \right\}  ^{\frac{1}{2}}\\
& \leq C\left\{  \mathbb{E}\left[  |\tilde{\phi}_{xx}^{\epsilon}(T)-\phi
_{xx}(\bar{X}(T))|^{2}|\xi^{1,\epsilon}(T)|^{4}+|\xi^{2,\epsilon}(T)|^{2}%
|\xi^{1,\epsilon}(T)+X_{1}(T)|^{2}\right]  \right\}  ^{\frac{1}{2}}\\
& =o(\epsilon).
\end{array}
\]
}

(2) {Since }%
\[
\mathbb{E}\left[  \int_{0}^{T}|m(t)A_{2}^{\epsilon}(t)|dt\right]
\leq\mathbb{E}\left[  \sup\limits_{t\in\lbrack0,T]}|m(t)|\int_{0}^{T}%
|A_{2}^{\epsilon}(t)|dt\right]  \leq\left\{  \mathbb{E}\left[  \sup
\limits_{t\in\lbrack0,T]}|m(t)|^{2}\right]  \right\}  ^{\frac{1}{2}}\left\{
\mathbb{E}\left[  \left(  \int_{0}^{T}|A_{2}^{\epsilon}(t)|dt\right)
^{2}\right]  \right\}  ^{\frac{1}{2}},
\]
{ then we only need to check
\begin{equation}
\mathbb{E}\left[  \left(  \int_{0}^{T}|A_{2}^{\epsilon}(t)|dt\right)
^{2}\right]  =o(\epsilon^{2}). \label{second order-part estimate}%
\end{equation}
Indeed, (\ref{second order-part estimate}) is due to the following estimates:
\[%
\begin{array}
[c]{l}%
\mathbb{E}\left[  \left(  \int_{0}^{T}|\delta b_{z}(t,\Delta)(\zeta
^{1,\epsilon}(t)-\Delta(t)I_{E_{\epsilon}}(t))|I_{E_{\epsilon}}(t)dt\right)
^{2}\right] \\
\leq\mathbb{E}\left[  \left(  \int_{E_{\epsilon}}|\delta b_{z}(t,\Delta
)|\left(  |\zeta^{2,\epsilon}(t)|+|K_{1}(t)X_{1}(t)|\right)  dt\right)
^{2}\right] \\
\leq C\mathbb{E}\left[  \left(  \int_{E_{\epsilon}}|\zeta^{2,\epsilon
}(t)|dt\right)  ^{2}\right]  +C\mathbb{E}\left[  \sup\limits_{t\in\lbrack
0,T]}|X_{1}(t)|^{2}\left(  \int_{E_{\epsilon}}|\delta b_{z}(t,\Delta
)|dt\right)  ^{2}\right] \\
\leq C\epsilon\mathbb{E}\left[  \int_{0}^{T}|\zeta^{2,\epsilon}(t)|^{2}%
dt\right]  +C\epsilon^{2}\mathbb{E}[\sup\limits_{t\in\lbrack0,T]}%
|X_{1}(t)|^{2}]\\
=o(\epsilon^{2}),
\end{array}
\]%
\begin{equation}%
\begin{array}
[c]{l}%
\mathbb{E}\left[  \left(  \int_{0}^{T}\left\vert \widetilde{b}_{zz}^{\epsilon
}(t)\left(  \zeta^{1,\epsilon}(t)-\Delta(t)I_{E_{\epsilon}}(t)\right)
^{2}-b_{zz}(t)K_{1}(t)^{2}X_{1}(t)^{2}\right\vert dt\right)  ^{2}\right] \\
\leq\mathbb{E}\left[  \left(  \int_{0}^{T}\left\vert \widetilde{b}%
_{zz}^{\epsilon}(t)\zeta^{2,\epsilon}(t)\left(  \zeta^{1,\epsilon}%
(t)-\Delta(t)I_{E_{\epsilon}}(t)+K_{1}(t)X_{1}(t)\right)  \right\vert
dt\right)  ^{2}\right] \\
\text{ \ }+\mathbb{E}\left[  \left(  \int_{0}^{T}\left\vert \left(
\widetilde{b}_{zz}^{\epsilon}(t)-b_{zz}(t)\right)  K_{1}(t)^{2}X_{1}%
(t)^{2}\right\vert dt\right)  ^{2}\right] \\
\leq C\mathbb{E}\left[  \int_{0}^{T}\left\vert \zeta^{2,\epsilon
}(t)\right\vert ^{2}dt\int_{0}^{T}\left\vert \zeta^{1,\epsilon}(t)-\Delta
(t)I_{E_{\epsilon}}(t)+K_{1}(t)X_{1}(t)\right\vert ^{2}dt\right] \\
\text{ \ }+C\mathbb{E}\left[  \sup\limits_{t\in\lbrack0,T]}|X_{1}%
(t)|^{4}\left(  \int_{0}^{T}\left\vert \left(  \widetilde{b}_{zz}^{\epsilon
}(t)-b_{zz}(t)\right)  \right\vert dt\right)  ^{2}\right] \\
=o(\epsilon^{2}),
\end{array}
\label{est-d2b}%
\end{equation}
the other terms are similar.}

(3){ }The estimate of $\mathbb{E}\left[  \int_{0}^{T}|n(t)B_{2}^{\epsilon
}(t)|dt\right]  $:
\[%
\begin{array}
[c]{ll}%
\mathbb{E}\left[  \int_{0}^{T}\left\vert n(t)\delta\sigma_{z}(t,\Delta
)\zeta^{2,\epsilon}(t)I_{E_{\epsilon}}(t)\right\vert dt\right]  & \leq
C\mathbb{E}\left[  \int_{E_{\epsilon}}|n(t)\zeta^{2,\epsilon}(t)|dt\right] \\
& \leq C\left\{  \mathbb{E}\left[  \int_{0}^{T}|\zeta^{2,\epsilon}%
(t)|^{2}dt\right]  \right\}  ^{\frac{1}{2}}\left\{  \mathbb{E}\left[
\int_{E_{\epsilon}}|n(t)|^{2}dt\right]  \right\}  ^{\frac{1}{2}}\\
& =o(\epsilon),
\end{array}
\]%
\[%
\begin{array}
[c]{l}%
\mathbb{E}\left[  \int_{0}^{T}\left\vert n(t)\right\vert \left\vert
\tilde{\sigma}_{zz}^{\epsilon}(t)\left(  \zeta^{1,\epsilon}(t)-\Delta
(t)I_{E_{\epsilon}}(t)\right)  ^{2}-\sigma_{zz}(t)K_{1}(t)^{2}X_{1}%
(t)^{2}\right\vert dt\right] \\
\leq\mathbb{E}\left[  \int_{0}^{T}\left\vert n(t)\right\vert \left\vert
\tilde{\sigma}_{zz}^{\epsilon}(t)\left(  \zeta^{1,\epsilon}(t)-\Delta
(t)I_{E_{\epsilon}}(t)+K_{1}(t)X_{1}(t)\right)  \zeta^{2,\epsilon
}(t)\right\vert dt\right] \\
\text{ \ }+\mathbb{E}\left[  \int_{0}^{T}\left\vert n(t)\right\vert \left\vert
\tilde{\sigma}_{zz}^{\epsilon}(t)-\sigma_{zz}(t)\right\vert K_{1}(t)^{2}%
X_{1}(t)^{2}dt\right] \\
\leq\mathbb{E}\left[  \int_{0}^{T}\left\vert n(t)\right\vert \left\vert
\tilde{\sigma}_{zz}^{\epsilon}(t)\left(  \zeta^{1,\epsilon}(t)-\Delta
(t)I_{E_{\epsilon}}(t)\right)  \zeta^{2,\epsilon}(t)\right\vert dt\right]
+\mathbb{E}\left[  \int_{0}^{T}\left\vert n(t)\right\vert \left\vert
\tilde{\sigma}_{zz}^{\epsilon}(t)K_{1}(t)X_{1}(t)\zeta^{2,\epsilon
}(t)\right\vert dt\right]  +o(\epsilon)\\
=\mathbb{E}\left[  \int_{0}^{T}\left\vert n(t)\right\vert \left\vert 2\int%
_{0}^{1}\theta\left[  \sigma_{z}(t,\Theta(t,\Delta I_{E_{\epsilon}}%
)+\theta(\Theta^{\epsilon}(t)-\Theta(t,\Delta I_{E_{\epsilon}})),u^{\epsilon
}(t))-\sigma_{z}(t,\Theta(t,\Delta I_{E_{\epsilon}}),u^{\epsilon}(t))\right]
d\theta\right\vert \left\vert \zeta^{2,\epsilon}(t)\right\vert dt\right] \\
\text{ \ }+\mathbb{E}\left[  \int_{0}^{T}\left\vert n(t)\right\vert \left\vert
\tilde{\sigma}_{zx}^{\epsilon}\left(  t\right)  \xi^{1,\epsilon}\left(
t\right)  +\tilde{\sigma}_{zy}^{\epsilon}\left(  t\right)  \eta^{1,\epsilon
}\left(  t\right)  \right\vert \left\vert \zeta^{2,\epsilon}(t)\right\vert
dt\right]  +C\mathbb{E}\left[  \sup\limits_{t\in\lbrack0,T]}\left\vert
X_{1}(t)\right\vert \int_{0}^{T}\left\vert n(t)K_{1}(t)\right\vert \left\vert
\zeta^{2,\epsilon}(t)\right\vert dt\right]  +o(\epsilon)\\
=o(\epsilon),
\end{array}
\]
{the other terms are similar.}

(4) {The estimate of $\mathbb{E}\left[  \int_{0}^{T}|h(t)C_{2}^{\epsilon
}(t)|dt\right]  $ is the same as $\mathbb{E}\left[  \int_{0}^{T}%
|m(t)A_{2}^{\epsilon}(t)|dt\right]  $. }

All the terms in (\ref{second order-estimate}) have been derived. Finally, we
obtain
\[
Y^{\epsilon}(0)-\bar{Y}(0)-Y_{1}(0)-Y_{2}(0)=o(\epsilon).
\]
The proof is complete.
\end{proof}

In the above lemma, we only prove $Y^{\epsilon}(0)-\bar{Y}(0)-Y_{1}%
(0)-Y_{2}(0)=o(\epsilon)$ and have not deduced
\[
\mathbb{E}[\sup\limits_{t\in\lbrack0,T]}|Y^{\epsilon}(t)-\bar{Y}%
(t)-Y_{1}(t)-Y_{2}(t)|^{2}]=o(\epsilon^{2}).
\]
The reason is
\[
\mathbb{E}\left[  \int_{0}^{T}\left\vert \tilde{\sigma}_{zz}^{\epsilon
}(t)\left(  \zeta^{1,\epsilon}(t)-\Delta(t)I_{E_{\epsilon}}(t)\right)
\right\vert ^{2}\left\vert \zeta^{2,\epsilon}(t)\right\vert ^{2}dt\right]
=o(\epsilon^{2})
\]
may be not hold. But if
\begin{equation}
\sigma(t,x,y,z,u)=A(t)z+\sigma_{1}(t,x,y,u) \label{xigma-zz}%
\end{equation}
where $A(t)$ is a bounded adapted process, then $\sigma_{zz}\equiv0.$ In this
case, we can prove the following estimate.

\begin{lemma}
\label{lemma-est-sup}Suppose Assumptions \ref{assum-2}, \ref{assum-3},
\ref{assm-q-bound} and $\sigma(t,x,y,z,u)=A(t)z+$ $\sigma_{1}(t,x,y,u)$ where
$A(t)$ is a bounded adapted process. Then%
\[%
\begin{array}
[c]{rl}%
\mathbb{E}\left[  \sup\limits_{t\in\lbrack0,T]}|X^{\epsilon}(t)-\bar
{X}(t)-X_{1}(t)-X_{2}(t)|^{2}\right]  & =o(\epsilon^{2}),\\
\mathbb{E}\left[  \sup\limits_{t\in\lbrack0,T]}|Y^{\epsilon}(t)-\bar
{Y}(t)-Y_{1}(t)-Y_{2}(t)|^{2}+\int_{0}^{T}|Z^{\epsilon}(t)-\bar{Z}%
(t)-Z_{1}(t)-Z_{2}(t)|^{2}dt\right]  & =o(\epsilon^{2}).
\end{array}
\]

\end{lemma}

\begin{proof}
We use all notations in Lemma \ref{est-second-order}. By Theorem
\ref{est-fbsde-lp}, we have%
\[%
\begin{array}
[c]{l}%
\mathbb{E}\left[  \sup\limits_{t\in\lbrack0,T]}(|\xi^{3,\epsilon}%
(t)|^{2}+|\eta^{3,\epsilon}(t)|^{2})+\int_{0}^{T}|\zeta^{3,\epsilon}%
(t)|^{2}dt\right] \\
\leq C\mathbb{E}\left[  \left(  \int_{0}^{T}|A_{2}^{\epsilon}(t)|dt\right)
^{2}+\left(  \int_{0}^{T}|C_{2}^{\epsilon}(t)|dt\right)  ^{2}+\int_{0}%
^{T}|B_{2}^{\epsilon}(t)|^{2}dt+|D_{2}^{\epsilon}(T)|^{2}\right]  ,
\end{array}
\]
where $A_{2}^{\epsilon}(\cdot)$, $C_{2}^{\epsilon}(\cdot)$, $D_{2}^{\epsilon
}(T)$ are the same as Lemma \ref{est-second-order}, and
\[%
\begin{array}
[c]{ll}%
B_{2}^{\epsilon}(t)= & \left[  \delta\sigma_{x}(t)\xi^{2,\epsilon}%
(t)+\delta\sigma_{y}(t)\eta^{2,\epsilon}(t)\right]  I_{E_{\epsilon}}%
(t)+\frac{1}{2}\left[  \xi^{1,\epsilon}(t),\eta^{1,\epsilon}(t)\right]
\widetilde{D^{2}\sigma^{\epsilon}}(t)\left[  \xi^{1,\epsilon}(t),\eta
^{1,\epsilon}(t)\right]  ^{\intercal}\\
& -\frac{1}{2}\left[  X_{1}(t),Y_{1}(t)\right]  D^{2}\sigma(t)\left[
X_{1}(t),Y_{1}(t)\right]  ^{\intercal}.
\end{array}
\]
In Lemma \ref{est-second-order}, we have proved
\[
\mathbb{E}\left[  \left(  \int_{0}^{T}|A_{2}^{\epsilon}(t)|dt\right)
^{2}+\left(  \int_{0}^{T}|C_{2}^{\epsilon}(t)|dt\right)  ^{2}+|D_{2}%
^{\epsilon}(T)|^{2}\right]  =o(\epsilon^{2}).
\]
Now we just need to check $\mathbb{E}\left[  \int_{0}^{T}|B_{2}^{\epsilon
}(t)|^{2}dt\right]  =o\left(  \epsilon^{2}\right)  $ as follows.%
\begin{equation}%
\begin{array}
[c]{l}%
\mathbb{E}\left[  \int_{0}^{T}\left\vert \delta\sigma_{x}(t)\xi^{2,\epsilon
}(t)\right\vert ^{2}I_{E_{\epsilon}}(t)dt\right]  \leq\mathbb{E}\left[
\sup\limits_{t\in\lbrack0,T]}\left\vert \xi^{2,\epsilon}(t)\right\vert
^{2}\int_{E_{\epsilon}}\left\vert \delta\sigma_{x}(t)\right\vert
^{2}dt\right]  =o(\epsilon^{2}).
\end{array}
\label{est-sigmax-x-x1}%
\end{equation}
The estimate of $\mathbb{E}\left[  \int_{0}^{T}\left\vert \delta\sigma
_{y}(t)\eta^{2,\epsilon}(t)\right\vert ^{2}dt\right]  $ is same to
\eqref{est-sigmax-x-x1}, and%
\[%
\begin{array}
[c]{l}%
\mathbb{E}\left[  \int_{0}^{T}\left\vert \tilde{\sigma}_{yy}^{\epsilon}%
(t)\eta^{1,\epsilon}(t)^{2}-\sigma_{yy}(t)Y_{1}(t)^{2}\right\vert
^{2}dt\right]  \ \ \\
\leq\mathbb{E}\left[  \int_{0}^{T}\left\vert \tilde{\sigma}_{yy}^{\epsilon
}(t)\eta^{2,\epsilon}(t)(\eta^{1,\epsilon}(t)+Y_{1}(t))\right\vert
^{2}dt\right]  +\mathbb{E}\left[  \int_{0}^{T}\left\vert \tilde{\sigma}%
_{yy}^{\epsilon}(t)-\sigma_{yy}(t)\right\vert ^{2}Y_{1}(t)^{4}dt\right] \\
\leq C\mathbb{E}\left[  \int_{0}^{T}\left\vert \eta^{2,\epsilon}(t)\right\vert
^{2}\left\vert \eta^{1,\epsilon}(t)+Y_{1}(t)\right\vert ^{2}dt\right]
+\mathbb{E}\left[  \sup\limits_{t\in\lbrack0,T]}\left\vert Y_{1}(t)\right\vert
^{4}\int_{0}^{T}\left\vert \tilde{\sigma}_{yy}^{\epsilon}(t)-\sigma
_{yy}(t)\right\vert ^{2}dt\right] \\
=o(\epsilon^{2}).
\end{array}
\]
Other terms are similar.
\end{proof}

\subsection{Maximum principle}

\label{section-mp}Note that $Y_{1}(0)=0$, by Lemma \ref{est-second-order}, we
have
\[
J(u^{\epsilon}(\cdot))-J(\bar{u}(\cdot))=Y^{\epsilon}(0)-\bar{Y}%
(0)=Y_{2}(0)+o(\epsilon).
\]
In order to obtain $Y_{2}(0)$, we introduce the following second-order adjoint
equation:
\begin{equation}
\left\{
\begin{array}
[c]{rl}%
-dP(t)= & \left\{  P(t)\left[  (D\sigma(t)^{\intercal}[1,p(t),K_{1}%
(t)]^{\intercal})^{2}+2Db(t)^{\intercal}[1,p(t),K_{1}(t)]^{\intercal}%
+H_{y}(t)\right]  \right. \\
& +2Q(t)D\sigma(t)^{\intercal}[1,p(t),K_{1}(t)]^{\intercal}+\left[
1,p(t),K_{1}(t)\right]  D^{2}H(t)\left[  1,p(t),K_{1}(t)\right]  ^{\intercal
}\left.  +H_{z}(t)K_{2}(t)\right\}  dt\\
& -Q(t)dB(t),\\
P(T)= & \phi_{xx}(\bar{X}(T)),
\end{array}
\right.  \label{eq-P}%
\end{equation}
where
\[%
\begin{array}
[c]{ll}%
H(t,x,y,z,u,p,q)= & g(t,x,y,z,u)+pb(t,x,y,z,u)+q\sigma(t,x,y,z,u),
\end{array}
\]%
\[%
\begin{array}
[c]{ll}%
K_{2}(t)= & (1-p(t)\sigma_{z}(t))^{-1}\left\{  p(t)\sigma_{y}(t)+2\left[
\sigma_{x}(t)+\sigma_{y}(t)p(t)+\sigma_{z}(t)K_{1}(t)\right]  \right\}  P(t)\\
& +(1-p(t)\sigma_{z}(t))^{-1}\left\{  Q(t)+p(t)[1,p(t),K_{1}(t)]D^{2}%
\sigma(t)[1,p(t),K_{1}(t)]^{\intercal}\right\}  ,
\end{array}
\]
and $DH(t)$, $D^{2}H(t)$ are defined similar to $D\psi$ and $D^{2}\psi$.

Note that (\ref{eq-P}) is a linear BSDE with uniformly Lipschitz continuous
coefficients, then it has a unique solution. Before we deduce the relationship
between $X_{2}(\cdot)$ and $(Y_{2}(\cdot),Z_{2}(\cdot))$, we introduce the
following equation:
\begin{equation}%
\begin{array}
[c]{l}%
\hat{Y}(t)=\int_{t}^{T}\left\{  (H_{y}(s)+\sigma_{y}(s)g_{z}%
(s)p(s)(1-p(s)\sigma_{z}(s))^{-1})\hat{Y}(s)+\left(  H_{z}(s)+\sigma
_{z}(s)g_{z}(s)p(s)(1-p(s)\sigma_{z}(s))^{-1}\right)  \hat{Z}(s)\right. \\
\ \ \ \ \ \ \ \ \ \ \left.  +\left[  \delta H(s,\Delta)+\frac{1}{2}%
P(s)\delta\sigma(s,\Delta)^{2}\right]  I_{E_{\epsilon}}(s)\right\}
ds-\int_{t}^{T}\hat{Z}(s)dB(s),
\end{array}
\label{eq-y-hat}%
\end{equation}
where $\delta H(s,\Delta):=p(s)\delta b(s,\Delta)+q(s)\delta\sigma
(s,\Delta)+\delta g(s,\Delta)$. It is also a linear BSDE and has a unique solution.

\begin{lemma}
\label{relation-y2} Suppose Assumptions \ref{assum-2}, \ref{assum-3} and.
\ref{assm-q-bound} hold. Then we have%
\[%
\begin{array}
[c]{rl}%
Y_{2}(t) & =p(t)X_{2}(t)+\frac{1}{2}P(t)X_{1}(t)^{2}+\hat{Y}(t),\\
Z_{2}(t) & =\mathbf{I(t)}+\hat{Z}(t),
\end{array}
\]
where $(\hat{Y}(\cdot),\hat{Z}(\cdot))$ is the solution to \eqref{eq-y-hat}
and
\begin{align*}
\mathbf{I(t)}  &  =K_{1}(t)X_{2}(t)+\frac{1}{2}K_{2}(t)X_{1}^{2}%
(t)+(1-p(t)\sigma_{z}(t))^{-1}p(t)(\sigma_{y}(t)\hat{Y}(t)+\sigma_{z}%
(t)\hat{Z}(t))+P(t)\delta\sigma(t,\Delta)X_{1}(t)I_{E_{\epsilon}}(t)\\
&  \;+(1-p(t)\sigma_{z}(t))^{-1}p(t)\left[  \delta\sigma_{x}(t,\Delta
)X_{1}(t)+\delta\sigma_{y}(t,\Delta)p(t)X_{1}(t)+\delta\sigma_{z}%
(t,\Delta)K_{1}(t)X_{1}(t)\right]  I_{E_{\epsilon}}(t).
\end{align*}

\end{lemma}

\begin{proof}
Using the same method as in Lemma \ref{lemma-y1}, we can deduce the above
relationship similarly.
\end{proof}

Consider the following equation:
\begin{equation}
\left\{
\begin{array}
[c]{rl}%
d\gamma(t)= & \gamma(t)\left[  H_{y}(t)+p(t)g_{z}(t)(1-p(t)\sigma_{z}%
(t))^{-1}\sigma_{y}(t)\right]  dt\\
& +\gamma(t)\left[  H_{z}(t)+p(t)(1-p(t)\sigma_{z}(t))^{-1}\sigma_{z}%
(t)g_{z}(t)\right]  dB(t),\\
\gamma(0)= & 1.
\end{array}
\right.  \label{eq-gamma}%
\end{equation}
Applying It\^{o}'s formula to $\gamma(t)\hat{Y}(t)$, we obtain
\[%
\begin{array}
[c]{rl}%
\hat{Y}(0)= & \mathbb{E}\left\{  \int_{0}^{T}\gamma(t)\left[  \delta
H(t,\Delta)+\frac{1}{2}P(t)\delta\sigma(t,\Delta)^{2}\right]  I_{E_{\epsilon}%
}(t)dt\right\}  .
\end{array}
\]
Define
\begin{equation}%
\begin{array}
[c]{ll}%
\mathcal{H}(t,x,y,z,u,p,q,P)= & pb(t,x,y,z+\Delta(t),u)+q\sigma(t,x,y,z+\Delta
(t),u)\\
& +\frac{1}{2}P(\sigma(t,x,y,z+\Delta(t),u)-\sigma(t,\bar{X}(t),\bar
{Y}(t),\bar{Z}(t),\bar{u}(t)))^{2}\ +g(t,x,y,z+\Delta(t),u),
\end{array}
\label{def-H}%
\end{equation}
where $\Delta(t)$ is defined in (\ref{def-delt}) corresponding to $u(t)=u$. It
is easy to check that
\begin{align*}
&  \delta H(t,\Delta)+\frac{1}{2}P(t)\delta\sigma(t,\Delta)^{2}\\
&  =\mathcal{H}(t,\bar{X}(t),\bar{Y}(t),\bar{Z}%
(t),u(t),p(t),q(t),P(t))-\mathcal{H}(t,\bar{X}(t),\bar{Y}(t),\bar{Z}%
(t),\bar{u}(t),p(t),q(t),P(t)).
\end{align*}
Noting that $\gamma(t)>0$ for $t\in\lbrack0,T]$, then we obtain the following
maximum principle.

\begin{theorem}
\label{Th-MP}Suppose Assumptions \ref{assum-2}, \ref{assum-3} and
\ref{assm-q-bound} hold. Let $\bar{u}(\cdot)\in\mathcal{U}[0,T]$ be optimal
and $(\bar{X}(\cdot),\bar{Y}(\cdot),\bar{Z}(\cdot))$ be the corresponding
state processes of (\ref{state-eq}). Then the following stochastic maximum
principle holds:
\begin{equation}
\mathcal{H}(t,\bar{X}(t),\bar{Y}(t),\bar{Z}(t),u,p(t),q(t),P(t))\geq
\mathcal{H}(t,\bar{X}(t),\bar{Y}(t),\bar{Z}(t),\bar{u}%
(t),p(t),q(t),P(t)),\ \ \ \forall u\in U,\ a.e.,\ a.s., \label{mp-1}%
\end{equation}
where $(p\left(  \cdot\right)  ,q\left(  \cdot\right)  )$, $\left(  P\left(
\cdot\right)  ,Q\left(  \cdot\right)  \right)  $ satisfy (\ref{eq-p}),
(\ref{eq-P}) respectively, and $\Delta(\cdot)$ satisfies (\ref{def-delt}).
\end{theorem}

\begin{remark}
If $b$ and $\sigma$ are independent of $y$ and $z$, then Theorem \ref{Th-MP}
degenerates to the maximum principle obtained in \cite{Hu17}.
\end{remark}

\begin{corollary}
\label{cor-mp-convex}Under the same assumptions as in Theorem \ref{Th-MP}.
Moreover, suppose that $b$, $\sigma$, $g$ are continuously differentiable with
respect to $u$ and $U$ is a convex set. Then
\begin{equation}
\Delta_{u}(t)|_{u=\bar{u}(t)}=\frac{p(t)\sigma_{u}(t)}{1-p(t)\sigma_{z}(t)}
\label{delt-convex}%
\end{equation}
and the maximum principle is
\begin{equation}
\mathcal{H}_{u}(t,\bar{X}(t),\bar{Y}(t),\bar{Z}(t),\bar{u}(t),p(t),q(t))\cdot
(u-\bar{u}(t))\geq0\ \ \ \forall u\in U,\ a.e.,\ a.s. \label{hamil-convex}%
\end{equation}
with%
\[%
\begin{array}
[c]{l}%
\mathcal{H}_{u}(t,\bar{X}(t),\bar{Y}(t),\bar{Z}(t),\bar{u}(t),p(t),q(t))\\
=p(t)b_{u}(t)+q(t)\sigma_{u}(t)+g_{u}(t)+(p(t)b_{z}(t)+q(t)\sigma_{z}%
(t)+g_{z}(t))\frac{p(t)\sigma_{u}(t)}{1-p(t)\sigma_{z}(t)}.
\end{array}
\]

\end{corollary}

\begin{proof}
By implicit function theorem for (\ref{def-delt}), we get (\ref{delt-convex}).
For each $u\in U$, taking $u_{\rho}(t)=\bar{u}(t)+\rho(u-\bar{u}(t))$, we can
get (\ref{hamil-convex}) by (\ref{mp-1}).
\end{proof}

\subsection{The case without Assumption \ref{assm-q-bound}}

The relations $Y_{1}(t)=p(t)X_{1}(t)$ and $\ Z_{1}(t)=K_{1}(t)X_{1}%
(t)+\Delta(t)I_{E_{\epsilon}}(t)$ in Lemma \ref{lemma-y1}, is the key point to
derive the maximum principle (\ref{mp-1}). Note that to prove Lemma
\ref{lemma-y1}, we need Assumption \ref{assm-q-bound}, which implies%
\begin{equation}
\mathbb{E}\left[  \sup\limits_{t\in\lbrack0,T]}|\tilde{X}_{1}(t)|^{2}\right]
<\infty. \label{eq-new211}%
\end{equation}
However, under the following assumption, combing Theorems
\ref{appen-th-linear-fbsde} and \ref{unique-pq} in appendix, we can obtain the
relations $Y_{1}(t)=p(t)X_{1}(t)$ and$\ Z_{1}(t)=K_{1}(t)X_{1}(t)+\Delta
(t)I_{E_{\epsilon}}(t)$ without Assumption \ref{assm-q-bound}.

\begin{assumption}
\label{assm-sig-small} $\sigma(t,x,y,z,u)=A(t)z+\sigma_{1}(t,x,y,u)$ and
$\left\Vert A(\cdot)\right\Vert _{\infty}$ is small enough.
\end{assumption}

In this case, the first-order adjoint equation becomes
\begin{equation}
\left\{
\begin{array}
[c]{rl}%
dp(t)= & -\left\{  g_{x}(t)+g_{y}(t)p(t)+g_{z}(t)K_{1}(t)+b_{x}(t)p(t)+b_{y}%
(t)p^{2}(t)+b_{z}(t)K_{1}(t)p(t)\right. \\
& \left.  +\sigma_{x}(t)q(t)+\sigma_{y}(t)p(t)q(t)+A(t)K_{1}(t)q(t)\right\}
dt+q(t)dB(t),\\
p(T)= & \phi_{x}(\bar{X}(T)),
\end{array}
\right.  \label{eq-p-q-unbound}%
\end{equation}
where
\[
K_{1}(t)=(1-p(t)A(t))^{-1}\left[  \sigma_{x}(t)p(t)+\sigma_{y}(t)p^{2}%
(t)+q(t)\right]  .
\]
The first-order variational equation becomes
\[
\left\{
\begin{array}
[c]{rl}%
dX_{1}(t)= & \left[  b_{x}(t)X_{1}(t)+b_{y}(t)Y_{1}(t)+b_{z}(t)(Z_{1}%
(t)-\Delta(t)I_{E_{\epsilon}}(t))\right]  dt\\
& +\left[  \sigma_{x}(t)X_{1}(t)+\sigma_{y}(t)Y_{1}(t)+A(t)(Z_{1}%
(t)-\Delta(t)I_{E_{\epsilon}}(t))+\delta\sigma(t,\Delta)I_{E_{\epsilon}%
}(t)\right]  dB(t),\\
X_{1}(0)= & 0,
\end{array}
\right.
\]
and
\[
\left\{
\begin{array}
[c]{lll}%
dY_{1}(t) & = & -\left[  g_{x}(t)X_{1}(t)+g_{y}(t)Y_{1}(t)+g_{z}%
(t)(Z_{1}(t)-\Delta(t)I_{E_{\epsilon}}(t))-q(t)\delta\sigma(t,\Delta
)I_{E_{\epsilon}}(t)\right]  dt+Z_{1}(t)dB(t),\\
Y_{1}(T) & = & \phi_{x}(\bar{X}(T))X_{1}(T),
\end{array}
\right.
\]
where
\[
\Delta(t)=\left(  1-p(t)A(t)\right)  ^{-1}p(t)\left(  \sigma_{1}(t,\bar
{X}(t),\bar{Y}(t),u(t))-\sigma_{1}(t,\bar{X}(t),\bar{Y}(t),\bar{u}(t))\right)
.
\]
By Theorems \ref{appen-th-linear-fbsde} and \ref{unique-pq} we have the
following relationship:

\begin{lemma}
Suppose Assumptions \ref{assum-2}(i)-(ii), \ref{assum-3} and
\ref{assm-sig-small} hold. Then we have
\begin{align*}
Y_{1}(t)  &  =p(t)X_{1}(t),\\
Z_{1}(t)  &  =K_{1}(t)X_{1}(t)+\Delta(t)I_{E_{\epsilon}}(t),
\end{align*}
where $p(\cdot)$ is the solution of \eqref{eq-p-q-unbound}.
\end{lemma}

\begin{lemma}
\label{est-one-order-q-unbound}Suppose Assumptions \ref{assum-2},
\ref{assum-3} and \ref{assm-sig-small} hold. Then for any $2\leq\beta<8$, we
have the following estimates
\begin{equation}
\mathbb{E}\left[  \sup\limits_{t\in\lbrack0,T]}\left(  |X_{1}(t)|^{\beta
}+|Y_{1}(t)|^{\beta}\right)  \right]  +\mathbb{E}\left[  \left(  \int_{0}%
^{T}|Z_{1}(t)|^{2}dt\right)  ^{\beta/2}\right]  =O(\epsilon^{\beta/2}),
\label{est-x1-y1-q-unbound}%
\end{equation}%
\[
\mathbb{E}\left[  \sup\limits_{t\in\lbrack0,T]}\left(  |X^{\epsilon}%
(t)-\bar{X}(t)-X_{1}(t)|^{4}+|Y^{\epsilon}(t)-\bar{Y}(t)-Y_{1}(t)|^{4}\right)
\right]  +\mathbb{E}\left[  \left(  \int_{0}^{T}|Z^{\epsilon}(t)-\bar
{Z}(t)-Z_{1}(t)|^{2}dt\right)  ^{2}\right]  =o(\epsilon^{2}).
\]

\end{lemma}

\begin{proof}
The estimate of (\ref{est-x1-y1-q-unbound}) is the same as (\ref{est-1-order})
in Lemma \ref{est-one-order}, except the following term,
\[%
\begin{array}
[c]{l}%
\mathbb{E}\left[  \left(  \int_{0}^{T}|q(t)\delta\sigma(t,\Delta
)|I_{E_{\epsilon}}(t)dt\right)  ^{\beta}\right] \\
\leq C\mathbb{E}\left[  \left(  \int_{E_{\epsilon}}|q(t)|\left(  1+|\bar
{X}(t)|+\left\vert \bar{Y}(t)\right\vert +\left\vert \bar{u}%
(t)|+|u(t)\right\vert \right)  dt\right)  ^{\beta}\right] \\
\leq C\mathbb{E}\left[  \left(  \int_{E_{\epsilon}}|q(t)|^{2}dt\right)
^{\frac{\beta}{2}}\left(  \int_{E_{\epsilon}}\left(  1+|\bar{X}(t)|^{2}%
+\left\vert \bar{Y}(t)\right\vert ^{2}+\left\vert \bar{u}(t)|^{2}%
+|u(t)\right\vert ^{2}\right)  dt\right)  ^{\frac{\beta}{2}}\right] \\
\leq C\left\{  \mathbb{E}\left[  \left(  \int_{E_{\epsilon}}\left(  1+|\bar
{X}(t)|^{2}+\left\vert \bar{Y}(t)\right\vert ^{2}+\left\vert \bar{u}%
(t)|^{2}+|u(t)\right\vert ^{2}\right)  dt\right)  ^{4}\right]  \right\}
^{\frac{\beta}{8}}\left\{  \mathbb{E}\left[  \left(  \int_{E_{\epsilon}%
}|q(t)|^{2}dt\right)  ^{\frac{4\beta}{8-\beta}}\right]  \right\}
^{\frac{8-\beta}{8}}\\
\leq C\left\{  \epsilon^{3}\mathbb{E}\left[  \int_{E_{\epsilon}}\left(
1+|\bar{X}(t)|^{8}+\left\vert \bar{Y}(t)\right\vert ^{8}+\left\vert \bar
{u}(t)|^{8}+|u(t)\right\vert ^{8}\right)  dt\right]  \right\}  ^{\frac{\beta
}{8}}\\
\leq C\epsilon^{\frac{\beta}{2}}.
\end{array}
\]
In this case, $A_{1}^{\epsilon}(\cdot)$, $C_{1}^{\epsilon}(\cdot)$,
$D_{1}^{\epsilon}(T)$ is the same as Lemma \ref{est-one-order}, and
\[
B_{1}^{\epsilon}(t)=(\tilde{\sigma}_{x}^{\epsilon}(t)-\sigma_{x}%
(t))X_{1}(t)+(\tilde{\sigma}_{y}^{\epsilon}(t)-\sigma_{y}(t))Y_{1}(t).
\]
By Theorem \ref{est-fbsde-lp}, we obtain
\[%
\begin{array}
[c]{l}%
\mathbb{E}\left[  \sup\limits_{t\in\lbrack0,T]}\left(  |\xi^{2,\epsilon
}(t)|^{4}+|\eta^{2,\epsilon}(t)|^{4}\right)  +\left(  \int_{0}^{T}%
|\zeta^{2,\epsilon}(t)|^{2}dt\right)  ^{2}\right] \\
\leq C\mathbb{E}\left[  \left(  \int_{0}^{T}\left(  |A_{1}^{\epsilon
}(t)|+|C_{1}^{\epsilon}(t)|\right)  dt\right)  ^{4}+\left(  \int_{0}^{T}%
|B_{1}^{\epsilon}(t)|^{2}dt\right)  ^{2}+|D_{1}^{\epsilon}(T)|^{4}\right] \\
\leq C\mathbb{E}\left[  |D_{1}^{\epsilon}(T)|^{4}+\left(  \int_{0}^{T}%
|C_{1}^{\epsilon}(t)|dt\right)  ^{4}+\left(  \int_{0}^{T}|B_{1}^{\epsilon
}(t)|^{2}dt\right)  ^{2}+\left(  \int_{0}^{T}|A_{1}^{\epsilon}(t)|dt\right)
^{4}\right]  .
\end{array}
\]
We estimate term by term as follows.

(1)
\[
\mathbb{E}\left[  |D_{1}^{\epsilon}(T)|^{4}\right]  \leq C\left\{
\mathbb{E}\left[  \sup\limits_{0\leq t\leq T}|X_{1}(t)|^{6}\right]  \right\}
^{\frac{2}{3}}\left\{  \mathbb{E}\left[  \left\vert \tilde{\phi}_{x}%
^{\epsilon}(T)-\phi_{x}\left(  \bar{X}\left(  T\right)  \right)  \right\vert
^{12}\right]  \right\}  ^{\frac{1}{3}}=o\left(  \epsilon^{2}\right)  .
\]

(2)
\begin{equation}%
\begin{array}
[c]{l}%
\mathbb{E}\left[  \left(  \int_{0}^{T}|\tilde{g}_{z}^{\epsilon}(t)-g_{z}%
(t)||Z_{1}(t)|dt\right)  ^{4}\right] \\
\leq C\mathbb{E}\left[  \left(  \int_{0}^{T}|\tilde{g}_{z}^{\epsilon}%
(t)-g_{z}(t)|^{2}dt\right)  ^{2}\left(  \int_{0}^{T}|Z_{1}(t)|^{2}dt\right)
^{2}\right] \\
\leq C\left\{  \mathbb{E}\left[  \left(  \int_{0}^{T}|\tilde{g}_{z}^{\epsilon
}(t)-g_{z}(t)|^{2}dt\right)  ^{6}\right]  \right\}  ^{\frac{1}{3}}\left\{
\mathbb{E}\left[  \left(  \int_{0}^{T}|Z_{1}(t)|^{2}dt\right)  ^{3}\right]
\right\}  ^{\frac{2}{3}}\\
=o\left(  \epsilon^{2}\right)  ,
\end{array}
\label{est-g-z-q-unbound}%
\end{equation}
the estimates of $\mathbb{E}\left[  \left(  \int_{0}^{T}\left\vert \left(
\tilde{g}_{x}^{\epsilon}(t)-g_{x}(t)\right)  X_{1}(t)\right\vert dt\right)
^{4}\right]  $ and $\mathbb{E}\left[  \left(  \int_{0}^{T}\left\vert \left(
\tilde{g}_{y}^{\epsilon}(t)-g_{y}(t)\right)  Y_{1}(t)\right\vert dt\right)
^{4}\right]  $ are the same as \eqref{est-g-z-q-unbound},
\begin{equation}
\mathbb{E}\left[  \left(  \int_{0}^{T}|g_{z}(t)\Delta(t)I_{E_{\epsilon}%
}(t)|dt\right)  ^{4}\right]  \leq C\epsilon^{3}\int_{E_{\epsilon}}%
\mathbb{E}[|\Delta(t)|^{4}]dt=o\left(  \epsilon^{2}\right)  ,
\label{est-g-del-qunbound}%
\end{equation}
the estimate of $\mathbb{E}\left[  \left(  \int_{0}^{T}\left\vert \delta
g\left(  t\right)  I_{E_{\epsilon}}(t)\right\vert dt\right)  ^{4}\right]  $ is
the same as \eqref{est-g-del-qunbound},%
\[%
\begin{array}
[c]{l}%
\mathbb{E}\left[  (\int_{0}^{T}|q(t)\delta\sigma(t,\Delta)I_{E_{\epsilon}%
}(t)|dt)^{4}\right] \\
\leq C\left\{  \mathbb{E}\left[  \left(  \int_{E_{\epsilon}}|q(t)|^{2}%
dt\right)  ^{4}\right]  \right\}  ^{\frac{1}{2}}\left\{  \mathbb{E}\left[
\left(  \int_{E_{\epsilon}}|\delta\sigma(t,\Delta)|^{2}dt\right)  ^{4}\right]
\right\}  ^{\frac{1}{2}}\\
\leq C\left\{  \mathbb{E}\left[  \left(  \int_{E_{\epsilon}}|q(t)|^{2}%
dt\right)  ^{4}\right]  \right\}  ^{\frac{1}{2}}\left\{  \mathbb{E}\left[
\left(  \int_{E_{\epsilon}}\left(  1+|\bar{X}(t)|^{2}+\left\vert \bar
{Y}(t)\right\vert ^{2}+\left\vert \bar{u}(t)|^{2}+|u(t)\right\vert
^{2}\right)  dt\right)  ^{4}\right]  \right\}  ^{\frac{1}{2}}\\
\leq C\epsilon^{2}\left\{  \mathbb{E}\left[  \left(  \int_{E_{\epsilon}%
}|q(t)|^{2}dt\right)  ^{4}\right]  \right\}  ^{\frac{1}{2}}\\
=o\left(  \epsilon^{2}\right)  .
\end{array}
\]
Then,
\[
\mathbb{E}\left[  \left(  \int_{0}^{T}|C_{1}^{\epsilon}(t)|dt\right)
^{4}\right]  =o\left(  \epsilon^{2}\right)  .
\]

(3)
\[%
\begin{array}
[c]{l}%
\mathbb{E}\left[  \left(  \int_{0}^{T}|\tilde{\sigma}_{y}^{\epsilon}%
(t)-\sigma_{y}(t)|^{2}|Y_{1}(t)|^{2}dt\right)  ^{2}\right] \\
\leq C\mathbb{E}\left[  \sup\limits_{0\leq t\leq T}|Y_{1}(t)|^{4}\left(
\int_{0}^{T}|\tilde{\sigma}_{z}^{\epsilon}(t)-\sigma_{z}(t)|^{2}dt\right)
^{2}\right] \\
\leq C\left\{  \mathbb{E}\left[  \sup\limits_{0\leq t\leq T}|Y_{1}%
(t)|^{6}\right]  \right\}  ^{\frac{2}{3}}\left\{  \mathbb{E}\left[  \left(
\int_{0}^{T}|\tilde{\sigma}_{z}^{\epsilon}(t)-\sigma_{z}(t)|^{2}dt\right)
^{6}\right]  \right\}  ^{\frac{1}{3}}\\
=o\left(  \epsilon^{2}\right)  ,
\end{array}
\]
the estimate of $\mathbb{E}\left[  \left(  \int_{0}^{T}|\tilde{\sigma}%
_{x}^{\epsilon}(t)-\sigma_{x}(t)|^{2}|X_{1}(t)|^{2}dt\right)  ^{2}\right]  $
is similar. Thus,
\[
\mathbb{E}\left[  \left(  \int_{0}^{T}|B_{1}^{\epsilon}(t)|^{2}dt\right)
^{2}\right]  =o\left(  \epsilon^{2}\right)  .
\]

(4) The estimate of $\mathbb{E}\left[  \left(  \int_{0}^{T}|A_{1}^{\epsilon
}(t)|dt\right)  ^{4}\right]  $ is the same as $\mathbb{E}\left[  \left(
\int_{0}^{T}|C_{1}^{\epsilon}(t)|dt\right)  ^{4}\right]  $.
\end{proof}

The second-order variational equation becomes
\begin{equation}
\left\{
\begin{array}
[c]{rl}%
dX_{2}(t)= & \left\{  b_{x}(t)X_{2}(t)+b_{y}(t)Y_{2}(t)+b_{z}(t)Z_{2}%
(t)+\delta b(t,\Delta)I_{E_{\epsilon}}(t)\right. \\
& \left.  +\frac{1}{2}\left[  1,p(t),K_{1}(t)\right]  D^{2}b(t)\left[
1,p(t),K_{1}(t)\right]  ^{\intercal}X_{1}(t)^{2}\right\}  dt\\
& +\left\{  \sigma_{x}(t)X_{2}(t)+\sigma_{y}(t)Y_{2}(t)+A(t)Z_{2}(t)+\left[
\delta\sigma_{x}(t)X_{1}(t)+\delta\sigma_{y}(t)Y_{1}(t)\right]  I_{E_{\epsilon
}}(t)\right. \\
& \left.  +\frac{1}{2}\left[  X_{1}(t),Y_{1}(t)\right]  D^{2}\sigma
_{1}(t)\left[  X_{1}(t),Y_{1}(t)\right]  ^{\intercal}\right\}  dB(t),\\
X_{2}(0)= & 0,
\end{array}
\right.  \label{new-form-x2-q-unbound}%
\end{equation}%
\begin{equation}
\left\{
\begin{array}
[c]{ll}%
dY_{2}(t)= & -\left\{  g_{x}(t)X_{2}(t)+g_{y}(t)Y_{2}(t)+g_{z}(t)Z_{2}%
(t)+\frac{1}{2}\left[  1,p(t),K_{1}(t)\right]  D^{2}g(t)\left[  1,p(t),K_{1}%
(t)\right]  ^{\intercal}X_{1}^{2}(t)\right. \\
& \left.  +q(t)\delta\sigma(t,\Delta)I_{E_{\epsilon}}(t)+\delta g(t,\Delta
)I_{E_{\epsilon}}(t)\right\}  dt+Z_{2}(t)dB(t),\\
Y_{2}(T)= & \phi_{x}(\bar{X}(T))X_{2}(T)+\frac{1}{2}\phi_{xx}(\bar{X}%
(T))X_{1}^{2}(T).
\end{array}
\right.  \label{new-form-y2-q-unbound}%
\end{equation}
The following second-order estimates hold.

\begin{lemma}
\label{est-second-order-q-unbound} Suppose Assumptions \ref{assum-2},
\ref{assum-3} and \ref{assm-sig-small} hold. Then we have the following
estimates%
\[%
\begin{array}
[c]{rl}%
\mathbb{E}\left[  \sup\limits_{t\in\lbrack0,T]}|X^{\epsilon}(t)-\bar
{X}(t)-X_{1}(t)-X_{2}(t)|^{2}\right]  & =o(\epsilon^{2}),\\
\mathbb{E}\left[  \sup\limits_{t\in\lbrack0,T]}|Y^{\epsilon}(t)-\bar
{Y}(t)-Y_{1}(t)-Y_{2}(t)|^{2}\right]  +\mathbb{E}\left[  \int_{0}%
^{T}|Z^{\epsilon}(t)-\bar{Z}(t)-Z_{1}(t)-Z_{2}(t)|^{2}dt\right]  &
=o(\epsilon^{2}).
\end{array}
\]

\end{lemma}

\begin{proof}
We use the same notations $A_{2}^{\epsilon}(t)$ $C_{2}^{\epsilon}(t)$ and
$D_{2}^{\epsilon}(T)$ as in\ Lemma \ref{est-second-order}. The only different
term is
\[%
\begin{array}
[c]{ll}%
B_{2}^{\epsilon}(t)= & \delta\sigma_{x}(t)\xi^{2,\epsilon}(t)I_{E_{\epsilon}%
}(t)+\delta\sigma_{y}(t)\eta^{2,\epsilon}(t)I_{E_{\epsilon}}(t)+\frac{1}%
{2}\left[  \xi^{1,\epsilon}(t),\eta^{1,\epsilon}(t)\right]  \widetilde{D^{2}%
\sigma^{\epsilon}}(t)\left[  \xi^{1,\epsilon}(t),\eta^{1,\epsilon}(t)\right]
^{\intercal}\\
& -\frac{1}{2}\left[  X_{1}(t),Y_{1}(t)\right]  D^{2}\sigma(t)\left[
X_{1}(t),Y_{1}(t)\right]  ^{\intercal}.
\end{array}
\]
Then, we have that
\begin{equation}
\left\{
\begin{array}
[c]{l}%
d\xi^{3,\epsilon}(t)=\left[  b_{x}(t)\xi^{3,\epsilon}(t)+b_{y}(t)\eta
^{3,\epsilon}(t)+b_{z}(t)\zeta^{3,\epsilon}(t)+A_{2}^{\epsilon}(t)\right]
dt\ \\
\text{ \ \ \ \ \ \ }+\left[  \sigma_{x}(t)\xi^{3,\epsilon}(t)+\sigma
_{y}(t)\eta^{3,\epsilon}(t)+A(t)\zeta^{3,\epsilon}(t)+B_{2}^{\epsilon
}(t)\right]  dB(t),\\
\xi^{3,\epsilon}(0)=0,
\end{array}
\right.  \label{x-x1-x2-bz-0}%
\end{equation}
and%
\begin{equation}
\left\{
\begin{array}
[c]{ll}%
d\eta^{3,\epsilon}(t)= & -\left[  g_{x}(t)\xi^{3,\epsilon}(t)+g_{y}%
(t)\eta^{3,\epsilon}(t)+g_{z}(t)\zeta^{3,\epsilon}(t)+C_{2}^{\epsilon
}(t)\right]  dt-\zeta^{3,\epsilon}(t)dB(t),\\
\eta^{3,\epsilon}(T)= & \phi_{x}(\bar{X}(T))\xi^{3,\epsilon}(T)+D_{2}%
^{\epsilon}(T).
\end{array}
\right.  \label{y-y1-y2-bz-0}%
\end{equation}
By Theorem \ref{est-fbsde-lp},%
\[%
\begin{array}
[c]{l}%
\mathbb{E}\left[  \sup\limits_{t\in\lbrack0,T]}\left(  |\xi^{3,\epsilon
}(t)|^{2}+|\eta^{3,\epsilon}(t)|^{2}\right)  +\int_{0}^{T}|\zeta^{3,\epsilon
}(t)|^{2}dt\right] \\
\leq\mathbb{E}\left[  \left(  \int_{0}^{T}|A_{2}^{\epsilon}(t)|dt\right)
^{2}+\left(  \int_{0}^{T}|C_{2}^{\epsilon}(t)|dt\right)  ^{2}+\int_{0}%
^{T}|B_{2}^{\epsilon}(t)|^{2}dt+|D_{2}^{\epsilon}(T)|^{2}\right]  .
\end{array}
\]
We estimate term by term in the followings.

(1)
\[%
\begin{array}
[c]{l}%
\mathbb{E}\left[  \left(  \int_{0}^{T}(\delta b_{z}(t,\Delta)(\zeta
^{1,\epsilon}(t)-\Delta(t)I_{E_{\epsilon}}(t))dt\right)  ^{2}\right] \\
\leq\mathbb{E}\left[  \left(  \int_{E_{\epsilon}}|\delta b_{z}(t,\Delta
)|\left(  |\zeta^{2,\epsilon}(t)|+|K_{1}(t)X_{1}(t)|\right)  dt\right)
^{2}\right] \\
\leq C\epsilon\mathbb{E}\left[  \int_{E_{\epsilon}}|\zeta^{2,\epsilon}%
(t)|^{2}dt\right]  +C\epsilon\mathbb{E}\left[  \sup\limits_{t\in\lbrack
0,T]}|X_{1}(t)|^{2}\int_{E_{\epsilon}}|K_{1}(t)|^{2}dt\right] \\
\leq C\epsilon\left\{  \mathbb{E}\left[  \left(  \int_{0}^{T}|\zeta
^{2,\epsilon}(t)|^{2}dt\right)  ^{2}\right]  \right\}  ^{\frac{1}{2}%
}+C\epsilon^{2}\left\{  \mathbb{E}\left[  \left(  \int_{E_{\epsilon}}%
|K_{1}(t)|^{2}dt\right)  ^{2}\right]  \right\}  ^{\frac{1}{2}}\\
=o(\epsilon^{2}),
\end{array}
\]%
\[%
\begin{array}
[c]{l}%
\mathbb{E}\left[  \left(  \int_{0}^{T}\left(  \tilde{b}_{zz}(t)(\zeta
^{1,\epsilon}(t)-\Delta(t)I_{E_{\epsilon}}(t))^{2}-b_{zz}(t)K_{1}(t)^{2}%
X_{1}(t)^{2}\right)  dt\right)  ^{2}\right] \\
\leq\mathbb{E}\left[  \left(  \int_{0}^{T}\tilde{b}_{zz}(t)\left(
(\zeta^{1,\epsilon}(t)-\Delta(t)I_{E_{\epsilon}}(t))^{2}-K_{1}(t)^{2}%
X_{1}(t)^{2}\right)  dt\right)  ^{2}+\left(  \int_{0}^{T}\left\vert \tilde
{b}_{zz}(t)-b_{zz}(t)\right\vert \left\vert K_{1}(t)X_{1}(t)\right\vert
^{2}dt\right)  ^{2}\right] \\
\leq C\mathbb{E}\left[  \left(  \int_{0}^{T}\left\vert \zeta^{2,\epsilon
}(t)\right\vert ^{2}dt\right)  ^{2}+\left(  \int_{0}^{T}\left\vert
\zeta^{2,\epsilon}(t)\right\vert \left\vert K_{1}(t)X_{1}(t)\right\vert
dt\right)  ^{2}\right] \\
\text{ \ }+\left\{  \mathbb{E}\left[  \sup\limits_{t\in\lbrack0,T]}\left\vert
X_{1}(t)\right\vert ^{6}\right]  \right\}  ^{\frac{2}{3}}\left\{
\mathbb{E}\left[  \left(  \int_{0}^{T}\left\vert \tilde{b}_{zz}(t)-b_{zz}%
(t)\right\vert \left\vert K_{1}(t)\right\vert ^{2}dt\right)  ^{6}\right]
\right\}  ^{\frac{1}{3}}\\
\leq C\left\{  \mathbb{E}\left[  \left(  \int_{0}^{T}|\zeta^{2,\epsilon
}(t)|^{2}dt\right)  ^{2}\right]  \right\}  ^{\frac{1}{2}}\left\{
\mathbb{E}\left[  \sup\limits_{t\in\lbrack0,T]}\left\vert X_{1}(t)\right\vert
^{6}\right]  \right\}  ^{\frac{1}{3}}\left\{  \mathbb{E}\left[  \left(
\int_{0}^{T}\left\vert K_{1}(t)\right\vert ^{2}dt\right)  ^{6}\right]
\right\}  ^{\frac{1}{6}}+o(\epsilon^{2})\\
=o(\epsilon^{2}),
\end{array}
\]
the other terms are similar. Then
\[
\mathbb{E}\left[  \left(  \int_{0}^{T}|A_{2}^{\epsilon}(t)|dt\right)
^{2}\right]  =o(\epsilon^{2}).
\]

(2) The estimate of $C_{2}^{\epsilon}(t)$ is the same as $A_{2}^{\epsilon}(t)$.

(3)
\[%
\begin{array}
[c]{l}%
\mathbb{E}\left[  \int_{0}^{T}\left\vert \delta\sigma_{x}(t)\xi^{2,\epsilon
}(t)\right\vert ^{2}I_{E_{\epsilon}}(t)dt\right]  \leq C\epsilon\left\{
\mathbb{E}\left[  \sup\limits_{t\in\lbrack0,T]}\left\vert \xi^{2,\epsilon
}(t)\right\vert ^{4}\right]  \right\}  ^{\frac{1}{2}}=o(\epsilon^{2}),
\end{array}
\]%
\[%
\begin{array}
[c]{l}%
\mathbb{E}\left[  \int_{0}^{T}\left\vert \tilde{\sigma}_{yy}^{\epsilon}%
(t)\eta^{1,\epsilon}(t)^{2}-\sigma_{yy}(t)Y_{1}(t)^{2}\right\vert
^{2}dt\right]  \ \ \\
\leq\mathbb{E}\left[  \int_{0}^{T}\left\vert \tilde{\sigma}_{yy}^{\epsilon
}(t)\eta^{2,\epsilon}(t)(\eta^{1,\epsilon}(t)+Y_{1}(t))\right\vert
^{2}dt\right]  +\mathbb{E}\left[  \int_{0}^{T}\left\vert \tilde{\sigma}%
_{yy}^{\epsilon}(t)-\sigma_{yy}(t)\right\vert ^{2}Y_{1}(t)^{4}dt\right] \\
\leq C\mathbb{E}\left[  \int_{0}^{T}\left\vert \eta^{2,\epsilon}(t)\right\vert
^{2}\left\vert \eta^{1,\epsilon}(t)+Y_{1}(t)\right\vert ^{2}dt\right]
+\mathbb{E}\left[  \sup\limits_{t\in\lbrack0,T]}\left\vert Y_{1}(t)\right\vert
^{4}\int_{0}^{T}\left\vert \tilde{\sigma}_{yy}^{\epsilon}(t)-\sigma
_{yy}(t)\right\vert ^{2}dt\right] \\
\leq C\left\{  \mathbb{E}\left[  \sup\limits_{t\in\lbrack0,T]}\left\vert
\eta^{2,\epsilon}(t)\right\vert ^{4}\right]  \right\}  ^{\frac{1}{2}}\left\{
\mathbb{E}\left[  \sup\limits_{t\in\lbrack0,T]}\left\vert \eta^{1,\epsilon
}(t)+Y_{1}(t)\right\vert ^{4}\right]  \right\}  ^{\frac{1}{2}}+o(\epsilon
^{2})\\
=o(\epsilon^{2}),
\end{array}
\]
the other terms are similar. Thus
\[
\mathbb{E}\left[  \int_{0}^{T}|B_{2}^{\epsilon}(t)|^{2}dt\right]
=o(\epsilon^{2}).
\]

(4)
\[
\mathbb{E}\left[  |D_{2}^{\epsilon}(T)|^{2}\right]  \leq C\mathbb{E}\left[
|\tilde{\phi}_{xx}^{\epsilon}(T)-\phi_{xx}(\bar{X}(T))|^{2}|\xi^{1,\epsilon
}(T)|^{4}+|\xi^{2,\epsilon}(T)|^{2}|\xi^{1,\epsilon}(T)+X_{1}(T)|^{2}\right]
=o(\epsilon^{2}).
\]
Thus,%
\[
\mathbb{E}\left\{  \sup\limits_{t\in\lbrack0,T]}[|\xi^{3,\epsilon}%
(t)|^{2}+|\eta^{3,\epsilon}(t)|^{2}]+\int_{0}^{T}|\zeta^{3,\epsilon}%
(t)|^{2}dt\right\}  =o(\epsilon^{2}).
\]

\end{proof}

Now we introduce the second-order adjoint equation:%
\begin{equation}
\left\{
\begin{array}
[c]{rl}%
-dP(t)= & \left\{  P(t)\left[  (D\sigma(t)^{\intercal}[1,p(t),K_{1}%
(t)]^{\intercal})^{2}+2Db(t)^{\intercal}[1,p(t),K_{1}(t)]^{\intercal}%
+H_{y}(t)\right]  \right. \\
& +2Q(t)D\sigma(t)^{\intercal}[1,p(t),K_{1}(t)]^{\intercal}+\left[
1,p(t),K_{1}(t)\right]  D^{2}H(t)\left[  1,p(t),K_{1}(t)\right]  ^{\intercal
}\left.  +H_{z}(t)K_{2}(t)\right\}  dt\\
& -Q(t)dB(t),\\
P(T)= & \phi_{xx}(\bar{X}(T)),
\end{array}
\right.  \label{eq-P-sigmaz0}%
\end{equation}

where
\[%
\begin{array}
[c]{ll}%
H(t,x,y,z,u,p,q)= & g(t,x,y,z,u)+pb(t,x,y,z,u)+q\sigma(t,x,y,z,u),
\end{array}
\]%
\[%
\begin{array}
[c]{ll}%
K_{2}(t)= & \left(  1-p(t)A(t)\right)  ^{-1}\left\{  p(t)\sigma_{y}%
(t)+2\left[  \sigma_{x}(t)+\sigma_{y}(t)p(t)+A(t)K_{1}(t)\right]  \right\}
P(t)\\
& +\left(  1-p(t)A(t)\right)  ^{-1}\left\{  Q(t)+p(t)[1,p(t)]D^{2}\sigma
_{1}(t)[1,p(t)]^{\intercal}\right\}  .
\end{array}
\]

(\ref{eq-P-sigmaz0}) is a linear BSDE with non-Lipschitz coefficient for
$P(\cdot)$. Then, by Theorem \ref{q-exp-th} in appendix, (\ref{eq-P-sigmaz0})
has a unique pair of solution according to Theorem 5.21 in \cite{Pardoux-book}%
. By the same analysis as in Lemma \ref{relation-y2}, we introduce the
following auxiliary equation:
\begin{equation}%
\begin{array}
[c]{l}%
\hat{Y}(t)=\int_{t}^{T}\left\{  (H_{y}(s)+\sigma_{y}(s)g_{z}%
(s)p(s)(1-p(s)A(s))^{-1})\hat{Y}(s)+\left(  H_{z}(s)+\sigma_{z}(s)g_{z}%
(s)p(s)(1-p(s)A(s))^{-1}\right)  \hat{Z}(s)\right. \\
\ \ \ \ \ \ \ \ \ \ \left.  +\left[  \delta H(s,\Delta)+\frac{1}{2}%
P(s)\delta\sigma(s,\Delta)^{2}\right]  I_{E_{\epsilon}}(s)\right\}
ds-\int_{t}^{T}\hat{Z}(s)dB(s),
\end{array}
\label{yhat-sigmaz0}%
\end{equation}
where $\delta H(s,\Delta):=p(s)\delta b(s,\Delta)+q(s)\delta\sigma
(s,\Delta)+\delta g(s,\Delta)$, and obtain the following relationship.

\begin{lemma}
\label{relation-second-order-q-unbound}Suppose Assumptions \ref{assum-2},
\ref{assum-3} and \ref{assm-sig-small} hold. Then%
\[%
\begin{array}
[c]{rl}%
Y_{2}(t) & =p(t)X_{2}(t)+\frac{1}{2}P(t)X_{1}(t)^{2}+\hat{Y}(t),\\
Z_{2}(t) & =\mathbf{I(t)}+\hat{Z}(t),
\end{array}
\]
where $(\hat{Y}(\cdot),\hat{Z}(\cdot))$ is the solution to
\eqref{yhat-sigmaz0} and
\begin{align*}
\mathbf{I(t)}  &  =K_{1}(t)X_{2}(t)+\frac{1}{2}K_{2}(t)X_{1}^{2}%
(t)+(1-p(t)A(t))^{-1}p(t)(\sigma_{y}(t)\hat{Y}(t)+A(t)\hat{Z}(t))+P(t)\delta
\sigma(t,\Delta)X_{1}(t)I_{E_{\epsilon}}(t)\\
&  \;+(1-p(t)A(t))^{-1}p(t)\left[  \delta\sigma_{x}(t,\Delta)X_{1}%
(t)+\delta\sigma_{y}(t,\Delta)p(t)X_{1}(t)\right]  I_{E_{\epsilon}}(t).
\end{align*}

\end{lemma}

\begin{proof}
Applying the techniques in Lemma \ref{lemma-y1}, we can deduce the above
relationship similarly.
\end{proof}

Combing the estimates in Lemma \ref{est-second-order-q-unbound} and the
relationship in Lemma \ref{relation-second-order-q-unbound}, we deduce that
\[
Y^{\epsilon}(0)-\bar{Y}(0)=Y_{1}(0)+Y_{2}(0)+o(\epsilon)=\hat{Y}%
(0)+o(\epsilon)\geq0.
\]
Define
\[%
\begin{array}
[c]{l}%
\mathcal{H}(t,x,y,z,u,p,q,P)=pb(t,x,y,z,u)+q\sigma(t,x,y,z,u)+\frac{1}%
{2}P(\sigma(t,x,y,z,u)-\sigma(t,\bar{X}(t),\bar{Y}(t),\bar{Z}(t),\bar
{u}(t)))^{2}\\
\text{ \ \ \ \ \ \ \ \ \ \ \ \ \ \ \ \ \ \ \ \ \ \ \ \ \ \ \ \ }%
\ +g(t,x,y,z+p(t)(\sigma(t,x,y,z,u)-\sigma(t,\bar{X}(t),\bar{Y}(t),\bar
{Z}(t),\bar{u}(t))),u).
\end{array}
\]
By the same analysis as in Theorem \ref{Th-MP}, we obtain the following
maximum principle.

\begin{theorem}
\label{th-mp-q-unboud}Suppose Assumptions \ref{assum-2}, \ref{assum-3} and
\ref{assm-sig-small} hold. Let $\bar{u}(\cdot)\in\mathcal{U}[0,T]$ be optimal
and $(\bar{X}(\cdot),\bar{Y}(\cdot),\bar{Z}(\cdot))$ be the corresponding
state processes of (\ref{state-eq}). Then the following stochastic maximum
principle holds:
\[
\mathcal{H}(t,\bar{X}(t),\bar{Y}(t),\bar{Z}(t),u,p(t),q(t),P(t))\geq
\mathcal{H}(t,\bar{X}(t),\bar{Y}(t),\bar{Z}(t),\bar{u}%
(t),p(t),q(t),P(t)),\ \ \ \forall u\in U,\ a.e.,\ a.s..
\]

\end{theorem}

\subsection{The general case\label{sec-general}}

When Brownian motion in \eqref{state-eq} is $d$-dimensional, by similar
analysis as for $1$-dimensional case, we obtain the following results.

The state equation becomes
\begin{equation}
\left\{
\begin{array}
[c]{rl}%
dX(t)= & b(t,X(t),Y(t),Z(t),u(t))dt+\sigma^{\intercal}%
(t,X(t),Y(t),Z(t),u(t))dB(t)\\
dY(t)= & -g(t,X(t),Y(t),Z(t),u(t))dt+Z^{\intercal}(t)dB(t),\\
X(0)= & x_{0},\ Y(T)=\phi(X(T)),
\end{array}
\right.  \label{state-eq-multi}%
\end{equation}
where
\[
\sigma:[0,T]\times\mathbb{R}\times\mathbb{R}\times\mathbb{R}^{d}\times
U\rightarrow\mathbb{R}^{d}.
\]
The first-order adjoint equation is
\begin{equation}
\left\{
\begin{array}
[c]{rl}%
dp(t)= & -\left\{  g_{x}(t)+g_{y}(t)p(t)+\left\langle g_{z}(t),K_{1}%
(t)\right\rangle +(b_{x}(t)+b_{y}(t)p(t))p(t)\right. \\
& \left.  +p(t)\langle b_{z}(t),K_{1}(t)\rangle+\langle(\sigma_{x}%
(t)+\sigma_{y}(t)p(t)),q(t)\rangle+\langle q(t),\sigma_{z}(t)K_{1}%
(t)\rangle\right\}  dt+q^{\intercal}(t)dB(t),\\
p(T)= & \phi_{x}(\bar{X}(T)),
\end{array}
\right.  \label{eq-p-multi}%
\end{equation}
where
\[%
\begin{array}
[c]{rl}%
K_{1}(t) & =(I-p(t)\sigma_{z}(t))^{-1}\left[  p(t)(\sigma_{x}(t)+\sigma
_{y}(t)p(t))+q(t)\right]  \in\mathbb{R}^{d},\\
g_{z}(t) & :=(g_{z^{1}}(t),g_{z^{2}}(t),...,g_{z^{d}}(t))^{\intercal},\text{
}\\
b_{z}(t) & :=(b_{z^{1}}(t),b_{z^{2}}(t),...,b_{z^{d}}(t))^{\intercal},\\
\sigma_{z}(t) & =\left(
\begin{array}
[c]{c}%
\sigma_{z^{1}}^{1}(t),\sigma_{z^{2}}^{1}(t),...,\sigma_{z^{d}}^{1}(t)\\
\sigma_{z^{1}}^{2}(t),\sigma_{z^{2}}^{2}(t),...,\sigma_{z^{d}}^{2}(t)\\
\vdots\\
\sigma_{z^{1}}^{d}(t),\sigma_{z^{2}}^{d}(t),...,\sigma_{z^{d}}^{d}(t)
\end{array}
\right)  \in\mathbb{R}^{d\times d}.
\end{array}
\]

Denote the Hessian matrix of $\sigma^{i}(t)$ with respect to $(x,y,z^{1}%
,z^{2},...,z^{d})$ by $D^{2}\sigma^{i}$. Set%
\[
\lbrack1,p(t),K_{1}^{\intercal}(t)]D^{2}\sigma(t)[1,p(t),K_{1}^{\intercal
}(t)]^{\intercal}=\left(
\begin{array}
[c]{c}%
\lbrack1,p(t),K_{1}^{\intercal}(t)]D^{2}\sigma^{1}(t)[1,p(t),K_{1}^{\intercal
}(t)]^{\intercal}\\
\lbrack1,p(t),K_{1}^{\intercal}(t)]D^{2}\sigma^{2}(t)[1,p(t),K_{1}^{\intercal
}(t)]^{\intercal}\\
\vdots\\
\lbrack1,p(t),K_{1}^{\intercal}(t)]D^{2}\sigma^{d}(t)[1,p(t),K_{1}^{\intercal
}(t)]^{\intercal}%
\end{array}
\right)  \in\mathbb{R}^{d}.
\]

Then, the second-order adjoint equation is
\begin{equation}
\left\{
\begin{array}
[c]{rl}%
-dP(t)= & \left\{  P(t)[(\sigma_{x}(t)+p(t)\sigma_{y}(t)+\sigma_{z}%
(t)K_{1}(t))^{\intercal}(\sigma_{x}(t)+p(t)\sigma_{y}(t)+\sigma_{z}%
(t)K_{1}(t))\right. \\
& +2(b_{x}(t)+b_{y}(t)p(t)+\langle b_{z}(t),K_{1}(t)\rangle)]+2\langle
Q(t),(\sigma_{x}(t)+p(t)\sigma_{y}(t)+\sigma_{z}(t)K_{1}(t))\rangle\\
& +p(t)b_{y}(t)P(t)+p(t)[1,p(t),K_{1}^{\intercal}(t)]D^{2}b(t)[1,p(t),K_{1}%
^{\intercal}(t)]^{\intercal}\\
& +\langle q(t),\mathbf{[}\sigma_{y}(t)P(t)+[1,p(t),K_{1}^{\intercal}%
(t)]D^{2}\sigma(t)[1,p(t),K_{1}^{\intercal}(t)]^{\intercal}\mathbf{]}%
\rangle+g_{y}(t)P(t)\\
& \left.  +[I,p(t),K_{1}^{\intercal}(t)]D^{2}g(t)[I,p(t),K_{1}^{\intercal
}(t)]^{\intercal}+\left\langle g_{z}(t)+b_{z}(t)p(t),K_{2}(t)\right\rangle
+\langle q(t),\sigma_{z}(t)K_{2}(t)\rangle\right\}  dt\\
& -Q^{\intercal}(t)dB(t),\\
P(T)= & \phi_{xx}(\bar{X}(T)),
\end{array}
\right.  \label{eq-P-multi}%
\end{equation}
where
\[%
\begin{array}
[c]{ll}%
K_{2}(t)= & (I-p(t)\sigma_{z}(t))^{-1}p(t)\left\{  \sigma_{y}%
(t)P(t)+[1,p(t),K_{1}^{\intercal}(t)]D^{2}\sigma(t)[1,p(t),K_{1}^{\intercal
}(t)]^{\intercal}\right\} \\
& +(I-p(t)\sigma_{z}(t))^{-1}\{Q(t)+2P(t)(\sigma_{x}(t)+\sigma_{y}%
(t)p(t)+\sigma_{z}(t)K_{1}(t))\}\in\mathbb{R}^{d}.
\end{array}
\]

Define
\begin{equation}%
\begin{array}
[c]{ll}%
\mathcal{H}(t,x,y,z,u,p,q,P)= & pb(t,x,y,z+\Delta(t),u)+\langle q,\sigma
(t,x,y,z+\Delta(t),u)\rangle\\
& +\frac{1}{2}P(\sigma(t,x,y,z+\Delta(t),u)-\sigma(t,\bar{X}(t),\bar
{Y}(t),\bar{Z}(t),\bar{u}(t)))^{\intercal}\\
& \cdot(\sigma(t,x,y,z+\Delta(t),u)-\sigma(t,\bar{X}(t),\bar{Y}(t),\bar
{Z}(t),\bar{u}(t)))\\
& +g(t,x,y,z+\Delta(t),u),
\end{array}
\label{h-function-multi}%
\end{equation}

where $\Delta(t)\ $satisfies
\begin{equation}
\Delta(t)=p(t)(\sigma(t,\bar{X}(t),\bar{Y}(t),\bar{Z}(t)+\Delta(t),u)-\sigma
(t,\bar{X}(t),\bar{Y}(t),\bar{Z}(t),\bar{u}(t))),\;t\in\lbrack0,T].
\label{def-delt-multi}%
\end{equation}

Thus, we obtain the following maximum principle.

\begin{theorem}
Suppose Assumptions \ref{assum-2}, \ref{assum-3} and \ref{assm-q-bound} hold.
Let $\bar{u}(\cdot)\in\mathcal{U}[0,T]$ be optimal and $(\bar{X}(\cdot
),\bar{Y}(\cdot),\bar{Z}(\cdot))$ be the corresponding state processes of
(\ref{state-eq-multi}). Then the following stochastic maximum principle
holds:
\[
\mathcal{H}(t,\bar{X}(t),\bar{Y}(t),\bar{Z}(t),u,p(t),q(t),P(t))\geq
\mathcal{H}(t,\bar{X}(t),\bar{Y}(t),\bar{Z}(t),\bar{u}%
(t),p(t),q(t),P(t)),\ \ \ \forall u\in U,\ a.e.,\ a.s.,
\]
where $(p\left(  \cdot\right)  ,q\left(  \cdot\right)  )$, $\left(  P\left(
\cdot\right)  ,Q\left(  \cdot\right)  \right)  $ satisfy (\ref{eq-p-multi}),
(\ref{eq-P-multi}) respectively, and $\Delta(\cdot)$ satisfies
(\ref{def-delt-multi}).\ 
\end{theorem}

\begin{remark}
The above theorem still hold under Assumptions \ref{assum-2}, \ref{assum-3}
and \ref{assm-sig-small}.
\end{remark}

\section{A linear quadratic control problem}

In this section, we study a linear quadratic control problem by the results in
the section 3. For simplicity of presentation, we suppose all the processes
are one dimensional.

Consider the following linear forward-backward stochastic control system
\begin{equation}
\left\{
\begin{array}
[c]{rcl}%
dX(t) & = & [A_{1}(t)X(t)+B_{1}(t)Y(t)+C_{1}(t)Z(t)+D_{1}(t)u(t)]dt\\
&  & +[A_{2}(t)X(t)+B_{2}(t)Y(t)+C_{2}(t)Z(t)+D_{2}(t)u(t)]dB(t),\\
dY(t) & = & -[A_{3}(t)X(t)+B_{3}(t)Y(t)+C_{3}(t)Z(t)+D_{3}%
(t)u(t)]dt+Z(t)dB(t),\\
X(0) & = & x_{0},\ Y(T)=FX(T)+J,
\end{array}
\right.  \label{state-lq}%
\end{equation}
and minimizing the following cost functional
\begin{equation}
J(u(\cdot))=\mathbb{E}\left[  \int_{0}^{T}\left(  A_{4}(t)X(t)^{2}%
+B_{4}(t)Y(t)^{2}+C_{4}(t)Z(t)^{2}+D_{4}(t)u(t)^{2}\right)  dt+GX(T)^{2}%
+Y(0)^{2}\right]  , \label{cost-lq}%
\end{equation}
where $A_{i}$, $B_{i}$, $C_{i}$, $D_{i}$ $i=1,2,3,4$ are deterministic
$\mathbb{R}$-valued functions, $F$, $G$ are deterministic constants and $J$ is
$\mathcal{F}_{T}$-measurable bounded random variable. Let $\bar{u}(\cdot)$ be
the optimal control, and the corresponding optimal state is $(\bar{X}%
(\cdot),\bar{Y}(\cdot),\bar{Z}(\cdot))$.

The variational equation becomes
\[
\left\{
\begin{array}
[c]{l}%
d\left(  X_{1}(t)+X_{2}(t)\right) \\
\text{ }=[A_{1}(t)\left(  X_{1}(t)+X_{2}(t)\right)  +B_{1}(t)\left(
Y_{1}(t)+Y_{2}(t)\right)  +C_{1}(t)\left(  Z_{1}(t)+Z_{2}(t)\right)
+D_{1}(t)(u^{\epsilon}(t)-\bar{u}(t))]dt\\
\text{ \ \ }+[A_{2}(t)\left(  X_{1}(t)+X_{2}(t)\right)  +B_{2}(t)\left(
Y_{1}(t)+Y_{2}(t)\right)  +C_{2}(t)\left(  Z_{1}(t)+Z_{2}(t)\right)
+D_{2}(t)(u^{\epsilon}(t)-\bar{u}(t))]dB(t),\\
d\left(  Y_{1}(t)+Y_{2}(t)\right) \\
\text{ }=-[A_{3}(t)\left(  X_{1}(t)+X_{2}(t)\right)  +B_{3}(t)\left(
Y_{1}(t)+Y_{2}(t)\right)  +C_{3}(t)\left(  Z_{1}(t)+Z_{2}(t)\right)
+D_{3}(t)(u^{\epsilon}(t)-\bar{u}(t))]dt\\
\ \ \ +\left(  Z_{1}(t)+Z_{2}(t)\right)  dB(t),\\
X_{1}(0)+X_{2}(0)=0,\ Y_{1}(T)+Y_{2}(T)=F\left(  X_{1}(T)+X_{2}(T)\right)  ,
\end{array}
\right.
\]
and the first order adjoint equation is
\begin{equation}
\left\{
\begin{array}
[c]{rl}%
dp(t)= & -\left\{  A_{3}(t)+B_{3}(t)p(t)+C_{3}(t)K_{1}(t)+A_{1}(t)p(t)+B_{1}%
(t)p^{2}(t)\right. \\
& \left.  +C_{1}(t)K_{1}(t)p(t)+A_{2}(t)q(t)+B_{2}(t)p(t)q(t)+C_{2}%
(t)K_{1}(t)q(t)\right\}  dt+q(t)dB(t),\\
p(T)= & F,
\end{array}
\right.  \label{eq-p-lq}%
\end{equation}
where
\[
K_{1}(t)=(1-p(t)C_{2}(t))^{-1}\left[  A_{2}(t)p(t)+B_{2}(t)p^{2}%
(t)+q(t)\right]  .
\]
This adjoint equation is a nonlinear backward stochastic differential equation
with deterministic coefficients and the solution to \eqref{eq-p-lq} is
$(p(\cdot),0)$, which $p(\cdot)$ satisfies the following ODE
\begin{equation}
\left\{
\begin{array}
[c]{rl}%
dp(t)= & -\left\{  A_{3}(t)+B_{3}(t)p(t)+C_{3}(t)K_{1}(t)+A_{1}(t)p(t)+B_{1}%
(t)p^{2}(t)+C_{1}(t)K_{1}(t)p(t)\right\}  dt,\\
p(T)= & F,
\end{array}
\right.  \label{eq-p-lq-ode}%
\end{equation}
with
\[
K_{1}(t)=(1-p(t)C_{2}(t))^{-1}\left[  A_{2}(t)p(t)+B_{2}(t)p^{2}(t)\right]  .
\]

\begin{remark}
It should be note that in our context the Assumption \ref{assm-q-bound} holds.
\end{remark}

Moreover, $\Delta(t)$ has the following explicitly form
\[
\Delta(t)=(1-p(t)C_{2}(t))^{-1}p(t)D_{2}(t)(u(t)-\bar{u}(t)).
\]
Since $\bar{u}(\cdot)$ is the optimal control,
\begin{equation}
J(u^{\epsilon}(\cdot))-J(\bar{u}(\cdot))\geq0. \label{eq-lq-cost}%
\end{equation}
By Lemma \ref{lemma-est-sup}, the following estimates hold,
\[%
\begin{array}
[c]{lll}%
X^{\epsilon}(t)-\bar{X}(t) & = & X_{1}(t)+X_{2}(t)+o(\epsilon),\\
Y^{\epsilon}(t)-\bar{Y}(t) & = & Y_{1}(t)+Y_{2}(t)+o(\epsilon),\\
Z^{\epsilon}(t)-\bar{Z}(t) & = & Z_{1}(t)+Z_{2}(t)+o(\epsilon).
\end{array}
\]
We can expand (\ref{eq-lq-cost}) term by term as follows
\[%
\begin{array}
[c]{l}%
\mathbb{E}\left[  \int_{0}^{T}A_{4}(t)\left(  X^{\epsilon}(t)^{2}-\bar
{X}(t)^{2}\right)  dt\right] \\
\text{ }=\mathbb{E}\left\{  \int_{0}^{T}\left[  2A_{4}(t)\bar{X}(t)\left(
X_{1}(t)+X_{2}(t)\right)  +A_{4}(t)X_{1}(t)^{2}\right]  dt\right\}
+o(\epsilon).
\end{array}
\]
Similarly, one has
\[%
\begin{array}
[c]{rl}%
\mathbb{E}\left[  \int_{0}^{T}B_{4}(t)\left(  Y^{\epsilon}(t)^{2}-\bar
{Y}(t)^{2}\right)  dt\right]  & =\mathbb{E}\left\{  \int_{0}^{T}\left[
2B_{4}(t)\bar{Y}(t)\left(  Y_{1}(t)+Y_{2}(t)\right)  +B_{4}(t)Y_{1}%
(t)^{2}\right]  dt\right\}  +o(\epsilon);\\
\mathbb{E}\left[  G\left(  X^{\epsilon}(T)^{2}-\bar{X}(T)^{2}\right)  \right]
& =\mathbb{E}\left[  2G\bar{X}(T)\left(  X_{1}(T)+X_{2}(T)\right)
+GX_{1}(T)^{2}\right]  +o(\epsilon);\\
Y^{\varepsilon}(0)^{2}-\bar{Y}(0)^{2} & =2\bar{Y}(0)\left(  Y_{1}%
(0)+Y_{2}(0)\right)  +Y_{1}(0)^{2}+o(\epsilon);\\
\mathbb{E}\left[  \int_{0}^{T}C_{4}(t)\left(  Z^{\epsilon}(t)^{2}-\bar
{Z}(t)^{2}\right)  dt\right]  & =\mathbb{E}\left\{  \int_{0}^{T}\left[
2C_{4}(t)\bar{Z}(t)\left(  Z_{1}(t)+Z_{2}(t)\right)  +C_{4}(t)Z_{1}%
(t)^{2}\right]  dt\right\}  +o(\epsilon).\
\end{array}
\]
Thus
\[%
\begin{array}
[c]{l}%
J(u^{\epsilon}(\cdot))-J(\bar{u}(\cdot))\\
=\mathbb{E}\left\{  \int_{0}^{T}\left[  2A_{4}(t)\bar{X}(t)\left(
X_{1}(t)+X_{2}(t)\right)  +2B_{4}(t)\bar{Y}(t)\left(  Y_{1}(t)+Y_{2}%
(t)\right)  +2C_{4}(t)\bar{Z}(t)\left(  Z_{1}(t)+Z_{2}(t)\right)  \right.
\right. \\
\left.  +A_{4}(t)X_{1}(t)^{2}+B_{4}(t)Y_{1}(t)^{2}+C_{4}(t)Z_{1}(t)^{2}%
+2D_{4}(t)\bar{u}(t)\left(  u^{\epsilon}(t)-\bar{u}(t)\right)  +D_{4}%
(t)\left(  u^{\epsilon}(t)-\bar{u}(t)\right)  ^{2}\right]  dt\\
\left.  +2G\bar{X}(T)\left(  X_{1}(T)+X_{2}(T)\right)  +GX_{1}(T)^{2}+2\bar
{Y}(0)\left(  Y_{1}(0)+Y_{2}(0)\right)  +Y_{1}(0)^{2}\right\}  +o(\epsilon).
\end{array}
\]
\ \ \ \ \ $\ \ $Introduce the adjoint equation for $X_{1}(t)+X_{2}%
(t),Y_{1}(t)+Y_{2}(t),Z_{1}(t)+Z_{2}(t)$ as
\[
\left\{
\begin{array}
[c]{rl}%
dh(t)= & \left[  B_{3}(t)h(t)+B_{1}(t)m(t)+B_{2}(t)n(t)+2B_{4}(t)Y(t)\right]
dt\\
& +\left[  C_{3}(t)h(t)+C_{1}(t)m(t)+C_{2}(t)n(t)+2C_{4}(t)Z(t)\right]
dB(t),\\
h(0)= & 2\bar{Y}(0),\\
dm(t)= & -\left[  A_{3}(t)h(t)+A_{1}(t)m(t)+A_{2}(t)n(t)+2A_{4}(t)X(t)\right]
dt+n(t)dB(t),\\
m(T)= & 2G\bar{X}(T)+Fh(T).
\end{array}
\right.
\]
Applying It\^{o}'s formula to $m(t)\left(  X_{1}(t)+X_{2}(t)\right)
-h(t)\left(  Y_{1}(t)+Y_{2}(t)\right)  $, we get
\[%
\begin{array}
[c]{l}%
J(u^{\epsilon}(\cdot))-J(\bar{u}(\cdot))\\
=\mathbb{E}\left\{  \int_{0}^{T}\left[  \left(  D_{1}(t)m(t)+D_{2}%
(t)n(t)+D_{3}(t)h(t)+2D_{4}(t)\bar{u}(t)\right)  \left(  u^{\epsilon}%
(t)-\bar{u}(t)\right)  \right.  \right. \\
\left.  \left.  +A_{4}(t)X_{1}(t)^{2}+B_{4}(t)Y_{1}(t)^{2}+C_{4}%
(t)Z_{1}(t)^{2}+D_{4}(t)\left(  u^{\epsilon}(t)-\bar{u}(t)\right)
^{2}\right]  dt+GX_{1}(T)^{2}+Y_{1}(0)^{2}\right\}  +o(\epsilon).
\end{array}
\]
Noting that the relationship between $X_{1}(t),Y_{1}(t)$ and $Z_{1}(t),$%
\[%
\begin{array}
[c]{ll}%
Y_{1}(t)= & p(t)X_{1}(t),\\
Z_{1}(t)= & K_{1}(t)X_{1}(t)+\Delta(t)I_{E_{\epsilon}}(t),
\end{array}
\]
thus,%
\[%
\begin{array}
[c]{l}%
J(u^{\epsilon}(\cdot))-J(\bar{u}(\cdot))\\
=\mathbb{E}\left\{  \int_{0}^{T}\left[  \left(  A_{4}(t)+B_{4}(t)p(t)^{2}%
+C_{4}(t)K_{1}(t)^{2}\right)  X_{1}(t)^{2}+C_{4}(t)\Delta(t)^{2}%
I_{E_{\epsilon}}(t)+D_{4}(t)\left(  u^{\epsilon}(t)-\bar{u}(t)\right)
^{2}\right.  \right. \\
\text{ \ \ \ }\left.  \left.  +\left(  D_{1}(t)m(t)+D_{2}(t)n(t)+D_{3}%
(t)h(t)+2D_{4}(t)\bar{u}(t)\right)  \left(  u^{\epsilon}(t)-\bar{u}(t)\right)
\right]  dt+GX_{1}(T)^{2}\right\}  +o(\epsilon).
\end{array}
\]
$\ \ \ $Introducing the adjoint equation for $X_{1}(t)^{2}$,
\begin{equation}
\left\{
\begin{array}
[c]{rl}%
-dP(t)= & \left[  R_{1}(t)P(t)+R_{2}(t)Q(t)+A_{4}(t)+B_{4}(t)p(t)^{2}%
+C_{4}(t)K_{1}(t)^{2}\right]  dt-Q(t)dB(t),\\
P(T)= & G,
\end{array}
\right.
\end{equation}
where
\[%
\begin{array}
[c]{ll}%
R_{1}(t)= & 2\left(  A_{1}(t)+B_{1}(t)p(t)+C_{1}(t)K_{1}(t)\right)  +\left(
A_{2}(t)+B_{2}(t)p(t)+C_{2}(t)K_{1}(t)\right)  ^{2},\\
R_{2}(t)= & 2\left(  A_{2}(t)+B_{2}(t)p(t)+C_{2}(t)K_{1}(t)\right)  .
\end{array}
\]
Similar to $p(\cdot)$, the solution to $(P(\cdot),Q(\cdot))$ is $(P(\cdot
),0)$, which $P(\cdot)$ satisfies the following ODE,
\begin{equation}
\left\{
\begin{array}
[c]{rl}%
-dP(t)= & \left[  R_{1}(t)P(t)+A_{4}(t)+B_{4}(t)p(t)^{2}+C_{4}(t)K_{1}%
(t)^{2}\right]  dt,\\
P(T)= & G,
\end{array}
\right.
\end{equation}
where
\[%
\begin{array}
[c]{ll}%
R_{1}(t)= & 2\left(  A_{1}(t)+B_{1}(t)p(t)+C_{1}(t)K_{1}(t)\right)  +\left(
A_{2}(t)+B_{2}(t)p(t)+C_{2}(t)K_{1}(t)\right)  ^{2}.\\
&
\end{array}
\]
We obtain
\[%
\begin{array}
[c]{l}%
J(u^{\epsilon}(\cdot))-J(\bar{u}(\cdot))\\
=\mathbb{E}\left\{  \int_{0}^{T}\left[  \left(  D_{1}(t)m(t)+D_{2}%
(t)n(t)+D_{3}(t)h(t)+2D_{4}(t)\bar{u}(t)\right)  \left(  u^{\epsilon}%
(t)-\bar{u}(t)\right)  \right.  \right. \\
\left.  \left.  +P(t)D_{2}(t)^{2}(u^{\epsilon}(t)-\bar{u}(t))^{2}%
+D_{4}(t)\left(  u^{\epsilon}(t)-\bar{u}(t)\right)  ^{2}+C_{4}(t)\Delta
(t)^{2}I_{E_{\epsilon}}(t)\right]  dt\right\}  +o(\epsilon).
\end{array}
\]
Thus, we obtain the following maximum principle for (\ref{state-lq}%
)-(\ref{cost-lq}).

\begin{theorem}
\label{th-mp-lq}Suppose Assumptions \ref{assum-2} and \ref{assum-3} hold. Let
$\bar{u}(\cdot)\in\mathcal{U}[0,T]$ be optimal and $(\bar{X}(\cdot),\bar
{Y}(\cdot),\bar{Z}(\cdot))$ be the corresponding state processes of
(\ref{state-lq}). Then the following stochastic maximum principle holds:
\[%
\begin{array}
[c]{l}%
\left(  D_{1}(t)m(t)+D_{2}(t)n(t)+D_{3}(t)h(t)+2D_{4}(t)\bar{u}(t)\right)
(u-\bar{u}(t))\\
+\left[  \frac{C_{4}(t)p(t)^{2}D_{2}(t)^{2}}{\left(  1-p(t)C_{2}(t)\right)
^{2}}+D_{4}(t)+P(t)D_{2}(t)^{2}\right]  (u-\bar{u}(t))^{2}\geq0,\ \forall u\in
U,\ a.e.,\ a.s..
\end{array}
\]

\end{theorem}

\begin{remark}
Using Theorem \ref{th-mp-q-unboud}, we can also consider linear quadratic
control problem with random coefficients.
\end{remark}

Now, we give an example to show the difference between the global and local
maximum principle.

\begin{example}
Consider the following linear forward-backward stochastic control system
\begin{equation}
\left\{
\begin{array}
[c]{rcl}%
dX(t) & = & [aZ(t)+bu(t)]dB(t),\\
dY(t) & = & -cu(t)dt+Z(t)dB(t),\\
X(0) & = & 1,\ Y(T)=dX(T),
\end{array}
\right.
\end{equation}
and minimizing the following cost functional
\[
J(u(\cdot))=\mathbb{E}\left[  \int_{0}^{T}u(t)^{2}dt\right]  +Y(0)^{2},
\]
where $a$, $b$, $c$, $d$ are constants such that, $0<\left\vert 2cd\right\vert
\leq1$ and $ad<1$, and $U=\left\{  -1,0,1\right\}  $. Let $\bar{u}(\cdot)$ be
the optimal control, and the corresponding optimal state is $(\bar{X}%
(\cdot),\bar{Y}(\cdot),\bar{Z}(\cdot))$. In this case $p(t)=d$, $q(t)=0$,
$P(t)=Q(t)=0$, $h(t)=2\bar{Y}(0)$, $m(t)=2d\bar{Y}(0)$, $n(t)=0$, for
$t\in\left[  0,T\right]  $. The maximum principle by Theorem \ref{th-mp-lq}
is
\begin{equation}
2\left(  c\bar{Y}(0)+\bar{u}(t)\right)  (u-\bar{u}(t))+(u-\bar{u}(t))^{2}%
\geq0,\ \forall u\in U\ a.e.,\ a.s.. \label{mp-exp}%
\end{equation}
Noting that $\bar{Y}(0)=d$ for $\bar{u}(t)=0$ , then it is easy to check that
$\bar{u}(t)=0$ satisfies the maximum principle \eqref{mp-exp}. Furthermore, we
can prove $\bar{u}(t)=0$ is the optimal control. For each $u\left(
\cdot\right)  \in\mathcal{U}[0,T]$,
\[
Y(0)=\mathbb{E}\left[  dX(T)+\int_{0}^{T}cu(t)dt\right]  =d+\mathbb{E}\left[
\int_{0}^{T}cu(t)dt\right]  .
\]
Then
\[%
\begin{array}
[c]{rl}%
J(u(\cdot))-J(\bar{u}(\cdot))= & \mathbb{E}\left[  \int_{0}^{T}u(t)^{2}%
dt\right]  +Y(0)^{2}-d^{2}\\
= & \left(  \mathbb{E}\left[  \int_{0}^{T}cu(t)dt\right]  \right)
^{2}+\mathbb{E}\left[  \int_{0}^{T}\left(  u(t)^{2}+2cdu(t)\right)  dt\right]
\\
\geq & 0,
\end{array}
\]
which implies $\bar{u}(\cdot)=0$ is optimal. \newline\ \ When the control
domain is $U=[-1,1]$, similar to Corollary \ref{cor-mp-convex}, by Theorem
\ref{th-mp-lq} we obtain the following maximum principle,
\begin{equation}
2\left(  c\bar{Y}(0)+\bar{u}(t)\right)  (u-\bar{u}(t))\geq0,\ \forall u\in
U\ a.e.,\ a.s.. \label{mp-exp-convex}%
\end{equation}
It is obvious that $\bar{u}(\cdot)=0$ does not satisfy the maximum principle \eqref{mp-exp-convex}.
\end{example}

\section{Appendix}

\subsection{$L^{p}$-estimate of decoupled FBSDEs}

The following Lemma is a combination of Theorem 3.17 and Theorem 5.17 in
\cite{Pardoux-book}.

\begin{lemma}
\label{sde-bsde}For each fixed $p>1$ and a pair of adapted stochastic process
$(y(\cdot),z(\cdot))$, consider the following system
\begin{equation}
\left\{
\begin{array}
[c]{rl}%
dX(t)= & b(t,X(t),y(t),z(t))dt+\sigma(t,X(t),y(t),z(t))dB(t),\\
dY(t)= & -g(t,X(t),Y(t),Z(t))dt+Z(t)dB(t),\\
X(0)= & x_{0},\ Y(t)=\phi(X(T)),
\end{array}
\right.  \label{fbsde-pardoux}%
\end{equation}
where $b$, $\sigma$, $g$, $\phi$ are the same in equation (\ref{fbsde}). If
the coefficients satisfy

(i) $b(\cdot,0,y(\cdot),z(\cdot))$, $\sigma(\cdot,0,y(\cdot),z(\cdot))$,
$g(\cdot,0,0,0)$ are $\mathbb{F}$-adapted processes and
\[
\mathbb{E}\left\{  |\phi(0)|^{p}+\left(  \int_{0}^{T}\left[
|b(t,0,y(t),z(t))|+|g(t,0,0,0)|\right]  dt\right)  ^{p}+\left(  \int_{0}%
^{T}|\sigma(t,0,y(t),z(t))|^{2}dt\right)  ^{\frac{p}{2}}\right\}  <\infty,
\]

(ii)%
\[%
\begin{array}
[c]{rl}%
|\psi(t,x_{1},y,z)-\psi(t,x_{2},y,z)| & \leq L_{1}|x_{1}-x_{2}|,\ \ \text{for
}\ \psi=b,\sigma;\\
|g(t,x_{1},y_{1},z_{1})-g(t,x_{2},y_{2},z_{2})| & \leq L_{1}(|x_{1}%
-x_{2}|+|y_{1}-y_{2}|+|z_{1}-z_{2}|),
\end{array}
\]
then (\ref{fbsde-pardoux}) has a unique solution $(X(\cdot),Y(\cdot
),Z(\cdot))\in L_{\mathcal{F}}^{p}(\Omega;C([0,T],\mathbb{R}^{n}))\times
L_{\mathcal{F}}^{p}(\Omega;C([0,T],\mathbb{R}^{m}))\times L_{\mathcal{F}%
}^{2,p}([0,T];\mathbb{R}^{m\times d})$ and there exists a constant $C_{p}$
which only depends on $L_{1}$, $p$, $T$ such that
\[%
\begin{array}
[c]{l}%
\mathbb{E}\left\{  \sup\limits_{t\in\lbrack0,T]}\left[  |X(t)|^{p}%
+|Y(t)|^{p}\right]  +\left(  \int_{0}^{T}|Z(t)|^{2}dt\right)  ^{\frac{p}{2}%
}\right\} \\
\ \leq C_{p}\mathbb{E}\left\{  \left[  \int_{0}^{T}\left(
|b(t,0,y(t),z(t))|+|g(t,0,0,0)|\right)  dt\right]  ^{p}+\left(  \int_{0}%
^{T}|\sigma(t,0,y(t),z(t))|^{2}dt\right)  ^{\frac{p}{2}}+|\phi(0)|^{p}%
+|x_{0}|^{p}\right\}  .
\end{array}
\]

\end{lemma}

\subsection{An estimate of $Z$ for some BSDEs}

Consider the following BSDE\ %

\begin{equation}
Y(t)=\xi+\int_{t}^{T}f(s,Y(s),Z(s))ds-\int_{t}^{T}Z(s)dB(s).
\label{appen-eq-bsde}%
\end{equation}

\begin{theorem}
\label{q-exp-th} Suppose that $(Y(\cdot),Z(\cdot))\in L_{\mathcal{F}}^{\infty
}(0,T;\mathbb{R})\times L_{\mathcal{F}}^{2,2}([0,T];\mathbb{R}^{d})$ solves
BSDE (\ref{appen-eq-bsde}), and \newline$\left\vert f(s,Y(s),Z(s))\right\vert
\leq C_{1}\left(  1+|Z(s)|^{2}\right)  $, where $C_{1}$ is a constant. Then
there exists a $\delta>0$ such that for each $\lambda_{1}<\delta$,
\begin{equation}
\mathbb{E}\left[  \left.  \exp\left(  \lambda_{1}\int_{t}^{T}|Z(s)|^{2}%
ds\right)  \right\vert \mathcal{F}_{t}\right]  \leq C\text{ and }%
\mathbb{E}\left[  \sup\limits_{0\leq t\leq T}\exp\left(  \lambda_{1}\int%
_{0}^{t}Z(s)dB(s)\right)  \right]  \leq C\text{,} \label{appwn-eq-11}%
\end{equation}
where $C$ depends on $C_{1}$, $\delta$, $T$ and $||\xi||_{\infty}$. Moreover,
for each $\lambda_{2}>0$,
\begin{equation}
\mathbb{E}\left[  \exp\left(  \lambda_{2}\int_{0}^{T}|Z(s)|ds\right)  \right]
<\infty. \label{appwn-eq-12}%
\end{equation}

\end{theorem}

\begin{proof}
In the following, $C$ is a constant, and will be changed from line to line.
Define
\[
u(x)=\frac{1}{4C_{1}^{2}}\left(  e^{2C_{1}x}-1-2C_{1}x\right)  .
\]
It is easy to check that $x\rightarrow u(|x|)$ is $C^{2}$. Applying It\^{o}'s
formula to $u(|Y(s)|)$, we get
\[%
\begin{array}
[c]{lll}%
u\left(  \left\vert Y(t)\right\vert \right)  & = & u\left(  \left\vert
Y(T)\right\vert \right)  +\int_{t}^{T}\{u^{^{\prime}}%
(|Y(s)|)sgn(Y(s))f(s,Y(s),Z(s))-\frac{1}{2}u^{^{\prime\prime}}%
(|Y(s)|)|Z(s)|^{2}\}ds\\
&  & -\int_{t}^{T}u^{^{\prime}}(|Y(s)|)sgn(Y(s))Z(s)dB(s)\\
& \leq & u\left(  \left\vert Y(T)\right\vert \right)  +\int_{t}^{T}%
\{u^{^{\prime}}(|Y(s)|)C_{1}(1+|Z(s)|^{2})-\frac{1}{2}u^{^{\prime\prime}%
}(|Y(s)|)|Z(s)|^{2}\}ds\\
&  & -\int_{t}^{T}u^{^{\prime}}(|Y(s)|)sgn(Y(s))Z(s)dB(s)\\
& \leq & C-\frac{1}{2}\int_{t}^{T}|Z(s)|^{2}ds-\int_{t}^{T}u^{^{\prime}%
}(|Y(s)|)sgn(Y(s))Z(s)dB(s).
\end{array}
\]
From the above inequality, we can deduce that, for each stopping time
$\tau\leq T$,%
\[
\mathbb{E}\left[  \left.  \int_{\tau}^{T}|Z\left(  s\right)  |^{2}%
ds\right\vert \mathcal{F}_{\tau}\right]  \leq C,
\]
where $C$ is independent of $\tau$. Thus $(\int_{0}^{t}Z\left(  s\right)
dB(s))_{t\in\lbrack0,T]}$ is a BMO martingale. By the Nirenberg inequality
(see Theorem 10.43 in \cite{HWY}), we obtain (\ref{appwn-eq-11}). For each
given $\lambda_{2}>0$, choose $\delta_{0}>0$ such that $\lambda_{2}%
\sqrt{\delta_{0}}<\delta$. Thus, by (\ref{appwn-eq-11}), we get
\[%
\begin{array}
[c]{l}%
\mathbb{E}\left[  \exp\left(  \lambda_{2}\int_{0}^{T}|Z(s)|ds\right)  \right]
\\
=\mathbb{E}\left[  \exp\left(  \lambda_{2}\int_{0}^{T-\delta_{0}%
}|Z(s)|ds\right)  \exp\left(  \lambda_{2}\int_{T-\delta_{0}}^{T}%
|Z(s)|ds\right)  \right] \\
=\mathbb{E}\left[  \exp\left(  \lambda_{2}\int_{0}^{T-\delta_{0}%
}|Z(s)|ds\right)  \mathbb{E}\left[  \left.  \exp\left(  \lambda_{2}%
\int_{T-\delta_{0}}^{T}|Z(s)|ds\right)  \right\vert \mathcal{F}_{T-\delta_{0}%
}\right]  \right] \\
\leq\mathbb{E}\left[  \exp\left(  \lambda_{2}\int_{0}^{T-\delta_{0}%
}|Z(s)|ds\right)  \mathbb{E}\left[  \left.  \exp\left(  \lambda_{2}%
\sqrt{\delta_{0}}\left[  \int_{T-\delta_{0}}^{T}|Z(s)|^{2}ds\right]
^{\frac{1}{2}}\right)  \right\vert \mathcal{F}_{T-\delta_{0}}\right]  \right]
\\
\leq\mathbb{E}\left[  \exp\left(  \lambda_{2}\int_{0}^{T-\delta_{0}%
}|Z(s)|ds\right)  \mathbb{E}\left[  \left.  e^{\lambda_{2}\sqrt{\delta_{0}}%
}+\exp\left(  \lambda_{2}\sqrt{\delta_{0}}\int_{T-\delta_{0}}^{T}%
|Z(s)|^{2}ds\right)  I_{\{\int_{T-\delta_{0}}^{T}|Z(s)|^{2}ds>1\}}\right\vert
\mathcal{F}_{T-\delta_{0}}\right]  \right] \\
\leq\left(  e^{\delta}+C\right)  \mathbb{E}\left[  \exp\left(  \lambda_{2}%
\int_{0}^{T-\delta_{0}}|Z(s)|ds\right)  \right] \\
\leq\left(  e^{\delta}+C\right)  ^{[\frac{T}{\delta_{0}}]+1}<\infty.
\end{array}
\]
This completes the proof.
\end{proof}

\subsection{Solution to linear FBSDEs}

\label{sect-solu-linearfbsde}

Considering the following forward-backward stochastic differential equation
\begin{equation}
\left\{
\begin{array}
[c]{rl}%
dX(t)= & \left[  \alpha_{1}(t)X(t)+\beta_{1}(t)Y(t)+\gamma_{1}(t)Z(t)+L_{1}%
(t)\right]  dt+\left[  \alpha_{2}(t)X(t)+\beta_{2}(t)Y(t)+\gamma
_{2}(t)Z(t)+L_{2}(t)\right]  dB(t),\\
dY(t)= & -\left[  \alpha_{3}(t)X(t)+\beta_{3}(t)Y(t)+\gamma_{3}(t)Z(t)+L_{3}%
(t)\right]  dt+Z(t)dB(t),\\
X(0)= & x_{0},\ Y(T)=\kappa X(T),
\end{array}
\right.  \label{appen-eq-xyz}%
\end{equation}
where $\alpha_{i}(\cdot)$, $\beta_{i}(\cdot)$, $\gamma_{i}(\cdot)$, $i=1,2,3$,
are bounded adapted processes, $L_{1}(\cdot)$, $L_{3}(\cdot)\in L_{\mathcal{F}%
}^{1,2}([0,T];\mathbb{R})$, $L_{2}(\cdot)\in L_{\mathcal{F}}^{2,2}%
([0,T];\mathbb{R})$ and $\kappa$ is an $\mathcal{F}_{T}$-measurable bounded
random variable. Suppose that the solution to (\ref{appen-eq-xyz}) has the
following relationship
\[
Y(t)=p(t)X(t)+\varphi(t),
\]
where $p(t)$, $\varphi(t)$ satisfy
\begin{equation}
\left\{
\begin{array}
[c]{rl}%
dp(t)= & -A(t)dt+q(t)dB(t),\\
d\varphi(t)= & -C(t)dt+\nu(t)dB(t),\\
p(T)= & \kappa,\ \varphi(T)=0,
\end{array}
\right.  \label{appen-eq-pq}%
\end{equation}
$A(t)$ and $C(t)$ will be determined later. Applying It\^{o}'s formula to
$p(t)X(t)+\varphi(t)$, we have
\begin{equation}%
\begin{array}
[c]{ll}%
d\left(  p(t)X(t)+\varphi(t)\right)  & =\left\{  p(t)\left[  \alpha
_{1}(t)X(t)+\beta_{1}(t)Y(t)+\gamma_{1}(t)Z(t)+L_{1}(t)\right]
-A(t)X(t)\right. \\
& \ \ \left.  +q(t)\left[  \alpha_{2}(t)X(t)+\beta_{2}(t)Y(t)+\gamma
_{2}(t)Z(t)+L_{2}(t)\right]  -C(t)\right\}  dt\\
& \ \ +\left\{  p(t)\left[  \alpha_{2}(t)X(t)+\beta_{2}(t)Y(t)+\gamma
_{2}(t)Z(t)+L_{2}(t)\right]  +q(t)X(t)+\nu(t)\right\}  dB(t).
\end{array}
\end{equation}
Comparing with the equation satisfied by $Y(t)$, one has
\begin{equation}
Z(t)=p(t)\left[  \alpha_{2}(t)X(t)+\beta_{2}(t)Y(t)+\gamma_{2}(t)Z(t)+L_{2}%
(t)\right]  +q(t)X(t)+\nu(t), \label{appen-eq-z}%
\end{equation}%
\begin{equation}%
\begin{array}
[c]{ll}%
-\left[  \alpha_{3}(t)X(t)+\beta_{3}(t)Y(t)+\gamma_{3}(t)Z(t)+L_{3}(t)\right]
& =p(t)\left[  \alpha_{1}(t)X(t)+\beta_{1}(t)Y(t)+\gamma_{1}(t)Z(t)+L_{1}%
(t)\right]  -A(t)X(t)\\
& \ \ +q(t)\left[  \alpha_{2}(t)X(t)+\beta_{2}(t)Y(t)+\gamma_{2}%
(t)Z(t)+L_{2}(t)\right]  -C(t).
\end{array}
\label{appen-relation-generator}%
\end{equation}
From equation \eqref{appen-eq-z}, we have the form of $Z(t)$ as
\begin{align*}
Z(t)  &  =\left(  1-p(t)\gamma_{2}(t)\right)  ^{-1}\left\{  p(t)\left[
\alpha_{2}(t)X(t)+\beta_{2}(t)Y(t)+L_{2}(t)\right]  +q(t)X(t)+\nu(t)\right\}
\\
&  =\left(  1-p(t)\gamma_{2}(t)\right)  ^{-1}\left[  \left(  \alpha
_{2}(t)p(t)+\beta_{2}(t)p(t)^{2}+q(t)\right)  X(t)+p(t)\beta_{2}%
(t)\varphi(t)+p(t)L_{2}(t)+\nu(t)\right]  .
\end{align*}
From the equation \eqref{appen-relation-generator}, and utilizing the form of
$Y(t)$ and $Z(t)$, we derive that%
\begin{equation}%
\begin{array}
[c]{rl}%
A(t)= & \alpha_{3}(t)+\beta_{3}(t)p(t)+\gamma_{3}(t)K_{1}(t)+\alpha
_{1}(t)p(t)+\beta_{1}(t)p^{2}(t)\\
& +\gamma_{1}(t)K_{1}(t)p(t)+\alpha_{2}(t)q(t)+\beta_{2}(t)p(t)q(t)+\gamma
_{2}(t)K_{1}(t)q(t),
\end{array}
\label{new-eq-111}%
\end{equation}
where
\[
K_{1}(t)=(1-p(t)\gamma_{2}(t))^{-1}\left[  \alpha_{2}(t)p(t)+\beta_{2}%
(t)p^{2}(t)+q(t)\right]  ,
\]
and%
\begin{equation}%
\begin{array}
[c]{rl}%
C(t)= & (\beta_{1}(t)p(t)+\beta_{2}(t)q(t)+\beta_{3}(t))\varphi(t)+p(t)L_{1}%
(t)+q(t)L_{2}(t)+L_{3}(t)\\
& +(\gamma_{1}(t)p(t)+\gamma_{2}(t)q(t)+\gamma_{3}(t))(1-p(t)\gamma
_{2}(t))^{-1}(\beta_{2}(t)p(t)\varphi(t)+p(t)L_{2}(t)+\nu(t)).
\end{array}
\label{new-eq-112}%
\end{equation}

\begin{theorem}
Assume \eqref{appen-eq-pq} has a solution $(p(\cdot),q(\cdot))$,
$(\varphi(\cdot),\nu(\cdot))\in L_{\mathcal{F}}^{2}(\Omega;C([0,T],\mathbb{R}%
))\times L_{\mathcal{F}}^{2,2}([0,T];\mathbb{R})$, and $(\tilde{X}%
(\cdot),\tilde{Y}(\cdot),\tilde{Z}(\cdot))\in L_{\mathcal{F}}^{2}%
(\Omega;C([0,T],\mathbb{R}))\times L_{\mathcal{F}}^{2}(\Omega
;C([0,T],\mathbb{R}))\times L_{\mathcal{F}}^{2,2}([0,T];\mathbb{R})$, where
$\tilde{X}(\cdot)$ is the solution to
\begin{equation}
\left\{
\begin{array}
[c]{rl}%
d\tilde{X}(t)= & \left\{  \alpha_{1}(t)\tilde{X}(t)+\beta_{1}(t)p(t)\tilde
{X}(t)+\beta_{1}(t)\varphi(t)+L_{1}(t)+\gamma_{1}(t)\left(  1-p(t)\gamma
_{2}(t)\right)  ^{-1}\right. \\
& \cdot\left[  \left(  \alpha_{2}(t)p(t)+\beta_{2}(t)p(t)^{2}+q(t)\right)
\tilde{X}(t)\right.  \left.  \left.  +p(t)\sigma_{y}(t)\varphi(t)+p(t)L_{2}%
(t)+\nu(t)\right]  \right\}  dt\\
& +\left\{  \alpha_{2}(t)\tilde{X}(t)+\beta_{2}(t)p(t)\tilde{X}(t)+\beta
_{2}(t)\varphi(t)+L_{2}(t)+\gamma_{2}(t)\left(  1-p(t)\gamma_{2}(t)\right)
^{-1}\right. \\
& \cdot\left[  \left(  \alpha_{2}(t)p(t)+\beta_{2}(t)p(t)^{2}+q(t)\right)
\tilde{X}(t)\right.  \left.  \left.  +p(t)\sigma_{y}(t)\varphi(t)+p(t)L_{2}%
(t)+\nu(t)\right]  \right\}  dB(t),\\
\tilde{X}(0)= & x_{0},
\end{array}
\right.
\end{equation}
and%
\begin{equation}%
\begin{array}
[c]{rl}%
\tilde{Y}(t)= & p(t)\tilde{X}(t)+\varphi(t),\\
\tilde{Z}(t)= & \left(  1-p(t)\gamma_{2}(t)\right)  ^{-1}\left[  \left(
\alpha_{2}(t)p(t)+\beta_{2}(t)p(t)^{2}+q(t)\right)  \tilde{X}(t)\right. \\
& \left.  +p(t)\beta_{2}(t)\varphi(t)+p(t)L_{2}(t)+\nu(t)\right]  .
\end{array}
\end{equation}
Then $(\tilde{X}(\cdot),\tilde{Y}(\cdot),\tilde{Z}(\cdot))$ solves
\eqref{appen-eq-xyz}.\ Moreover, if
\[
p(t)L_{1}(t)+q(t)L_{2}(t)+L_{3}(t)+(\gamma_{1}(t)p(t)+\gamma_{2}%
(t)q(t)+\gamma_{3}(t))(1-p(t)\gamma_{2}(t))^{-1}p(t)L_{2}(t)=0,
\]
then $(\varphi(\cdot),\nu(\cdot))=(0,0)$ is a solution to \eqref{appen-eq-pq} and

$(\tilde{X}(t),p(t)\tilde{X}(t),\left(  1-p(t)\gamma_{2}(t)\right)
^{-1}\left[  \left(  \alpha_{2}(t)p(t)+\beta_{2}(t)p(t)^{2}+q(t)\right)
\tilde{X}(t)+p(t)L_{2}(t)\right]  )_{t\in\lbrack0,T]}$ solves
\eqref{appen-eq-xyz}. \label{appen-th-linear-fbsde}
\end{theorem}

\begin{proof}
The results follow by applying It\^{o}'s formula.
\end{proof}

Now we study the uniqueness of $(p(\cdot),q(\cdot))$ in (\ref{appen-eq-pq}).
It is important to note that the form of $(p(\cdot),q(\cdot))$ in
(\ref{appen-eq-pq}) does not depend on $x_{0}$, $L_{1}$, $L_{2}$, $L_{3}$. So
we set $\ x_{0}=1$, $L_{1}=L_{2}=L_{3}=0$ in the followings. In this case
$(\varphi(\cdot),\nu(\cdot))=(0,0)$ as in the above theorem.

\begin{theorem}
\label{unique-pq} Assume (\ref{appen-eq-xyz}) has a unique solution
$(X(\cdot),Y(\cdot),Z(\cdot))\in L_{\mathcal{F}}^{2}(\Omega;C([0,T],\mathbb{R}%
))\times L_{\mathcal{F}}^{2}(\Omega;C([0,T],\mathbb{R}))\times L_{\mathcal{F}%
}^{2,2}([0,T];\mathbb{R})$.
\end{theorem}

\begin{description}
\item[(i)] If $\gamma_{2}(\cdot)$ is small enough and $(p_{i}(\cdot
),q_{i}(\cdot))\in L_{\mathcal{F}}^{\infty}(0,T;\mathbb{R})\times
L_{\mathcal{F}}^{2,4}([0,T];\mathbb{R})$, $i=1,2$, are two solutions to
(\ref{appen-eq-pq}), then $(p_{1}(\cdot),q_{1}(\cdot))=(p_{2}(\cdot
),q_{2}(\cdot));$

\item[(ii)] If $(p_{i}(\cdot),q_{i}(\cdot))\in L_{\mathcal{F}}^{\infty
}(0,T;\mathbb{R})\times L_{\mathcal{F}}^{\infty}(0,T;\mathbb{R})$, $i=1,2$,
are two solutions to (\ref{appen-eq-pq}), then $(p_{1}(\cdot),q_{1}%
(\cdot))=(p_{2}(\cdot),q_{2}(\cdot))$.
\end{description}

\begin{proof}
We only prove (i), (ii) is similar. Consider the following SDEs:%
\begin{equation}
\left\{
\begin{array}
[c]{ll}%
d\tilde{X}_{i}(t)= & \left[  \alpha_{1}(t)+\beta_{1}(t)p_{i}(t)+\gamma
_{1}(t)K_{i,1}(t)\right]  \tilde{X}_{i}(t)dt\\
& +\left[  \alpha_{2}(t)+\beta_{2}(t)p_{i}(t)+\gamma_{2}(t)K_{i,1}(t)\right]
\tilde{X}_{i}(t)dB(t),\\
\tilde{X}_{i}(0)= & 1,\text{ }i=1,2.
\end{array}
\right.
\end{equation}
Then $\tilde{X}_{i}(\cdot)$ has a explicit form%
\[
\tilde{X}_{i}(t)=\exp\left\{  \int_{0}^{t}\left(  N_{i,1}(s)-\frac{1}%
{2}(N_{i,2}(s))^{2}\right)  ds+\int_{0}^{t}N_{i,2}(s)dB(s)\right\}  ,
\]
where%
\[%
\begin{array}
[c]{rl}%
K_{i,1}(s) & =(1-p_{i}(s)\gamma_{2}(s))^{-1}\left[  \alpha_{2}(s)p_{i}%
(s)+\beta_{2}(s)p_{i}^{2}(s)+q_{i}(s)\right]  ,\\
N_{i,1}(s) & =\alpha_{1}(s)+\beta_{1}(s)p_{i}(s)+\gamma_{1}(s)K_{i,1}(s),\\
N_{i,2}(s) & =\alpha_{2}(s)+\beta_{2}(s)p_{i}(s)+\gamma_{2}(s)K_{i,1}(s).
\end{array}
\]
By Theorem \ref{q-exp-th}, it is easy to check that when $\gamma_{2}(\cdot)$
is small enough,
\[
\mathbb{E}\left[  \underset{0\leq t\leq T}{\sup}\left\vert \tilde{X}%
_{i}(t)\right\vert ^{4}\right]  <\infty.
\]
Thus we have $\left(  K_{i,1}(t)\tilde{X}_{i}(t)\right)  _{t\in\lbrack0,T]}\in
L_{\mathcal{F}}^{2,2}([0,T];\mathbb{R})$. Since (\ref{appen-eq-xyz}) has a
unique solution, by Theorem \ref{appen-th-linear-fbsde}, we get for
$t\in\lbrack0,T]$, \
\[
(\tilde{X}_{1}(t),p_{1}(t)\tilde{X}_{1}(t),K_{1,1}(t)\tilde{X}_{1}%
(t))=(\tilde{X}_{2}(t),p_{2}(t)\tilde{X}_{2}(t),K_{2,1}(t)\tilde{X}_{2}(t)).
\]
Note that $\tilde{X}_{1}(\cdot)>0$, then $(p_{1}(\cdot),q_{1}(\cdot
))=(p_{2}(\cdot),q_{2}(\cdot))$.
\end{proof}

\section*{Acknowledgement}

We are highly grateful to Dr. Falei Wang for his helpful suggestions and comments.

\end{document}